\documentclass[12pt]{amsart}
\usepackage[dvipsnames]{xcolor}
\usepackage{tikz}
\usepackage{tikz-3dplot}
\usetikzlibrary{decorations.pathreplacing,calligraphy}
\usepackage{graphicx}
\usepackage{amssymb, amsthm, amsmath}
\usepackage{etoolbox}
\apptocmd{\sloppy}{\hbadness 10000\relax}{}{}
\apptocmd{\sloppy}{\vbadness 10000\relax}{}{}
\usepackage[justification=justified]{caption}
\usepackage{hyperref}
\usepackage[letterpaper,margin=1.1in]{geometry}

\usepackage{booktabs}       
\usepackage{array}          
\usepackage{colortbl}       

\definecolor{rowgray}{gray}{0.93}
\definecolor{sepcol}{gray}{0.6}
\definecolor{figorange}{RGB}{235,175,127}
\definecolor{figcyan}{RGB}{140,251,253}
\definecolor{figgrid}{RGB}{150,130,115}
\definecolor{figbord}{RGB}{200,200,200}

\definecolor{PTblue}{HTML}{0077BB}
\definecolor{PTteal}{HTML}{009988}
\definecolor{PTmagenta}{HTML}{EE3377}

\newtheorem{oldtheorem}{Theorem}

\newtheorem{theorem}{Theorem}[section]
\newtheorem{proposition}[theorem]{Proposition}
\newtheorem{lemma}[theorem]{Lemma}
\newtheorem{corollary}[theorem]{Corollary}

\theoremstyle{definition}
\newtheorem{definition}[theorem]{Definition}
\newtheorem{example}[theorem]{Example}
\newtheorem{principle}[theorem]{Principle}
\newtheorem{challenge}{Challenge}

\theoremstyle{remark}

\let\inf\relax \DeclareMathOperator*\inf{\vphantom{p}inf}

\newcommand{\side}{\mathop\mathrm{side}\nolimits}
\newcommand{\vol}{\mathop\mathrm{vol}\nolimits}

\newcommand{\Haus}{\mathcal{H}}
\newcommand{\Pack}{\mathcal{P}}
\newcommand{\Net}[1]{\mathcal{M}^{\Delta,#1}}
\newcommand{\Cube}{\mathcal{Q}}

\newcommand{\RR}{\mathbb{R}}

\numberwithin{equation}{section}
\numberwithin{figure}{section}
\numberwithin{table}{section}

\def\res{\hbox{ {\vrule height .22cm}{\leaders\hrule\hskip.2cm} } }

\newcommand{\interior}{\mathop\mathrm{int}\nolimits}
\newcommand{\diam}{\mathop\mathrm{diam}\nolimits}
\newcommand{\Child}{\mathop\mathsf{Child}\nolimits}
\newcommand{\gap}{\mathop\mathrm{gap}\nolimits}
\newcommand{\dist}{\mathop\mathrm{dist}\nolimits}
\newcommand{\ex}{\mathop\mathrm{ex}\nolimits}

\newcommand{\Site}{\mathop\mathsf{Site}\nolimits}

\title[Dimension of harmonic measure]{One-ring Bourgain alternative and a 5-decimal bound on the dimension of harmonic measure}
\date{October 6, 2026}
\author{Michael Albert}
\author{Matthew Badger}
\author{Alyssa Genschaw}
\thanks{M.~Albert and M.~Badger were partially supported by NSF DMS grant 2154047.}
\address{Analysis and Partial Differential Equations Unit,
         Okinawa Institute of Science and Technology,
         1919-1 Tancha, Onna-son, Okinawa 904-0495, Japan}
\email{michael.albert@oist.jp}
\address{Department of Mathematics\\ University of Connecticut\\ Storrs, CT 06269-1009, USA}
\email{matthew.badger@uconn.edu}
\address{Mathematics Department\\ Milwaukee School of Engineering\\ Milwaukee, WI 53202, USA}
\email{genschaw@msoe.edu}
\begin{document}

\begin{abstract}Improving theorems of Bourgain and Badger--Genschaw, we prove that the Hausdorff dimension of harmonic measure on any domain in $\RR^3$, $\RR^4$, and $\RR^n$ for $n\geq3$ is at most $3-10^{-5}$, $4-10^{-7}$, and $n-(n-1)^{-(2n+6.12\sqrt{n})}$, respectively. The previous bounds in dimensions $3$ and $4$ were $3-10^{-15}$ and $4-10^{-26}$, while Bishop conjectures that the sharp bound is $n-(n-1)^{-1}$. The main difficulty in optimizing Bourgain's method is estimating non-additive $s$-dimensional Hausdorff-type contents when $s<n$ is close to $n$. Previous arguments used net contents associated to a fixed $m$-adic lattice, with $m$ large depending on $n$. We redevelop the dimension theory using pairwise disjoint cubes in general position and obtain a new estimate showing that one thin ring without large holes surrounding a cube makes it unlikely for Brownian motion to reach the cube before exiting the domain. A key feature of the construction is that the holes are round rather than rectangular. We also propose ten challenges on the dimension of harmonic measure that may provide accessible targets en route to Bishop's conjecture.

\end{abstract}

\maketitle

\setcounter{tocdepth}{1}
\tableofcontents

\section{Introduction}

\subsection{Background}

One way to understand harmonic measure is as the exit distribution of Brownian motion from a domain. Let $\Omega\subset\RR^n$ be a bounded, connected, open set and let $x\in\Omega$. Standard Brownian motion $B:[0,\infty)\rightarrow\RR^n$ started at $B(0)=x$ eventually exits $\Omega$ at a random point $B(t_{\rm exit})\in\partial\Omega$. The probability that the exit point belongs to a set $E\subset\partial\Omega$ is called the \emph{harmonic measure} of $E$ in $\Omega$ with \emph{pole} at $x$ and is denoted by $\omega^x_\Omega(E)$. Thus, with $\Omega$ and $x$ fixed, $E\mapsto\omega^x_\Omega(E)$ is a Borel probability measure on $\partial\Omega$. Kakutani \cite{kakutani} proved the complementary analytic fact that, with $\Omega$ and $E$ fixed, the function $x\mapsto\omega^x_\Omega(E)$ is harmonic in $\Omega$, and either $\omega^x_\Omega(E)=0$ for every $x\in\Omega$ or $\sup_{x\in\Omega}\omega^x_\Omega(E)=1$. In particular, the null sets of $\omega^x_\Omega$ do not depend on the pole, and we commonly write that a property holds $\omega_\Omega$-a.e.~or $\omega$-a.e.~without specifying $x$.

Harmonic measure also represents solutions of the Dirichlet problem. For every real-valued $f\in C(\partial\Omega)$, the function
\begin{equation}\label{e:harmonic-extension}
u(x):=\int_{\partial\Omega}f(y)\,d\omega^x_\Omega(y),\quad\text{for all }x\in\Omega,
\end{equation}
is harmonic in $\Omega$ and satisfies
\begin{equation}\label{e:boundary-behavior}
\lim_{x\rightarrow y}u(x)=f(y)\quad\text{at $\omega_\Omega$-a.e.~$y\in\partial\Omega$}.
\end{equation}
More precisely, \eqref{e:boundary-behavior} holds outside a set of logarithmic capacity zero when $n=2$ and Newtonian capacity zero when $n\geq3$, hence outside a set of Hausdorff dimension at most $n-2$. Equations \eqref{e:harmonic-extension} and \eqref{e:boundary-behavior}, together with the probabilistic interpretation above, provide most of the background on harmonic measure needed in this paper. For more extensive introductions, see Bishop's review \cite{bishop-review} of the book of Garnett and Marshall \cite{GM}, as well as Kenig \cite{kenig-cbms}, Heinonen \emph{et al.} \cite{HKM-v2}, Helms \cite{helms}, M{\"o}rters and Peres \cite{brownian}, and the forthcoming book of Prats and Tolsa \cite{prats-tolsa}.

\begin{figure}[t]
\noindent
\begin{minipage}{.3\textwidth}%
\vspace{0pt}%
\includegraphics[width=\linewidth]{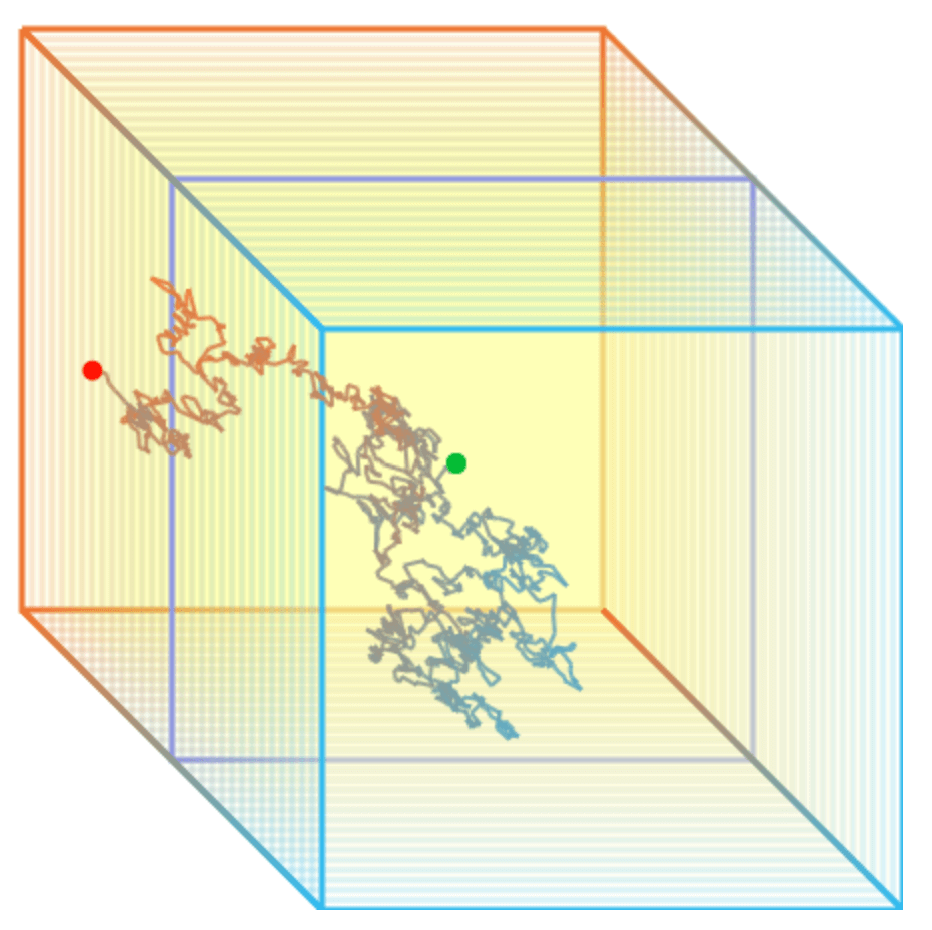}
\end{minipage}%
\begin{minipage}{.7\textwidth}%
\vspace{0pt}%
\caption{Brownian motion simulated in a cube, starting at the green dot and exiting at the red dot. With its pole at the center, each face of the cube has harmonic measure $1/6$, while each edge has harmonic measure $0$.}
\label{fig:bm-demo}
\end{minipage}
\end{figure}

We are interested in the geometric size of harmonic measure. Recall that the \emph{upper Hausdorff dimension} of a measure is the least Hausdorff dimension of a set of full measure, while the \emph{lower Hausdorff dimension} is the infimum of the Hausdorff dimensions of sets of positive measure. The Hausdorff dimensions of $\omega_\Omega=\omega_\Omega^x$ depend on the domain but not on the choice of pole. On any bounded domain in $\RR^n$, the lower Hausdorff dimension of harmonic measure is at least $n-2$, see e.g.~\cite[Appendix A]{Badger-thesis},  \cite{wu-singularity}. On a smooth or Lipschitz domain,  harmonic measure is mutually absolutely continuous with respect to surface measure on $\partial\Omega$ and thus has Hausdorff dimension equal to $n-1$ \cite{Dahlberg}; see Figure \ref{fig:bm-demo}. By contrast, on general domains, whose boundaries can be rough and have Hausdorff dimension greater than $n-1$, the dimension of harmonic measure may be either larger or smaller than $n-1$, yet it can never reach the ambient dimension $n$. The modern higher-dimensional theory begins with the following two results.

\begin{oldtheorem}[{Wolff \cite{Wolff}\footnote{Although published in 1995, preliminary versions of Wolff's paper circulated in the 1980s.}}]\label{t:wolff}
For all $n\geq3$, there exist simply connected NTA\footnote{The non-tangentially accessible domains of Jerison and Kenig are quantitatively open and quantitatively connected; they include Lipschitz domains and many fractal domains \cite{jerisonandkenig}, \cite{kenigtoroduke}.} domains $\Omega\subset\RR^n$ such that the lower and upper Hausdorff dimensions of harmonic measure agree and are $>n-1$. There also exist simply connected NTA domains $\Omega\subset\RR^n$ such that the lower and upper Hausdorff dimensions of harmonic measure agree and are $<n-1$.
\end{oldtheorem}

\begin{oldtheorem}[{Bourgain \cite{Bourgain}}]\label{t:bourgain}
For all $n\geq3$, there exists $b_n>0$ such that the upper Hausdorff dimension of harmonic measure on any domain in $\RR^n$ is at most $n-b_n$.
\end{oldtheorem}

Theorem \ref{t:bourgain} implies that whenever $\dim_H\partial\Omega>n-b_n$, harmonic measure is carried by a set $E\subset\partial\Omega$ with $\dim_H E<\dim_H\partial\Omega$. In probabilistic terms, Brownian motion almost surely exits $\Omega$ through a set that is geometrically smaller than the full boundary. This phenomenon, called \emph{dimension drop}, was observed by Kaufman and Wu \cite{kw-snowflake} for the domain bounded by a von Koch snowflake and by Carleson \cite{carleson-cantor} for Cantor-type planar domains, and its study has grown substantially in recent years. For classical harmonic measure, see \cite{Volberg-dd}, \cite{cantor-dim}, \cite{Batakis-Zdunik}, \cite{azzam-dimension-drop}, \cite{DJJ-no-drop}, \cite{Tolsa-dd}, \cite{cai-tolsa}; for $p$-harmonic measures with $p\neq2$, see \cite{LNV-p-drop}, \cite{ALV-p-drop}; for elliptic measures, see \cite{Sweezy-planar}, \cite{Sweezy-space}, \cite{DM-good-elliptic}, \cite{Perstneva-good-elliptic}, \cite{gmpt-dd}, \cite{gm-dd}; for Robin harmonic measures, see \cite{robin-dimension}; and for caloric and parabolic measures, see \cite{Sweezy-ac}, \cite{dim-caloric}. For further results related to Theorem \ref{t:wolff}, see \cite{LVV}, \cite{kenigpreisstoro}, \cite{Badger-nullsets}, \cite{azzam-corkscrews}.

Let $d_n$ denote the supremum of the upper Hausdorff dimensions of harmonic measures over all domains in $\RR^n$. The optimal constant in Theorem \ref{t:bourgain},
\begin{equation}\label{bourgain-constant-def}
b_n:=n-d_n,
\end{equation}
is called \emph{Bourgain's constant} for harmonic measure in $\RR^n$. Thus, the theorems of Bourgain and Wolff imply $0<b_n<1$ for every $n\geq3$. Moreover, if $\Omega\subset\RR^n$, then the cylinder $\Omega\times\RR\subset\RR^{n+1}$ satisfies $\overline{\dim}_H\,\omega_{\Omega\times\RR}=\overline{\dim}_H\,\omega_\Omega+1$. Consequently,
\begin{equation}\label{b-decreasing}
d_n+1\leq d_{n+1}
\qquad\text{and}\qquad
b_n\geq b_{n+1}
\quad\text{for all $n\geq2$}.
\end{equation}
The planar theory provides an important baseline. Using conformal mappings and subharmonicity of \(\log|f|\) for holomorphic functions \(f\), Jones and Wolff \cite{Jones-Wolff} proved that \(b_2=1\). Makarov \cite{Makarov} proved that on every simply connected planar domain the lower and upper Hausdorff dimensions of harmonic measure are exactly \(1\). These planar results helped motivate the theorems of Bourgain and Wolff in higher dimensions.

For $n\geq3$, the exact value of $b_n$ is unknown. The case $n=3$ is especially compelling, both as the first unknown dimension and because of its connection with diffusion in physical space \cite{filoche-passivation}. Bishop \cite{Bishop-questions} conjectures that
$b_3=1/2$, and more generally that $b_n=1/(n-1)$ for every $n\geq3$. More than thirty years later, even the order of magnitude of $b_n$ as $n\rightarrow\infty$ remains unknown. A numerical experiment of Grebenkov \emph{et al.} \cite{grebenkov-numerical-experiment} on the cubical von Koch snowball in Figure \ref{fig:physics} suggests that $b_3\leq0.995$; the first challenge in \S2 calls for independently reproducing or refuting this estimate and for numerical investigations of harmonic measure on other explicit domains in $\RR^3$.
\begin{figure}[t]
\begin{center}
\includegraphics[width=0.9\textwidth]{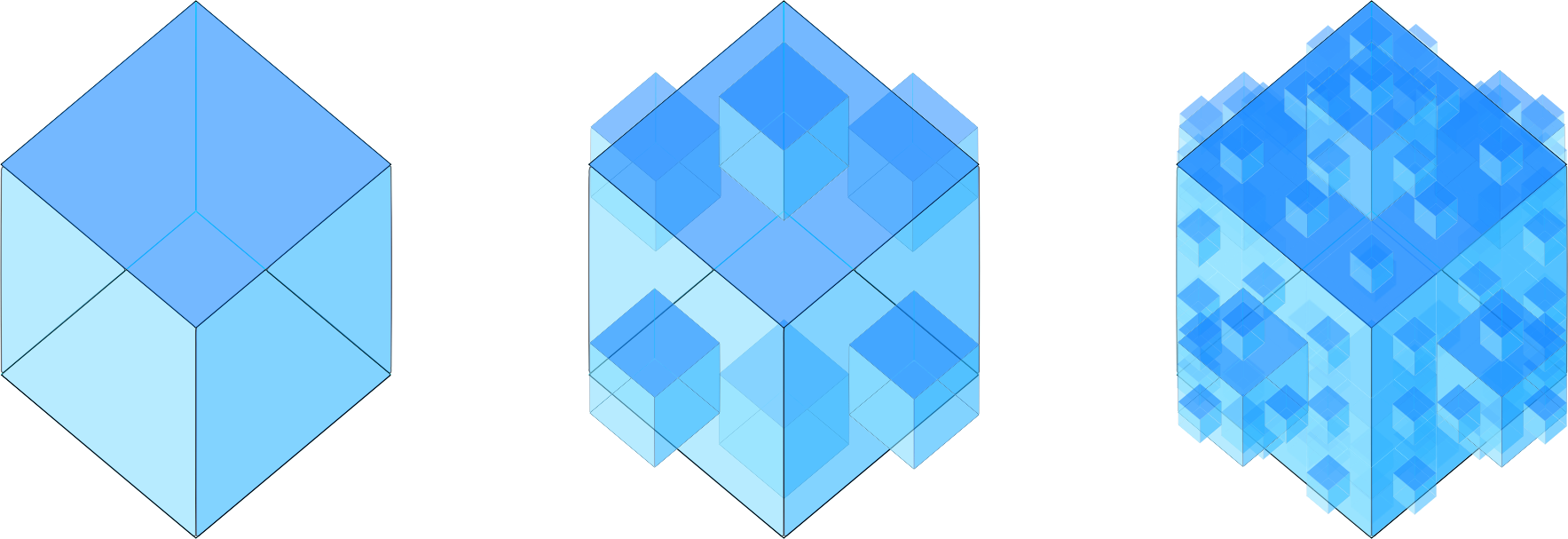}
\end{center}
\caption{Levels $0$, $1$, and $2$ of the cubical von Koch snowball domain $G\subset\RR^3$ examined in \cite{grebenkov-numerical-experiment}. Each square of side length $3^{-l}$ in the level $l$ boundary carries an outward cube of side length $3^{-(l+1)}$, and $G$ is the interior of the union of the finite-level approximations. Here $\dim_H\partial G=\log_3 13$, while the numerics of Grebenkov \emph{et al.} give $\dim_H\omega_G=2.007\pm0.002$.}
\label{fig:physics}
\end{figure}
More recently, Badger and Genschaw revisited Bourgain's proof and obtained the first explicit lower bounds on $b_n$ in dimensions $n\geq 3$.

Let $\log$ denote the natural logarithm.

\begin{oldtheorem}[{Badger--Genschaw \cite{15-9s}}]\label{bg:old}
We have $b_3\geq 1\times10^{-15}$, $b_4\geq2\times10^{-26}$, and
$b_n\geq0.6n^{-2n(n-1)}/\log n$ for all $n\gg3$. In particular, the upper Hausdorff dimension of harmonic measure on any domain in $\RR^3$ is at most $2.99999\,99999\,99999$.
\end{oldtheorem}

For comparison with Bishop's conjecture and the new estimates below, Theorem \ref{bg:old} implies that there exists some $n_0$ (uncomputed) such that
\[
b_n\geq\left(\frac{1}{n-1}\right)^{2n^2}
\quad\text{for all $n\geq n_0$}.
\]

\subsection{New results and proof strategy}\label{ss:new-stuff}

The main result of this paper improves both the explicit low-dimensional bounds in Theorem \ref{bg:old} and the asymptotic exponent as $n\rightarrow\infty$. We also obtain the first fully effective bound in arbitrary dimensions.

\begin{theorem}\label{t:big-dim}
For all $n\geq 2$,
\begin{equation}\label{e:main-bound}
b_n\geq\left(\frac{1}{n-1}\right)^{2n+6.12\sqrt n}.
\end{equation}
\end{theorem}

\begin{corollary}\label{c:six-decimals}
We have $b_3\geq 2^{-(6+6.12\sqrt3)}=1.0065\ldots\times10^{-5}$. Thus, the upper Hausdorff dimension of harmonic measure on any domain in $\RR^3$ is less than $2.99999$.
\end{corollary}

The proof follows Bourgain's original outline, which is based on a local dichotomy. Roughly speaking, at each scale either the boundary has a sufficiently large hole, yielding an efficient cover, or the absence of such holes creates a barrier that makes it difficult for Brownian motion to pass deeper into the domain. Previous implementations of the second alternative required many concentric barrier rings. Rather than optimizing a single step, we rework every stage of the argument—replacing coarse estimates by sharp computations, allowing the covering geometry to be chosen adaptively, and introducing new barrier estimates—to eliminate accumulated quantitative losses and arrive at what we believe is the essentially optimal form of Bourgain's alternative.

On the geometric side, Section \ref{sec:cubical-content} develops quantitative principles for $s$-dimensional Hausdorff contents when $s<n$, where non-additivity makes the efficiency of a cover sensitive to the scales of its cubes: a hole in a cubical region can favor a mixed-scale cover when $s$ is close to $n$, while a single cube may be more efficient as $s$ decreases. A general treatment of such calculations does not appear to exist elsewhere. In \S\ref{sec:gmt}, we replace the fixed $m$-adic lattice used previously by a Bourgain-type dimension lemma that permits pairwise disjoint covering cubes in general position and allows the entropy order $\alpha\in(0,1)$ to be chosen to sharpen the estimate. In \S\ref{s:polar}, we develop a method of \emph{approximate polar integration} that transfers radial integral estimates against Frostman measures to estimates against Lebesgue measure with explicit control of the error. Our method yields a nearly sharp Bourgain estimate of the local harmonic measure in round regions.

The argument culminates in the \emph{one-ring alternative} of \S\ref{sec:parameters}. At a given scale, choose a thin ring. If the boundary has a sufficiently large hole in the ring, then the hole yields an efficient mixed-scale cover and relative Hausdorff content below \(1\). If no such hole exists, the local estimate on harmonic measure makes the ring a barrier to Brownian motion. Then a new computation converts this local barrier into quantitative control of harmonic measure crossing a single ring, replacing the previous multi-ring construction. A further refinement accounts for arbitrarily many returns of Brownian motion to the barrier and sharpens the constant. The resulting reduction in geometric cost allows the Hausdorff-content gain in the large-hole case to remain effective for exponents \(s\) farther below \(n\), producing the improved bounds on \(b_n\). These gains are cumulative: reaching the five-decimal estimate in dimension $3$ depends on the refinements working together. In \S\ref{s:punchline}, an explicit family of parameters yields Theorem \ref{t:big-dim} for $n\geq6$, while separately optimized explicit choices of parameters gives the stated estimate in dimensions $3$, $4$, and $5$.

\section{Ten challenges on dimension of harmonic measure}

Given the general lack of progress on computing Bourgain's constants $b_n$, when $n\geq 3$, we propose ten intermediate challenges in order to spur further research on dimension of harmonic measure. The challenges range from concrete problems that invite direct computation to broader, more open-ended theoretical questions. Some ask for explicit examples or quantitative estimates, while others concern qualitative phenomena and the nature of dimension drop. We start with four challenges in $\RR^3$.

\begin{challenge}[Numerical Experiment Challenge]\label{challenge:numerical}
Using either the same or different methodology as in \cite{grebenkov-numerical-experiment}, perform an experiment in order to numerically estimate the dimension of harmonic measure on the cubical von Koch snowball domain in Figure \ref{fig:physics}. Estimate the dimension of harmonic measure on other domains in $\RR^3$ such as domains whose boundaries are cubical Cantor sets.
\end{challenge}

\begin{challenge}[Explicit Upper Bound Challenge]\label{challenge:explicit-b3}
Prove that $b_3\leq w<1$ for some explicit number $w$. If possible, rigorously verify the numerical bound discussed above.
\end{challenge}

If $b_3=0.5$, as conjectured, then it should be possible to approximate this value from either side. However, the current bounds $10^{-5}\leq b_3\stackrel{?}{\leq} 0.995$ are frustratingly far from $0.5$. In the authors' view, a resolution of the following challenge would represent a significant accomplishment on par with determining the value of $b_3$.

\begin{challenge}[One Decimal Challenge]\label{challenge:one-decimal}
For trivial reasons, $b_3\geq 0.1$ or $b_3\leq 0.9$, but we don't currently know the truth of either bound. Prove one (or both).
\end{challenge}

Note that the previous challenges could each be resolved by estimates on the dimension of harmonic measure alone, without ever computing the dimension of an actual domain. A resolution of the following distinct challenge would be very welcome.

\begin{challenge}[Exact Dimension Challenge]\label{challenge:exact-dimension}
Construct a domain $\Omega\subset\RR^3$ such that $\overline{\dim}_H\,\omega_\Omega>2$ and compute $\overline{\dim}_H\,\omega_\Omega$ exactly.
\end{challenge}

Let us turn to high-dimensional asymptotics. We know that $b_2=1>b_3$ and $b_n\geq b_{n+1}$ for all $n\geq 3$; see \eqref{b-decreasing}. We do not currently know whether or not $b_n>b_{n+1}$ for all $n\geq 3$. To aid discussion, for each $n\geq 2$, define the \emph{Wolff constant for harmonic measure} by
\begin{equation}
w_n:=\sup_{\Omega\subset\RR^n} w_\Omega,
\qquad
w_\Omega:=\overline{\dim}_H\,\omega_\Omega-(n-1).
\end{equation}
Of course, $w_n+b_n=1$, but the two constants have different interpretations: $w_n$ measures how much larger the dimension of harmonic measure can be than the classical value $n-1$, while $b_n$ measures how much smaller it must be than the ambient dimension $n$. A folklore problem is to prove that $\lim_{n\rightarrow\infty} b_n = 0$, or equivalently, $\lim_{n\rightarrow\infty} w_n = 1$, but remains intractable. The following challenge may serve as a more accessible first step.

\begin{challenge}[Growing Excess Challenge]\label{challenge:growing-excess}
Construct a family of domains $\Omega^n\subset\RR^n$, for $n\geq3$, arising from a single geometric construction, such that
\begin{equation}\label{w-ascending}
0<w_{\Omega^3}<w_{\Omega^4}<w_{\Omega^5}<\cdots.
\end{equation}
\end{challenge}

What constitutes a solution to Challenge \ref{challenge:growing-excess} is open to interpretation. A solution should exclude trivial increases obtained by taking products with Euclidean factors. If you start with domains $\Omega_1,\Omega_2\subset\RR^3$ that satisfy $2<\overline{\dim}_H\,\omega_{\Omega_1}<\overline{\dim}_H\,\omega_{\Omega_2}$ (e.g.~built using Wolff's construction), then $\Omega_2\times\RR\subset\RR^4$ has $w_{\Omega_2\times\RR}=w_{\Omega_2}>w_{\Omega_1}$, but the pair $\Omega^3:=\Omega_1$ and $\Omega^4:=\Omega_2\times\RR$ does not solve the challenge.
The goal is to find a procedure that takes a given domain in $\RR^n$ and produces a domain in $\RR^{n+1}$ that has larger Wolff excess.
Challenge \ref{challenge:growing-excess} is difficult, because it is not clear which domains should be used. The obvious extension of the cubical von Koch snowball to higher dimensions does not solve the challenge. To wit, if $G^n\subset\RR^n$ is obtained from $[0,1]^n$ by repeatedly attaching smaller cubes to the middle-third of each face, then
$\overline{\dim}_H\,\omega_{G^n}
\leq
\dim_H\partial G^n
=
\log_3(3^{n-1}+2n-1).$ Since the latter numbers are asymptotic to $n-1$, we have $\lim_{n\rightarrow\infty} w_{G^n}=0$. In the authors' opinion, this is also indirect evidence that $\overline{\dim}_H \omega_{G^3}$ is likely to be much smaller than $w_3$.

An outcome of the present investigation, which culminated in Theorem \ref{t:big-dim}, is that an exponential lower bound on Bourgain's constant probably cannot be established with Bourgain's alternative alone. Thus, despite being facially weaker than Bishop's conjecture, a resolution of the following challenge would likely represent significant progress towards understanding the dimension of harmonic measure.

\begin{challenge}[Exponential Bound Challenge]\label{challenge:exponential}
Prove $b_n\geq 10^{-Cn}$ for all $n\geq n_0$ for some constants $C$ and $n_0$, or prove \emph{any} sub-exponential lower bound on $b_n$ when $n$ is large.
\end{challenge}

Packing measures and packing dimension, introduced by Tricot \cite{Tricot-packing}, provide another fundamental notion of size in geometric measure theory. Packing dimension may be viewed as the countable stabilization of a better-known dimension often called upper Minkowski or box-counting dimension; see e.g.~\cite[Chapter 5]{Mattila}. For each $n\geq2$, define
\begin{equation}
p_n:=n-\sup_{\Omega\subset\RR^n}\overline{\dim}_P\,\omega_\Omega,
\end{equation}
which is the packing-dimension analogue of Bourgain's constant. Since packing dimension dominates Hausdorff dimension, $0\leq p_n\leq b_n$.

\begin{challenge}[Packing Dimension Challenge]\label{challenge:packing}
Prove or disprove that $p_n>0$ for every $n\geq2$. If $p_n>0$, is $p_n=b_n$? In particular, is $p_2=1$?
\end{challenge}

At present, it is not clear to us whether one should expect a universal packing dimension drop. The evidence is weak and somewhat mixed. On the pro side, Housworth already proved the packing dimension analogue of Makarov's theorem, i.e.~harmonic measure on every simply connected planar domain lives on a set of packing dimension 1 \cite{Housworth}. On the other hand, harmonic measure is already known to behave differently with respect to Hausdorff and packing measures. Choi's theorem on the lower $1$-density at twist points \cite{Choi-lower} implies that on many planar domains, including the von Koch snowflake, every set of $\sigma$-finite packing measure $\Pack^1$ has harmonic measure zero. This contrasts with a theorem of Wolff \cite{Wolff-sigma} that says for every bounded domain in $\RR^2$ there exists a set of $\sigma$-finite Hausdorff measure $\Haus^1$ and full harmonic measure. Additional weaker evidence that $p_n$ might vanish is that the proof of the Jones-Wolff theorem ($b_2=1$) relies on a notion of generalized Hausdorff capacity that does not mesh well with packing dimension. Bourgain's method for showing $b_n>0$ when $n\geq 3$ also cannot be easily modified to give an affirmative answer to Challenge \ref{challenge:packing}, because of its heavy reliance on Frostman's lemma for Hausdorff dimension.

Even if the Bourgain constants were to be computed, there are other natural extremal questions about dimension drop. Define the \emph{maximal dimension drop constant}
\[
m_n=\sup_{\Omega\subset\RR^n}
\left(\dim_H\partial\Omega-\overline{\dim}_H\,\omega_\Omega\right).
\]
There are domains with $\dim_H\partial\Omega=n$, so Bourgain's theorem implies $m_n\geq b_n$.

\begin{challenge}[Maximal Dimension Drop Challenge]\label{challenge:maximal-drop}
For each $n\geq 2$, is $m_n=b_n$ or is $m_n>b_n$? If $m_2>1$, compute $m_2$ or obtain nontrivial bounds for it.
\end{challenge}

A concrete setting for investigating where dimension drop first appears is the class of \emph{Denjoy domains}, i.e.~domains $\Omega\subset\RR^2$ whose complements are contained in a line. David, Jeznach, and Julia \cite{DJJ-no-drop} constructed $s$-Ahlfors regular Cantor sets $E\subset\RR$ with $s$ small for which harmonic measure on $\RR^2\setminus E$ is equivalent to $\Haus^s|_E$, and hence for which harmonic measure does not have dimension drop. In the opposite direction, Tolsa \cite{Tolsa-dd} proved that $\overline{\dim}_H\,\omega_{\RR^2\setminus E}<s$ whenever $E\subset\RR$ is $s$-Ahlfors regular and $1/2\leq s<1$, and asked for the sharp transition between these two phenomena.

\begin{challenge}[Denjoy Transition Challenge]\label{challenge:denjoy-transition}
Determine the least $s_0\in(0,1)$ such that, for every $s\in(s_0,1)$ and every $s$-Ahlfors regular compact set $E\subset\RR$, the harmonic measure of $\RR^2\setminus E$ satisfies $\overline{\dim}_H\,\omega_{\RR^2\setminus E}<s$. More generally, determine whether there is an analogous transition for $s$-Ahlfors regular compact sets $E\subset\RR^2$ that are not confined to a line.
\end{challenge}

We end our list with a challenge concerning elliptic measure, the analogue of harmonic measure associated to divergence-form uniformly elliptic operators $L=-\operatorname{div}(A\nabla\cdot)$. It is known that fine properties of elliptic measures can differ from classical harmonic measure, e.g.~elliptic measure on a smooth domain can be singular with respect to surface measure \cite{MMS80}, \cite{CFK81}, \cite{Wu-porous}. For $1\leq \Lambda<\infty$, define the \emph{elliptic Bourgain constant}
\begin{equation}
b_n(\Lambda)
:=
n-\sup_{\Omega,A}\overline{\dim}_H\,\omega_{\Omega}^A,
\end{equation}
where the supremum is taken over domains $\Omega\subset\RR^n$ and symmetric, measurable, uniformly elliptic matrix fields $A$ satisfying $\Lambda^{-1}|\xi|^2\leq\langle A(x)\xi,\xi\rangle$ and $\langle A(x)\xi ,\eta\rangle \leq\Lambda|\xi||\eta|$ for a.e.~$x$ and every $\xi,\eta\in\RR^n$. With this normalization, $b_n(1)=b_n$. Examples of Sweezy \cite{Sweezy-planar,Sweezy-space} show that $\lim_{\Lambda\to\infty}b_n(\Lambda)=0$ for every $n\geq2$. In the plane, Guill\'en-Mola \cite{gm-dd} proved the bound $b_2(\Lambda)\geq 2/(\Lambda+1)$. For $n\geq3$, it seems likely that positivity of $b_n(\Lambda)$ for each fixed $\Lambda<\infty$ can be obtained by a modification of Bourgain's method, but we are unaware of such a result in the literature.

\begin{challenge}[Elliptic Bourgain Challenge]\label{challenge:elliptic-bourgain}
Prove that $b_n(\Lambda)>0$ for every $n\geq3$ and every $1\leq\Lambda<\infty$, and determine the rate at which $b_n(\Lambda)\to0$ as $\Lambda\to\infty$ for each $n\geq2$. In particular, obtain quantitative upper and lower bounds for $b_n(\Lambda)$ in terms of $\Lambda$.
\end{challenge}

It would also be interesting to study analogous questions under additional regularity assumptions on the coefficients $A$ or for other classes of operators $L$.

\section{Principles of estimating cubical Hausdorff contents}\label{sec:cubical-content}

The purpose of this section is to understand quantitatively when empty space inside a cube makes it advantageous to replace the cube by a cover using smaller cubes. When the dimension parameter is below the ambient dimension, subdivision carries an intrinsic cost, so a hole must be large enough, and the covering cubes chosen efficiently enough, to compensate for that cost. We compare uniform and mixed-scale covers and develop mechanisms for deciding when a hole produces a genuine gain in $s$-dimensional cubical content. These principles will later be applied to the round holes that arise in Bourgain's alternative. As an indication of the subtleties involved, even the exact cubical Hausdorff content of a Euclidean ball appears to be unknown.

Let $n\geq 1$ be an integer. For the purpose of this paper, a \emph{cube} $Q$ in $\RR^n$ of \emph{side length} $0<\ell<\infty$ is any set of the form $I_1\times\cdots\times I_n$, where each set $I_i$ is an interval of the form $(a,a+\ell)$, $(a,a+\ell]$, $[a,a+\ell)$, or $[a,a+\ell]$. A singleton $\{a\}$ is not a cube. If $Q$ is a cube, then we may denote its side length by $\side Q$ and its volume by $\vol Q$, i.e.~$\vol Q=(\side Q)^n$. Two cubes are said to have the \emph{same type} if their $i$-th component intervals are of the same type for each $1\leq i\leq n$. Thus, there are $4^n$ types of cubes in $\RR^n$.

\begin{definition}\label{def:cubical-Hausdorff} Let $s\in[0,\infty)$. For all $\delta\in(0,\infty]$ and for all sets $E\subset\RR^n$, we define \begin{equation}\Cube^s_\delta(E):=\inf\left\{\sum_1^\infty (\side Q_i)^s:E\subseteq\bigcup_1^\infty Q_i,\text{ $Q_i$ is a cube with $\side Q_i\leq \delta$}\right\}.\end{equation} We call $\Cube^s_\infty$ the \emph{$s$-dimensional cubical content} on $\RR^{n}$.  The \emph{$s$-dimensional cubical measure} on $\RR^{n}$ is defined by $\Cube^s(E):=\lim_{\delta\downarrow 0} \Cube^s_\delta(E)$ for all $E\subset\RR^{n}$.\end{definition}

The quantities $\Cube^s_\delta$ and $\Cube^s$ are outer measures on $\RR^n$ and Borel sets are $\Cube^s$ measurable in the sense of Caratheodory. Moreover, $\Cube^s_\delta$ and $\Cube^s$ share the same null sets. Notably, the contents $\Cube^s_{\delta}$ are \emph{not} Borel measures for any $\delta\in(0,\infty]$ unless $s=n$.

\begin{example} When $s>n$, $\mathcal{Q}^s_\infty=\mathcal{Q}^s=0$ (the zero outer measure), and when $s=n$, $\mathcal{Q}^n_\infty=\mathcal{Q}^n=\mathcal{L}^n$ (the Lebesgue outer measure on $\RR^n$). When $s<n$, $\mathcal{Q}^s_\infty$ and $\mathcal{Q}^s$ are distinct outer measures. For instance, when $s=0$, $\mathcal{Q}^0$ is the counting measure and $$\mathcal{Q}^0_\infty(E) = \left\{\begin{array}{ll}
0, & \text{if $E=\emptyset$}\\
1, & \text{if $E$ is nonempty and bounded}\\
\infty, &\text{if $E$ is unbounded}\end{array}\right..$$ All sets are $\mathcal{Q}^0$ measurable, but the only $\mathcal{Q}^0_\infty$ measurable sets are $\RR^n$ and $\emptyset$.
\end{example}

\begin{principle}[bigger cubes are more efficient than smaller cubes] \label{p:bigger} Let $s=n-\rho<n$. Suppose that a cube $P$ is partitioned into two or more cubes $Q_1,Q_2,\dots,Q_k$, i.e.~ $P=\bigcup_1^k Q_i$ and $Q_i\cap Q_j=\emptyset$ whenever $i\neq j$. While $\vol P = \sum_1^k \vol Q_i$, \begin{equation*} (\side P)^{s} = (\vol P)^{1-\rho/n} < (\vol P)^{1-\rho/n} \sum_1^k \left(\frac{\vol Q_i}{\vol P}\right)^{1-\rho/n} = \sum_1^k (\side Q_i)^{s},\end{equation*} because $t_1,\dots,t_k\in(0,1)$ and $\sum_1^k t_i=1$ implies $\sum_1^k t_i^{1-\epsilon}>1$ for all $\epsilon\in(0,1]$. From this we obtain the following simple, but important principle. When using covers of disjoint cubes to estimate $s$-dimensional cubical contents of sets in $\RR^n$ from above, it is generally more efficient to combine smaller cubes into bigger cubes wherever possible.\end{principle}

\begin{example}\label{example:goes-to-infinity} Suppose that $s<n$ is fixed. Let $m\geq 3$ be an integer. If we partition a cube $P$ of side length 1 into $m^n$ cubes $Q_1,\dots,Q_{m^n}$ of side length $1/m$, then \begin{equation}\label{goes-to-infinity-with-m}\sum_{i=1}^{m^n} (\side Q_i)^s = m^{n-s}\rightarrow\infty\quad \text{as $m\rightarrow \infty$}.\end{equation} Note that even if we deleted one of the cubes from the list, then it is still possible for the sum over the remaining cubes to exceed 1: \begin{equation*}S=\sum_{i=1}^{m^n-1}(\side Q_i)^s = m^{n-s}-m^{-s}>1\end{equation*} unless $n-s$ is sufficiently small depending on $m$. Thus, if $E\subset Q_1\cup\cdots\cup Q_{m^n-1}$, then it might still be more efficient to use $P$ to estimate the $s$-dimensional cubical content of $E$ instead of the smaller cubes $Q_1,\dots,Q_{m^n-1}$.

Now, $S\leq 1$ if and only if $n-s \leq -\log_m(1-m^{-n})$. By concavity, the tangent line of $\log(1+x)$ at $x=0$ lies above the graph of the function. Hence, for all $-1<x<\infty$, $\log(1+x)\leq x$. It follows that for all $0\leq x<1$, \begin{equation}\label{log-convex-2} x\leq -\log(1-x)=\log\left(\frac{1}{1-x}\right)=\log\left(1+\frac{x}{1-x}\right)\leq \frac{x}{1-x}.\end{equation} Thus, we see that a sufficient condition for $S\leq 1$ is for $n-s \leq 1/(m^n\log(m))$, which roughly speaking is smaller than the volume of the ``hole'' (the deleted cube) by a factor of $\log(m)$ that blows up as $m\rightarrow\infty$.\end{example}

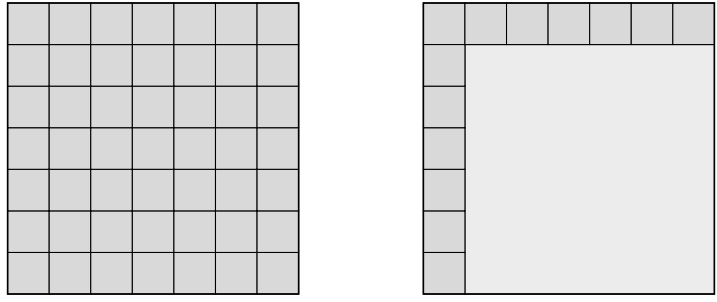
\begin{figure}[t]%
\centering
\begin{tikzpicture}[x=0.55cm,y=0.55cm,line join=miter]
  \def\m{7}
  \pgfmathtruncatemacro{\mA}{\m-1}
  \pgfmathtruncatemacro{\mB}{\m-2}

  \foreach \i in {0,...,\mA}
    \foreach \j in {0,...,\mA}
      \fill[gray!30,draw=black] (\i,\j) rectangle ++(1,1);
  \draw[semithick] (0,0) rectangle (\m,\m);

  \begin{scope}[shift={(\m+3,0)}]
    \fill[gray!15,draw=black] (1,0) rectangle (\m,\mA);
    \foreach \i in {0,...,\mA}
      \fill[gray!30,draw=black] (\i,\mA) rectangle ++(1,1);
    \foreach \j in {0,...,\mB}
      \fill[gray!30,draw=black] (0,\j) rectangle ++(1,1);
    \draw[semithick] (0,0) rectangle (\m,\m);
  \end{scope}
\end{tikzpicture}
\caption{Illustrating Examples \ref{example:goes-to-infinity} and \ref{example:goes-to-1} with $n=2$ and $m=7$. For $s\in(n-1,n)$, the mixed subdivision on the right is substantially more efficient than the regular $m$-adic subdivision on the left.} \label{fig:two-divisions}
\end{figure}

\begin{example}\label{example:goes-to-1} Let $s\in(n-1,n)$ be fixed. Let $m\geq 3$ be an integer. If we partition a cube $P$ of side length 1 into one cube $Q_0$ of side length $1-(1/m)$ and $N=m^n-(m-1)^n$ cubes $Q_1,\dots,Q_{N}$ of side length $1/m$ (see Figure \ref{fig:two-divisions}), then by the binomial theorem, \begin{equation}\begin{split} \label{goes-to-1-with-m}
\sum_{i=0}^{N}(\side Q_i)^{s} &= (1-m^{-1})^{s}+m^{n-s}\left(1-(1-m^{-1})^n\right)\\
&= (1-m^{-1})^s+O(m^{n-s-1})\rightarrow 1\quad \text{as }m\rightarrow\infty,\end{split}\end{equation} because $n-s-1<0$ and $O(m^{n-s-1})\leq C m^{n-s-1}$ for some constant $C$ depending only on $n$. Compare \eqref{goes-to-1-with-m} with \eqref{goes-to-infinity-with-m}. This simple observation is one source of our improved estimates for Bourgain's constant $b_n$ compared with the estimates in \cite{15-9s}.

If we omit one of the smaller cubes, then the $s$-dimensional content of the list becomes \begin{equation*}\begin{split}\widetilde S&=\sum_{i=0}^{N-1}(\side Q_i)^s = (1-m^{-1})^{s}+m^{n-s}\left(1-(1-m^{-1})^n-m^{-n}\right)\\
&= m^{\rho}\Big(1-(1-(m-1)^{-\rho})(1-m^{-1})^n-m^{-n}\Big),\quad\text{where }\rho=n-s.\end{split}\end{equation*} Hence $\widetilde S\leq 1$ if and only if \begin{equation}\label{S-tilde-iff}\underbrace{1-m^{-\rho}}_{A} \leq \underbrace{(1-(m-1)^{-\rho})}_{B}\underbrace{(1-m^{-1})^n}_{C}+m^{-n}.\end{equation} To find sufficient conditions on $\rho$ for $\widetilde S\leq 1$, let us bound $A$ from above and $B$ and $C$ from below. Applying \eqref{log-convex-2} with $x=1-m^{-\rho}$, we have \begin{equation*}A=1-m^{-\rho} \leq -\log(1-(1-m^{-\rho}))=\rho\log(m).\end{equation*} From the Taylor series bound $e^{-x}\leq 1-x+(1/2)x^2$ for all $x\geq 0$, it follows that \begin{equation*} B=1-e^{-\rho\log(m-1)} \geq \rho\log(m-1)-\frac{1}{2}\rho^2\log(m-1)^2.\end{equation*} Next, by considering the sign of the derivative of $f(x)=(1+x)^n-(1+nx)$, one may verify that $f(x)$ has a global minimum on $[-1,\infty)$ at $x=0$. This is Bernoulli's inequality: \begin{equation}\label{Bernoulli} (1+x)^n \geq 1+nx\quad\text{for all $x\geq -1$.}\end{equation} Hence $C=(1-m^{-1})^n \geq 1-(n/m)$. Substituting our bounds for $A$, $B$, and $C$ into \eqref{S-tilde-iff} and rearranging, we discover that $\widetilde S\leq 1$ if $\rho=n-s<1$ and $\rho$ satisfies the quadratic inequality \begin{equation}\label{quadratic-inequality} \left[\frac{1}{2}\left(1-\frac{n}{m}\right)\log(m-1)^2\right]\rho^2 +\left[\log\left(1+\frac{1}{m-1}\right)+\frac{n}{m}\log(m-1)\right]\rho\leq m^{-n}.\end{equation} Suppose that $0< \rho \leq m^{-n}/(n+1)$ and recall that $n\geq 1$ and $m\geq 3$. Estimating generously, the left hand side of \eqref{quadratic-inequality} is bounded above by \begin{equation*}\begin{split} \left[\frac{1}{2}\log(m)^2\right]\frac{m^{-2n}}{(n+1)^2} &+\left[\log(3/2) + \frac{n}{m}\log(m)^2\right]\frac{m^{-n}}{n+1} \\
&\leq \left(\frac{1}{2}\log(3/2) + \frac{9}{8}\frac{\log(m)^2}{m}\right)m^{-n} < 0.82 m^{-n},\end{split} \end{equation*} where the reader can verify the final inequality by checking that the expression in round parentheses is maximized over integers $m\geq 3$ when $m=7$. Thus, a sufficient condition for $\widetilde S\leq 1$ is for $n-s \leq m^{-n}/(n+1)$, which is smaller than the volume of the deleted cube by a factor $n+1$ that is independent of $m$.
\end{example}

The previous pair of examples illustrate the following principle.

\begin{principle}[volume of a hole provides a good heuristic for when to use smaller cubes] \label{p:hole} When there a ``hole'' or ``empty space'' $Q$ inside of a cube $P$, then a cover of $P\setminus Q$ with cubes of side length smaller than $P$ can provide a better estimate for the $s$-dimensional cubical content of $P\setminus Q$ than $(\side P)^s$ as long as $n-s$ is (a bit) smaller than $\mathcal{L}^n(Q)/\mathcal{L}^n(P)$. Using uniform covers of side length $t\side P$ requires something like $n-s$ is less than the relative Lebesgue measure divided by $\log(1/t)$. Using covers of mixed sizes can lead to a larger range of valid $n-s$.\end{principle}

Let $B(x,r)$ denote the open ball with center $x\in\RR^n$ and radius $r>0$, respectively.

\begin{lemma}\label{ball-volume} Denote the volume of the unit ball in $\RR^n$ by $\omega_n$, i.e.~ \begin{equation}\omega_n = \frac{\pi^{n/2}}{\Gamma((n+2)/2)},\qquad \Gamma(x)= \int_0^\infty t^{x-1}e^{-t}\,dt\quad (x>0).\end{equation} For all $n\geq 1$,  \begin{equation}\frac{1}{e^{1/(6n+12)}}\cdot \frac{\sqrt{2\pi e}^{n+2}}{2\sqrt{\pi}^3\sqrt{n+2}^{n+1}} \leq \omega_n \leq \frac{\sqrt{2\pi e}^{n+2}}{2\sqrt{\pi}^3\sqrt{n+2}^{n+1}}.\end{equation}
\end{lemma}

\begin{proof} The formulas for the volume of an $n$-ball are well-known. Apply the Stirling's inequalities for the gamma function, see e.g.~\cite{Artin-Gamma} or \cite{Jameson-Stirling}, with $x=(n+2)/2$: \begin{equation} \frac{\sqrt{2\pi}}{\sqrt{x}} \left(\frac{x}{e}\right)^x\leq \Gamma(x) \leq \frac{\sqrt{2\pi}}{\sqrt{x}} \left(\frac{x}{e}\right)^x e^{1/(12x)}\quad\text{for all $x>0$}.\qedhere\end{equation}\end{proof}

\begin{figure}[t]
    \centering
    \small
    \begin{tabular}{c c c c @{\quad} !{\color{sepcol}\vrule} @{\quad} c c c c}
        \toprule
        {$n$} & {$\operatorname{lower}(n)$} & {$\omega_n$} & {$\operatorname{upper}(n)$} &
        {$n$} & {$\operatorname{lower}(n)$} & {$\omega_n$} & {$\operatorname{upper}(n)$} \\
        \midrule
        \rowcolor{rowgray}
         1 & 1.99851\ldots & 2.00000\ldots & 2.11268\ldots &  6 & 5.16749\ldots & 5.16771\ldots & 5.27627\ldots \\
         2 & 3.14056\ldots & 3.14159\ldots & 3.27419\ldots &  7 & 4.72462\ldots & 4.72476\ldots & 4.81293\ldots\\
        \rowcolor{rowgray}
         3 & 4.18807\ldots & 4.18879\ldots & 4.33003\ldots &  8 & 4.05862\ldots & 4.05871\ldots & 4.12683\ldots\\
         4 & 4.93430\ldots & 4.93480\ldots & 5.07329\ldots & 9 & 3.29845\ldots & 3.29850\ldots & 3.34881\ldots\\
        \rowcolor{rowgray}
         5 & 5.26345\ldots & 5.26378\ldots & 5.39028\ldots & 10 & 2.55013\ldots & 2.55016\ldots & 2.58579\ldots \\
        \bottomrule
    \end{tabular}
    \normalsize
     \caption{Exact values vs.~bounds for $\omega_n$ in Lemma \ref{ball-volume} for $1\leq n\leq 10$.}
    \label{tab:omega_bounds}
\end{figure}

Table \eqref{tab:omega_bounds} shows  that the bounds in Lemma~\ref{ball-volume} rapidly approximate $\omega_n$, particularly the lower bound with its exponential correction. Note that $\omega_n$ is maximized when $n=5$.

\begin{example}\label{example:15-175} Let $0<s<n$. Let's estimate the $s$-dimensional cubical content of the open unit ball $B(0,1)$. When $s>n-1$, it is easy to verify that $\mathcal{Q}_\infty^s(\partial B(0,1))=0$ and $\mathcal{Q}_\infty^s(\overline{B(0,1)})=\mathcal{Q}_\infty^s(B(0,1))$. Circumscribing $B(0,1)$ with a single cube of side length $2=\diam B(0,1)$, we obtain an initial bound \begin{equation}\label{2-to-s-bound}\Cube^s_\infty(B(0,1)) \leq 2^s.\end{equation} Cover $B(0,1)$ with pairwise disjoint cubes of side length $t$ that intersect $B(0,1)$; e.g., take cubes from some partition of $\RR^n$ by half-open cubes of side length $t$. Any such cube has diameter $\sqrt{n}t$ and is contained inside $\widehat B=B(0,1+\sqrt{n}t)$. Therefore, the number of cubes $N_j$ of side length $t$ needed to cover $B(0,1)$ is no more than $$\frac{\text{volume of $\widehat{B}$}}{\text{volume of a cube}}=\frac{\omega_n(1+\sqrt{n}t)^n}{t^n}.$$
Covering $B(0,1)$ with said cubes, we obtain \begin{equation*}\Cube^s_\infty(B(0,1)) \leq \omega_n(1+\sqrt{n}t)^nt^{-(n-s)}.\end{equation*} By a simple calculus exercise, the latter expression is minimized when $t=(n-s)/(s\sqrt{n})$: \begin{equation}\Cube^s_\infty(B(0,1)) \leq \omega_n\frac{n^n}{s^s}\left(\frac{\sqrt{n}}{n-s}\right)^{n-s}.\end{equation} This is a reasonable discovery insofar as the upper bound on $\Cube^s_\infty(B(0,1))$ tends to $\omega_n=\Cube^n(B(0,1))=\Cube^n_\infty(B(0,1))$ as $s\uparrow n$. However, we note that this bound is not sharp by Principle \ref{p:bigger}, because we could combine many of the little cubes of side length $t$ inside of $B(0,1)$ into bigger cubes. To the authors' knowledge, the exact value of the content $\omega_{n,s}:=\mathcal{Q}^s_\infty(B(0,1))$ is unknown and computing it is an \emph{open problem}.

\begin{figure}[t]
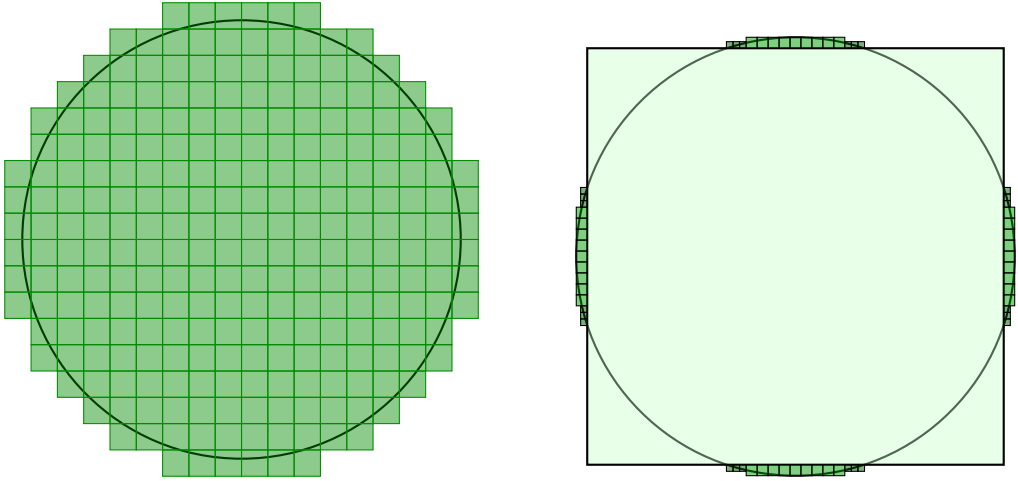


\caption{Illustrating Example \ref{example:15-175} with $n=2$ and $s=1.75$. As $s\uparrow n$, uniform covers require increasingly small cubes, while covers using cubes of several sizes can be more efficient.}
\label{fig:15-175}
\end{figure}

For instance, when $n=2$ and $s=1.75$, we can cover $B(0,1)$ with 256 squares of side length $t=0.12$ (see Figure \ref{fig:15-175}); this cover leads to the bound $\Cube^{1.75}_\infty(B(0,1)) \leq 6.2633...$, which is substantially worse than the bound $\Cube^{1.75}_\infty(B(0,1))\leq 2^{1.75}=3.3635...$ from \eqref{2-to-s-bound}. However, we can beat the $2^{1.75}$ bound by using several side lengths: cover $B(0,1)$ with one cube of side length $1.9$, 36 cubes of side length $0.05$ and 24 cubes of side length $0.03$ to obtain $\Cube^{1.75}_\infty(B(0,1))\leq 3.3170...$, greater than 1\% improvement over \eqref{2-to-s-bound}.\end{example}

\begin{example}\label{example:outer-cube} Let $s\in(0,n)$, close to $n$, and let $B=B(x,r)$ be any (open) ball contained in $[0,1]^n$. We wish to estimate the $s$-dimensional cubical content of the set $E=[0,1]^n\setminus B$ in terms of $n$, $s$, and $r$, but independent of the location of $B$ within the $[0,1]^n$. In contrast with the previous example, we add a further stipulation: we want to estimate $\Cube^s_\infty(E)$ using only pairwise disjoint cubes that are contained in $[0,1]^n$.

\begin{figure}[t]
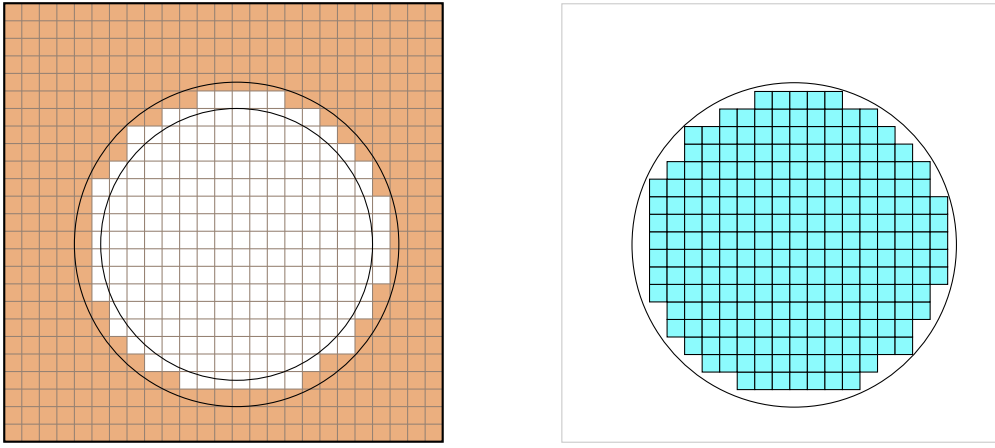

\centering

\caption{Illustrating Example \ref{example:outer-cube} with $n=2$ and $h=25$. The blue squares lie inside $B(x,r)$, while the orange squares intersect the complement of $B(x,r)$ and are disjoint from the interior of $B(x,t)$.} \label{fig:exterior-circle}
\end{figure}

Let $h\geq 2$ be a (large) integer. Partition $[0,1]^n$ into a family $\mathcal{Q}_h$ of $h^n$ pairwise disjoint cubes of side length $1/h$. Let $\mathcal{R}_h$ denote the subfamily of all cubes in $\mathcal{Q}_h$ that intersect $E$. Given $0<t<r$ (close to $r$), assume that $h$ is sufficiently large so that $r-\sqrt{n}/h\geq t$, i.e.~assume that $h\geq \sqrt{n}/(r(1-t/r))$. We remark that $h$ tends to infinity as the radius of $B$ goes to $0$ and as well as when $t$ approaches $r$. No cube in $\mathcal{R}_h$ intersects the interior of $B(x,t)$. Thus, $$\#\mathcal{R}_h \leq \frac{\text{volume of $[0,1]^n$ $-$ volume of $B(x,t)$}}{\text{volume of a cube}}=\frac{1-\omega_nt^n}{(1/h)^n}.$$ Because the cubes in $\mathcal{R}_h$ cover $E$ and have side length $1/h$, we obtain \begin{equation}\label{cube-minus-ball} \Cube^s_\infty(E) \leq \sum_{R\in\mathcal{R}_h}(\side R)^s \leq h^{n-s}\left(1-\omega_n t^n\right).\end{equation} In order for \eqref{cube-minus-ball} to improve on the trivial bound $\Cube^s_\infty(E)\leq \Cube^s_\infty([0,1]^n)\leq 1$, we need $s$ to be sufficiently close to $n$ so that $(n-s)\log(h) <-\log\left(1-\omega_n t^n\right)$. In particular, it suffices to take $n-s<\omega_n t^n/(\log h)$,
which is consistent with Principle~\ref{p:hole}.\end{example}

\section{Bourgain-type dimension lemma with cubes in general position}\label{sec:gmt}

For $s\in[0,\infty)$ and $E\subset\RR^n$, let $\Haus^s(E)$ denote the $s$-dimensional Hausdorff measure, defined as in Definition \ref{def:cubical-Hausdorff} but allowing covers by arbitrary sets $E_i\subset\RR^n$ and replacing $(\side Q_i)^s$ by $(\diam E_i)^s$. Since $\Cube^s$ and $\Haus^s$ are comparable up to dimensional constants,
$$
\dim_H E:=\inf\{s\geq 0:\Haus^s(E)<\infty\}
=\inf\{s\geq 0:\Cube^s(E)<\infty\}.
$$
For a Borel measure $\mu$ on $\RR^n$, its \emph{upper Hausdorff dimension} is
$$
\overline{\dim}_H\,\mu
:=\inf\{\dim_H E:\mu(\RR^n\setminus E)=0,\ E\text{ Borel}\}.
$$
We will use the countable stability of both notions: $\dim_H\bigcup_iE_i=\sup_i\dim_HE_i$, and if $\RR^n=\bigcup_iE_i$ with each $E_i$ Borel, then $\overline{\dim}_H,\mu=\sup_i\overline{\dim}_H(\mu\res E_i)$.

The following dimension lemma originates in Bourgain's work on the dimension of harmonic measure \cite{Bourgain}, where the relevant estimates appear within the proof in a form adapated to its specific setting. An explicit formulation and proof appear in the parabolic setting in \cite[Theorem 2.10]{dim-caloric} and a Euclidean version appears in \cite{15-9s}. The formulation below allows pairwise disjoint covering cubes in general position and arbitrary $\alpha\in(0,1)$. We also refine the proof to eliminate a loss in the type~(II) iteration, thereby recovering the factor $1-\alpha$ in the definition of $\lambda_\alpha$.

\begin{theorem}[Bourgain-type dimension lemma with cubes in general position]\label{dim-lemma} Let $\mu$ be a Borel regular outer measure on $\RR^{n}$ that is finite on bounded sets, and let $E\subset\RR^{n}$ be a Borel set with $\mu(\RR^{n}\setminus E)=0$. Suppose that there exist parameters $\alpha\in(0,1)$, $\rho\in(0,n)$, $\beta\in(0,1)$, $\tau\in(1,\infty)$, and $\gamma\in(0,1)$ such that for every cube $P$ in $\RR^n$ with $\side P\leq 1$, at least one (possibly both) of the following conditions holds: \begin{enumerate}
\item[(I)] there is a cover of $E\cap P$ by a finite or countable list of pairwise disjoint cubes $Q_1, Q_2,\dots$ contained in $P$ with $\side Q_i\leq \beta \side P$ such that $$\sum_{i} (\side Q_i)^{n-\rho}\leq (\side P)^{n-\rho};$$
\item[(II)] there is a cover of $E\cap P$ by a finite or countable list of pairwise disjoint cubes $R_1,R_2,\dots$ contained in $P$ with $\side R_i\leq \beta\side P$ such that $$\sum_i (\side R_i)^{n-\rho}\leq \tau\,(\side P)^{n-\rho}\quad\text{and}\quad\sum_i \mu(R_i)^{\alpha}(\vol R_i)^{1-\alpha} \leq \gamma\, \mu(P)^{\alpha}(\vol P)^{1-\alpha}.$$
\end{enumerate} Then the upper Hausdorff dimension of $\mu$ is at most \begin{equation}\label{lambda-t-general} n-\left(\frac{\lambda_\alpha}{1+\lambda_\alpha}\right)\rho,\quad \text{where }\lambda_\alpha:=\frac{-\log_\tau(\gamma)}{1-\alpha}\in(0,\infty).\end{equation} One may, if desired, restrict throughout to cubes that are half-open of some fixed type.
\end{theorem}

\begin{proof}
We may assume that $\mu(E)>0$, otherwise $\mu$ is just the zero measure. Although the theorem permits cubes of different types (for later applications), the proof uses only their nesting, pairwise disjointness, and side lengths, not the assignment of cube boundaries. Thus, if the hypotheses are known only for half-open cubes of one fixed type, the same proof applies with the initial unit cube chosen to have that type. By countable stability, it suffices to restrict $\mu$ to a unit cube $P_0$ of positive measure. Replacing $E$ by $E\cap P_0$ and $\mu$ by $\mu\res P_0$, and then normalizing $\mu$, we may assume that $E\subset P_0$, $\side P_0=1$, and $\mu(P_0)=1$. Conditions (I) and (II) persist under these reductions.

Fix $0<\epsilon<1$ and $0<\sigma<\lambda_\alpha\rho\epsilon$. We will prove that
\[
\overline{\dim}_H\,\mu\leq \max\{n-\rho(1-\epsilon),n-\sigma\}.
\]
This suffices, because for $\epsilon=1/(1+\lambda_\alpha)$, we have $\rho(1-\epsilon)=\lambda_\alpha\rho\epsilon=\lambda_\alpha\rho/(1+\lambda_\alpha),$ and $\sigma$ may be taken arbitrarily close to $\lambda_\alpha\rho\epsilon$ from below. Fix any $0<\delta<1$ sufficiently small such that $\log_\tau(\delta^{-\rho\epsilon})\geq 2$ and let $s=\lceil \log_\tau(\delta^{-\rho\epsilon})\rceil$. Since $s-1<\log_\tau(\delta^{-\rho\epsilon})\leq s$,
\begin{equation}\label{t-power-decay}
\tau^{s-1} < \delta^{-\rho\epsilon}\leq \tau^s.
\end{equation}

We now inductively build a tree $\mathcal{T}=\bigcup_{l=0}^{L} \mathcal{T}_l$ of finite depth $L$. Each set in $\mathcal{T}$ is a cube contained in $P_0$ and the cubes in any given level $\mathcal{T}_l$ are pairwise disjoint. Each cube $Q\in\mathcal{T}_{l+1}$ is contained in a unique cube $P\in\mathcal{T}_l$ and satisfies $\side Q\leq \beta\side P$, but in general there is no lower bound on $\side Q$ in terms of $\side P$. For any cube $P\in\mathcal{T}_l$, we let $\Child(P)=\{Q\in\mathcal{T}_{l+1}:Q\subset P\}$. We call $P\in\mathcal{T}$ \emph{terminal} if $\side P\leq\delta$ or $E\cap P=\emptyset$. Terminal cubes will have no children, while every non-terminal cube will have at least one child. The depth $L$ is finite, because $\side Q\leq\beta\side P$ forces the side lengths along any chain to decay geometrically until the terminal condition $\side P\leq\delta$ is reached.

\emph{Base Case.} At the top level, put $\mathcal{T}_0=\{P_0\}$.

\emph{Induction Step.} Suppose that $\mathcal{T}_l\neq\emptyset$ for some $l\geq 0$ and let $P\in\mathcal{T}_l$. If $P$ is terminal, we add no children. Otherwise, we shall include a nonempty, finite or countable family of pairwise disjoint subcubes of $P$ as sets in $\mathcal{T}_{l+1}$. On the one hand, if $P$ satisfies condition (I), then we say that $P$ is \emph{type 1}. Choose pairwise disjoint cubes $Q_1,Q_2,\dots$ contained in $P$ with $\side Q_i\leq\beta\side P$ that cover $E\cap P$ and have
\begin{equation}\label{e:type1}
\sum_i(\side Q_i)^{n-\rho}\leq(\side P)^{n-\rho}.
\end{equation}
In this case, we include the cube $Q_i\in\mathcal{T}_{l+1}$ for each $i$, so that $\Child(P)=\{Q_1,Q_2,\dots\}$. On the other hand, if $P$ fails condition (I), then $P$ satisfies condition (II) and we say that $P$ is \emph{type 2}. Choose pairwise disjoint cubes $R_1,R_2,\dots$ contained in $P$ with $\side R_i\leq\beta\side P$ that cover $E\cap P$ and satisfy
\begin{equation}\label{e:type2}
\sum_i(\side R_i)^{n-\rho}\leq\tau\,(\side P)^{n-\rho},\quad
\sum_i\mu(R_i)^\alpha(\vol R_i)^{1-\alpha}
\leq\gamma\,\mu(P)^\alpha(\vol P)^{1-\alpha}.
\end{equation}
In this case, we include each $R_i$ in $\mathcal{T}_{l+1}$, so that $\Child(P)=\{R_1,R_2,\dots\}$. Repeating this procedure for each $P\in\mathcal{T}_l$ completes the definition of $\mathcal{T}_{l+1}$.

\emph{Goal.} Write $\eta:=\max\{n-\rho(1-\epsilon),n-\sigma\}$. We aim to find a Borel set $F_\delta\subset E$ with
\begin{equation}\label{M-goal}
\lim_{\delta\downarrow 0}\mu(\RR^{n}\setminus F_{\delta})
=\lim_{\delta\downarrow 0}\mu(E\setminus F_\delta)=0
\quad\text{and}\quad
\Cube^{\eta}_{\zeta(\delta)}(F_\delta)\leq 2<\infty,
\end{equation}
where $\zeta(\delta)\rightarrow0$ as $\delta\rightarrow0$ is a scale to be specified below.

Let $\mathcal{T}^*$ denote the terminal cubes in $\mathcal{T}$ that intersect $E$. Every $Q\in\mathcal{T}^*$ satisfies $\side Q\leq\delta$. Since $\side P_0=1>\delta$ and $\mu(E)=1$, the initial cube $P_0$ is not terminal. We now use terminal cubes to identify two collections $\mathcal{E},\mathcal{F}\subset\mathcal{T}$ with $\mathcal{E}\cap\mathcal{F}=\emptyset$, as well as a subfamily $\mathcal{G}\subset\mathcal{F}$. We say that $Q\in\mathcal{T}^*\cap\mathcal{T}_l$ is \emph{efficient for $\Cube^{n-\rho}$} and include $Q\in\mathcal{E}$ if the chain $P_0,\dots,P_{l-1}$ of ancestors of $Q$ contains fewer than $s$ type 2 cubes. Otherwise, there is a unique ancestor $P$ of $Q$ that is the $s$-th type 2 cube in the chain, and we include $P\in\mathcal{F}$. By construction, $$E\subset\bigcup\mathcal{T}^*\subset\bigcup\mathcal{E}\cup\bigcup\mathcal{F}.$$ Furthermore, $\side Q\leq\delta$ for all $Q\in\mathcal{E}$ and $\delta<\side Q\leq\beta^{s-1}$ for all $Q\in\mathcal{F}$, since cubes in $\mathcal{E}$ belong to $\mathcal{T}^*$ and cubes in $\mathcal{F}$ are not terminal and have at least $s-1$ ancestors. Thus, $\mathcal{E}\cup\mathcal{F}$ covers $E$ by cubes of side length at most $\zeta(\delta):=\max\{\delta,\beta^{s-1}\}$. With $\epsilon$ fixed, as $\delta\rightarrow0$ we have $s\rightarrow\infty$ and $\zeta(\delta)\rightarrow0$. Finally, let $\mathcal{G}$ consist of those $Q\in\mathcal{F}$ satisfying
\[
(\side Q)^{n-\sigma}\leq\mu(Q)
\]
and put $F_\delta=E\cap\bigcup(\mathcal{E}\cup\mathcal{G})$. Since $E$ is Borel, the set $F_\delta$ is Borel.

We now make estimates towards \eqref{M-goal}. Both of the sums below run over cubes lying at different levels of $\mathcal{T}$ and are controlled by the following observation. For $Q\in\mathcal{T}$, let $t(Q)$ denote the number of ancestors of $Q$ in $\mathcal{T}$ of type 2. Suppose that a number $w(Q)\geq0$ is assigned to each $Q\in\mathcal{T}$ and $\kappa>0$ is a constant such that
\begin{equation}\label{w-hypothesis}
\sum_{Q\in\Child(P)}w(Q)\leq w(P)\text{ if $P$ is type 1},\quad
\sum_{Q\in\Child(P)}w(Q)\leq\kappa w(P)\text{ if $P$ is type 2}.
\end{equation}
If $\mathcal{A}\subset\mathcal{T}$ is a pairwise disjoint family of cubes, then
\begin{equation}\label{tree-sum}
\sum_{Q\in\mathcal{A}}w(Q)
\leq\Big(\sup_{Q\in\mathcal{A}}\kappa^{\,t(Q)}\Big)w(P_0).
\end{equation}
Indeed, put $W(P)=\sum\{w(Q):Q\in\mathcal{A},\,Q\subseteq P\}$ and $M(P)=\sup\{\kappa^{\,t(Q)-t(P)}:Q\in\mathcal{A},\,Q\subseteq P\}$. A downward induction gives $W(P)\leq M(P)w(P)$: this is immediate if no cube of $\mathcal{A}$ lies in $P$ or if $P\in\mathcal{A}$, and otherwise follows by summing over $\Child(P)$ and applying \eqref{w-hypothesis}. Taking $P=P_0$, for which $t(P_0)=0$, gives \eqref{tree-sum}.

Assign $w(Q)=(\side Q)^{n-\rho}$ and $\kappa=\tau$, so that \eqref{w-hypothesis} is precisely \eqref{e:type1} together with the first inequality in \eqref{e:type2}. Cubes in $\mathcal{E}$ are terminal, hence pairwise disjoint, and $t(Q)\leq s-1$ for each of them. Thus, by \eqref{tree-sum} and \eqref{t-power-decay},
\[
\sum_{Q\in\mathcal{E}}(\side Q)^{n-\rho}
\leq\tau^{s-1}\leq\delta^{-\rho\epsilon}.
\]
On the one hand, since $\side Q\leq\delta<1$ for all $Q\in\mathcal{E}$,
\[
\Cube^{\eta}_{\delta}(E\cap\bigcup\mathcal{E})
\leq\sum_{Q\in\mathcal{E}}(\side Q)^{n-\rho(1-\epsilon)}
\leq\delta^{\rho\epsilon}\sum_{Q\in\mathcal{E}}(\side Q)^{n-\rho}
\leq1.
\]
On the other hand, since $\side Q\leq\beta^{s-1}<1$ for all $Q\in\mathcal{G}$,
\[
\Cube^{\eta}_{\beta^{s-1}}(E\cap\bigcup\mathcal{G})
\leq\sum_{Q\in\mathcal{G}}(\side Q)^{n-\sigma}
\leq\sum_{Q\in\mathcal{G}}\mu(Q)\leq1,
\]
because cubes in $\mathcal{G}$ are pairwise disjoint subsets of $P_0$, satisfy the defining inequality for $\mathcal{G}$, and $\mu(P_0)=1$. Thus,
\[
\Cube^{\eta}_{\zeta(\delta)}(F_\delta)
\leq\Cube^{\eta}_{\delta}(E\cap\bigcup\mathcal{E})
+\Cube^{\eta}_{\beta^{s-1}}(E\cap\bigcup\mathcal{G})
\leq2<\infty.
\]
To verify \eqref{M-goal}, it remains to estimate $\mu(E\setminus F_\delta)$ in terms of $\delta$.

For every $Q\in\mathcal{F}\setminus\mathcal{G}$, we have
\begin{equation}\label{vol-decomp}
\vol Q=(\side Q)^\sigma(\side Q)^{n-\sigma}>\delta^\sigma\mu(Q).
\end{equation}
Furthermore, by H\"older's inequality with exponents $1/\alpha$ and $1/(1-\alpha)$, for all non-terminal cubes $P\in\mathcal{T}$,
\begin{equation}\begin{split}\label{cauchy-schwarz}
\sum_{Q\in\Child(P)}\mu(Q)^\alpha(\vol Q)^{1-\alpha}
&\leq\bigg(\sum_{Q\in\Child(P)}\mu(Q)\bigg)^\alpha
\bigg(\sum_{Q\in\Child(P)}\vol Q\bigg)^{1-\alpha}\\
&\leq\mu(P)^\alpha(\vol P)^{1-\alpha},
\end{split}\end{equation}
because the children of a cube are pairwise disjoint subcubes. Now assign $w(Q)=\mu(Q)^\alpha(\vol Q)^{1-\alpha}$ and $\kappa=\gamma$, so that \eqref{w-hypothesis} is \eqref{cauchy-schwarz} together with the second inequality in \eqref{e:type2}. The cubes in $\mathcal{F}$ are pairwise disjoint and $t(Q)=s-1$ for each of them, whence \eqref{tree-sum} and $\mu(P_0)=\vol(P_0)=1$ yield
\[
\sum_{Q\in\mathcal{F}}\mu(Q)^\alpha(\vol Q)^{1-\alpha}
\leq\gamma^{s-1}.
\]
Combining these estimates gives
\begin{align*}
\mu(E\setminus F_\delta)
\leq\sum_{Q\in\mathcal{F}\setminus\mathcal{G}}\mu(Q)
&=\sum_{Q\in\mathcal{F}\setminus\mathcal{G}}\mu(Q)^\alpha\mu(Q)^{1-\alpha}
\stackrel{\eqref{vol-decomp}}{<}\delta^{-\sigma(1-\alpha)}
\sum_{Q\in\mathcal{F}\setminus\mathcal{G}}\mu(Q)^\alpha(\vol Q)^{1-\alpha}\\
&\leq\delta^{-\sigma(1-\alpha)}\gamma^{s-1}.
\end{align*}
Since $\gamma=\tau^{-(1-\alpha)\lambda_\alpha}$ and $\tau^{-s}\leq\delta^{\rho\epsilon}$ by \eqref{t-power-decay},
\[
\gamma^{s-1}
=\tau^{(1-\alpha)\lambda_\alpha}(\tau^{-s})^{(1-\alpha)\lambda_\alpha}
\leq\tau^{(1-\alpha)\lambda_\alpha}
\delta^{(1-\alpha)\lambda_\alpha\rho\epsilon}.
\]
Combining the last two estimates,
\begin{equation}\label{mu-decay}
\mu(E\setminus F_\delta)
\leq\tau^{(1-\alpha)\lambda_\alpha}
\delta^{(1-\alpha)(\lambda_\alpha\rho\epsilon-\sigma)}
\rightarrow0\quad\text{as }\delta\rightarrow0,
\end{equation}
because $\sigma<\lambda_\alpha\rho\epsilon$. We have verified \eqref{M-goal}.

\emph{Conclusion.} By \eqref{mu-decay}, we may choose $\delta_1>\delta_2>\cdots\rightarrow0$ such that $\sum_1^\infty\mu(E\setminus F_{\delta_j})<\infty$ and $\log_\tau(\delta_1^{-\rho\epsilon})\geq2$, and put $F:=\liminf_{j\rightarrow\infty}F_{\delta_j}=\bigcup_{j=1}^\infty\bigcap_{i=j}^\infty F_{\delta_i}$. By the Borel--Cantelli lemma, $E\setminus F=\limsup_{j\rightarrow\infty}(E\setminus F_{\delta_j})$ has $\mu$ measure zero. Finally, $F$ is Borel and
$$
\Cube^{\eta}(F)
=\lim_{j\rightarrow\infty}\Cube^{\eta}\left(\bigcap_{i=j}^\infty F_{\delta_i}\right)
=\lim_{j\rightarrow\infty}\sup_{k\geq j}\Cube^{\eta}_{\zeta(\delta_k)}
\left(\bigcap_{i=j}^\infty F_{\delta_i}\right)
\leq2<\infty
$$
by continuity from below and \eqref{M-goal}. Therefore, $\overline{\dim}_H\,\mu\leq\dim_H F\leq\eta$.
\end{proof}

\section{Frostman measures and approximate polar integration} \label{s:polar}

\subsection{Frostman's lemma with constant 1}

In this paper, in order to describe the local size of a set using disjoint coverings of cubes, with both flexible positioning and enough rigidity for computational ease, we utilize the following terminology.

\begin{definition}\label{sites-and-nets} Let $n\geq 1$ and $m\geq 2$ be integers. Let $t\in(0,\infty)$. An \emph{$m$-adic net} $\Delta$ \emph{at initial scale $t$} is a family of cubes in $\RR^n$ such that \begin{enumerate}
\item Every cube $Q\in\Delta$ has side length $tm^{-k}$ for some integer $k\geq 0$.
\item The set $\Delta_k$ of cubes in $\Delta$ of side length $tm^{-k}$ is pairwise disjoint.
\item For every integer $k\geq 0$ and every cube $Q\in\Delta_k$, there exists a family $\Child(Q)\subseteq\Delta_{k+1}$ such that $Q=\bigcup\Child(Q)$; moreover, $\Delta_{k+1}=\bigcup_{Q\in\Delta_k}\Child(Q)$.
\end{enumerate} We call the set $\Delta_k$ the \emph{$k$-th level} of $\Delta$. We call the set $\bigcup\Delta$ the \emph{site} of $\Delta$ and denote it by $\Site(\Delta)$. A 2-adic net is called a \emph{dyadic net}. \end{definition}

Note that we do not require the cubes in an $m$-adic net $\Delta$ to be of any particular type. Furthermore, each level of a net partitions the site of the net. We introduced the concept of the site of a net in order to localize certain estimates; e.g., see \S\ref{ss:radial} and \eqref{hole-def}. In the latter application, the site of the net $\Delta_{P,x}$ is assembled from small cubes in a multiscale partition of a cube $P$ that fill out a ball from inside. Net content then permits these top-level cubes to be refined at smaller scales to cover $\partial\Omega$ efficiently within the site.

\begin{definition}Let $\Delta$ be an $m$-adic net and let $s\in[0,\infty)$. For every set $E\subset\RR^n$, we define its \emph{$s$-dimensional $\Delta$ content} by \begin{equation}\Net{s}_\infty(E):=\inf\left\{\sum_1^\infty (\side Q_i)^s:E\cap\Site(\Delta)\subseteq\bigcup_1^\infty Q_i,\text{ $Q_i\in\Delta$}\right\}.\end{equation}\end{definition}

The quantity $\Net{s}_\infty$ is an outer measure on $\RR^n$ that is carried by the site of the net in the sense that $\Net{s}_\infty(\RR^n\setminus\Site(\Delta))=0$. The main advantage of net contents over cubical contents lies in our ability to restrict to disjoint covers. The main disadvantage is that net contents are neither translation nor dilation invariant.

\begin{lemma}\label{pick-disjoint} Let $\Delta$ be an $m$-adic net in $\RR^n$ and let $s\in[0,\infty)$. For every set $E\subset\RR^n$, \begin{equation*}\Net{s}_\infty(E)=\inf\bigg\{\sum_i (\side Q_i)^s:\!\!\!\begin{array}{l}\text{$Q_1,Q_2,\dots\in\Delta$ are pairwise disjoint, and}\\ \text{$E\cap\Site(\Delta)\subseteq\bigcup_i Q_i\subseteq \Site(\Delta)$}\end{array}\!\!\bigg\}.\end{equation*}\end{lemma}

\begin{proof} Each family on the right is admissible in the definition of $\Net{s}_\infty(E)$, whence the right hand side is at least $\Net{s}_\infty(E)$. Conversely, let $Q_1,Q_2,\dots\in\Delta$ cover $E\cap\Site(\Delta)$. If $Q_j$ has side length $tm^{-k}$, then any cube in the family that contains $Q_j$ has side length $tm^{-\ell}$ for some integer $0\leq \ell\leq k$. Thus, each $Q_j$ is contained in a cube of the family that is maximal with respect to inclusion. Let $I$ index the distinct maximal cubes. Two cubes in $\Delta$ are either disjoint or nested, so $\{Q_i\}_{i\in I}$ is pairwise disjoint and covers $E\cap\Site(\Delta)$. Since the maximal cubes form a subfamily of the original family, $\sum_{i\in I}(\side Q_i)^s\leq \sum_i(\side Q_i)^s$.\end{proof}

We use the following variant of Frostman's lemma for dimensions close to \(n\). Without the assumption \(s>n-1\), the lemma remains valid, but the bound \(\mu(K)\geq \Net{s}_\infty(K)\) must be replaced by \(\mu(K)\geq 2^{-n}\Net{s}_\infty(K)\).

\begin{lemma}[{\cite[Lemma 2.3]{15-9s}}]\label{frostman} Let $\Delta$ be any $m$-adic net in $\RR^n$. If $K$ is a compact set with $K\subseteq \Site(\Delta)$ and $s\in(n-1,n]$, then there exists a Radon measure $\mu$ on $\RR^{n}$ such that $\mu(Q)\leq (\side Q)^s$ for all $Q\in\Delta$ and $\mu(\RR^n)=\mu(K)\geq \Net{s}_\infty(K)$.\end{lemma}

\subsection{Approximate polar integration of radial functions}\label{ss:radial}

In this subsection, we let $\Delta$ denote any fixed dyadic net with initial scale $h^{-1}\in(0,1]$ for some integer $h\geq 1$. For convenience, let $B^\Delta(x,r):=B(x,r)\cap\Site(\Delta)$ denote the portion of the open ball $B(x,r)$ lying in the site of the net.

\begin{definition}Let $n\geq 3$. For all $s\in(n-1,n]$ and $z\in\RR^n$, define the constant \begin{equation} O_z(n,s;\Delta) := \sup_\mu \int_{B^\Delta(0,1/2)} \frac{1}{|x-z|^{n-2}}d\mu(x),\end{equation} where the supremum is over all Borel measures $\mu$ on $\RR^n$ carried by $B^\Delta(0,1/2)$, in the sense that $\mu(\RR^n\setminus B^\Delta(0,1/2))=0$, and satisfying the Frostman-type bound $\mu(Q)\leq (\side Q)^s$ for all $Q\in\Delta$. Define \begin{equation}O(n,s;\Delta):=\sup_{z\in\RR^n}O_z(n,s;\Delta).\end{equation}\end{definition}

\begin{lemma}\label{with-Lebesgue-measure} Let $n\geq 3$ and $s\in(n-1,n]$. If $B(0,1/2)\subset\Site(\Delta)$, then \begin{equation*} O(n,s;\Delta) \geq O_0(n,s;\Delta)\geq \int_{B(0,1/2)}\frac{1}{|x|^{n-2}}\,dx=\frac{n\omega_n}{8}.\end{equation*}\end{lemma}

\begin{proof} The first inequality is apparent. If $\mu$ is Lebesgue measure restricted to $B(0,1/2)$, then $\mu(Q)\leq (\side Q)^n\leq (\side Q)^s$ for all $Q\in\Delta$, because cubes in $\Delta$ have side length at most $h^{-1}\leq 1$. Thus, we obtain a lower bound on $O_0(n,s;\Delta)$ by taking $z=0$ and integrating $\int_{B(0,1/2)} |x|^{-(n-2)}\,dx$ using polar integration (recall that $n\omega_n$ is the surface measure of the unit sphere $S^{n-1}\subset\RR^n$).\end{proof}

To estimate the dimension of harmonic measure using Bourgain's method, we need to upper bound $O(n,s;\Delta)$ when $s$ is close to $n$. We achieve this through a flexible scheme to approximate integrals of radial functions against Frostman-type measures. The main idea is to use Whitney-type cubes to transfer part of the homogeneity from the Frostman condition to the integrand in order to convert the original integral into polar integration against Lebesgue measure. The details will take up the remainder of this subsection.

\begin{lemma}\label{whitney cube lemma} Let $n\geq 3$ and $h\geq n^3$. For all $z\in B(0,2)$ and $0<\xi\leq (3h)^{-1}$, there exists a family of pairwise disjoint cubes $\mathcal{W}\subset\Delta$
such that \begin{equation}
B^\Delta(0,1/2)\setminus\{z\}\subset \bigcup\mathcal{W}\subset B^\Delta(0,1/2+2.5\sqrt{n}\xi),\end{equation} \begin{equation}\label{lower bound}0<\frac{\xi\dist(z,Q)}{2(1+\sqrt{n}\xi)}< \side Q \leq \xi \dist(z,Q)\quad\text{for all }Q\in\mathcal{W}.\end{equation}
\end{lemma}

\begin{proof} Let $z$ and $\xi$ be given. If $B^\Delta(0,1/2)\setminus\{z\}=\emptyset$, then the lemma holds with $\mathcal{W}=\emptyset$. Thus, let us assume that $B^\Delta(0,1/2)\setminus\{z\}$ is nonempty. Say that a cube $Q\in\Delta$ is \emph{admissible} if (i) $Q\cap B^\Delta(0,1/2)\neq\emptyset$, (ii) $\dist(z,Q)>0$, and (iii) $\side Q \leq \xi \dist(z,Q)$. Every admissible cube satisfies \begin{equation}\label{admissible-side}\side Q \leq \xi \dist(z,Q)\leq \xi\left(|z|+\tfrac{1}{2}\right)\leq \frac{5}{6h}<h^{-1},\end{equation} since $Q\cap B^\Delta(0,1/2)\neq\emptyset$, $z\in B(0,2)$, and $\xi\leq (3h)^{-1}$. In particular, since $\Delta$ has initial scale $h^{-1}$, every admissible cube $Q$ has a unique parent $Q^\uparrow\in\Delta$.

Let $\mathcal{W}$ denote the family of admissible cubes that are maximal with respect to inclusion among admissible cubes. Then $\mathcal{W}$ is pairwise disjoint, because distinct cubes in $\Delta$ are disjoint or nested. If $x\in B^\Delta(0,1/2)\setminus\{z\}$, then every sufficiently small cube $Q\in\Delta$ containing $x$ is admissible, since $\dist(z,Q)\rightarrow|x-z|>0$ and $\side Q\rightarrow 0$ as $\side Q\downarrow 0$; hence $\bigcup\mathcal{W}\supset B^\Delta(0,1/2)\setminus\{z\}$. Similarly, since $\diam Q =\sqrt{n}\side Q \leq 2.5\sqrt{n}\xi$ for every $Q\in\mathcal{W}$ by \eqref{admissible-side}, we have $\bigcup\mathcal{W}\subset B^\Delta(0,1/2+2.5\sqrt{n}\xi)$.

\emph{Claim: $\dist(z,Q^\uparrow)>0$ for every $Q\in\mathcal{W}$.} Assume to the contrary that $\dist(z,Q^\uparrow)=0$ for some $Q\in\mathcal{W}$. Then, since $\xi\leq (3h)^{-1}\leq(3n^3)^{-1}\leq (6\sqrt{n})^{-1}$ and $\side Q\leq \xi\dist(z,Q)$,
    \begin{equation*}\dist(z,Q)\leq \dist(z,Q^\uparrow)+\diam Q^\uparrow=2\sqrt{n}\side Q\leq \frac{1}{3\xi}\side Q
    \leq (1/3)\dist(z,Q).\end{equation*} Since $\dist(z,Q)>0$, this is impossible. The claim is established.

It remains to check the lower bound in \eqref{lower bound}. To that end, fix $Q\in\mathcal{W}$. The parent $Q^\uparrow$ satisfies (i), because $Q\subset Q^\uparrow$, and satisfies (ii), by the Claim. Thus, maximality of $Q$ forces $Q^\uparrow$ to fail (iii); that is, $\side Q^\uparrow >\xi\dist(z,Q^\uparrow)$. Hence \begin{equation*}\begin{split} 2\side Q =\side Q^\uparrow >\xi\dist(z,Q^\uparrow) &\geq \xi\left(\dist(z,Q)-\diam Q^\uparrow\right)\\
&=\xi\dist(z,Q)-2\sqrt{n}\xi\side Q,\end{split}\end{equation*} and collecting the terms in $\side Q$ establishes \eqref{lower bound}.
\end{proof}

\begin{lemma}\label{cube integral estimate} Let $n\geq 3$, let $s\in(n-1,n]$, and let $\mu$ be any admissible measure in the definition of $O(n,s;\Delta)$. Let $h\geq n^3$. Given $z\in B(0,2)$ and $0<\xi\leq (3h)^{-1}$, let $\mathcal{W}$ be a family of cubes as in Lemma \ref{whitney cube lemma}. For all $Q\in\mathcal{W}$,
   \begin{equation}\label{transfer-integral}
       \int_Q\frac{1}{|x-z|^{n-2}}\,d\mu(x)\leq \left(\frac{2}{\xi}\right)^{n-s}(1+\xi \sqrt{n})^{n-2+2(n-s)}\int_Q\frac{1}{|x-z|^{n-2+(n-s)}}\,dx,
   \end{equation}
   where the second integral is with respect to Lebesgue measure.
\end{lemma}
\begin{proof} On the one hand, the Frostman-type bound $\mu(Q)\leq (\side Q)^s$ and lower bound in \eqref{lower bound} imply
\begin{equation*}
        \int_Q \frac{1}{|x-z|^{n-2}} \,d\mu(x)
            \leq \frac{(\side Q)^n}{\dist(z,Q)^{n-2}(\side Q)^{n-s}} \leq \left(\frac{2(1+\xi\sqrt{n})}{\xi}\right)^{n-s}\frac{(\side Q)^n}{\dist(z,Q)^{n-2+(n-s)}}.
    \end{equation*}
    On the other hand, if we write $\ex(Q,z)=\sup_{y\in Q}|y-z|$ for the \emph{excess} of $Q$ over $z$, then the upper bound in \eqref{lower bound} implies $\ex(Q,z) \leq \dist(z,Q)+\diam Q\leq (1+\xi\sqrt{n})\dist(z,Q)$. Hence
    \begin{equation*}
     \int_Q \frac{1}{|x-z|^{n-2+(n-s)}}dx \geq \frac{(\side Q)^{n}}{\ex(Q,z)^{n-2+(n-s)}} \geq \frac{(\side Q)^n}{(1+\xi\sqrt{n})^{n-2+(n-s)}\dist(z,Q)^{n-2+(n-s)}}.
    \end{equation*} Merging the displayed equations yields \eqref{transfer-integral}.
\end{proof}

\begin{theorem}\label{t:O} Let $n\geq 3$ and $h\geq n^3$ be integers. For all $s\in(n-1,n]$ satisfying $n-s\leq 0.2/(n^6\log h)$,
\begin{equation}\label{O-est-2} O(n,s;\Delta) \leq \left(1+\frac{1}{n^6}\right) \frac{n\omega_n}{8}.\end{equation}
\end{theorem}

\noindent\emph{Remark.} The upper bound in \eqref{O-est-2} depends on $n$ and $h$, but otherwise is independent of the net $\Delta$.

\begin{proof} Let $s$ satisfy the hypothesis. Then $(n-s)\log h\leq 0.2n^{-6}\leq 0.2/729<\log 2$; in particular, $h^{n-s}\leq 2$. Fix $z\in\RR^n$ and let $\mu$ be an admissible measure in the definition of $O_z(n,s;\Delta)$. We must bound $O_z(n,s;\Delta)$ from above independently of $z$.

We first dispense with the case when $z$ is far away from the origin. Suppose that $|z|\geq 2$. Every cube $Q\in\Delta$ of side length $h^{-1}$ that meets $B(0,1/2)$ is contained in $B(0,(1/2)+\sqrt{n}h^{-1})$ and carries $\mu(Q)\leq h^{-s}=h^{n-s}\vol Q$. Since the cubes of side length $h^{-1}$ in $\Delta$ are pairwise disjoint and $\mu$ is carried by $B^\Delta(0,1/2)$, summing volumes gives
\begin{equation*}\mu(\RR^n)\leq h^{n-s}\omega_n\left(\frac{1}{2}+\frac{\sqrt{n}}{h}\right)^n.\end{equation*}
Because $|x-z|\geq |z|-(1/2)\geq 3/2$ for all $x\in B(0,1/2)$, and because $h^{n-s}\leq 2$, $\sqrt{n}/h\leq n^{-2.5}$, and $(1+2n^{-2.5})^n\leq \exp(2n^{-1.5})<1.5$,
\begin{equation}\label{O-part0}\int_{B^\Delta(0,1/2)} \frac{d\mu(x)}{|x-z|^{n-2}}\leq \frac{h^{n-s}\omega_n}{(3/2)^{n-2}}\left(\frac{1}{2}+\frac{\sqrt{n}}{h}\right)^n \leq \frac{n\omega_n}{8}\cdot\frac{54}{n\,3^n}\leq \frac{n\omega_n}{8},\end{equation}
the last inequality because $n\,3^n\geq 81$ for all $n\geq 3$.

Suppose instead that $z\in B(0,2)$. Fix $0<\xi\leq (3h)^{-1}$ and let $\mathcal{W}$ be a family of cubes as in Lemma \ref{whitney cube lemma}. Since $\mu$ is carried by $B^\Delta(0,1/2)$ and the Frostman-type bound applied to cubes containing a point of $\Site(\Delta)$ at arbitrarily small scales gives zero mass to singletons, we have $\mu(\{z\})=0$. Thus, by Lemmas \ref{whitney cube lemma} and \ref{cube integral estimate},
\begin{equation}\begin{split}\label{O-part1}\int_{B^\Delta(0,1/2)}&\frac{1}{|x-z|^{n-2}}\,d\mu(x) \leq \sum_{Q\in\mathcal{W}}
\int_Q \frac{1}{|x-z|^{n-2}}\,d\mu(x) \\&\leq \left(\frac{2}{\xi}\right)^{n-s}(1+\xi \sqrt{n})^{n-2+2(n-s)}\sum_{Q\in\mathcal{W}}\int_Q\frac{1}{|x-z|^{n-2+(n-s)}}\,dx \\
&\leq \left(\frac{2}{\xi}\right)^{n-s}(1+\xi \sqrt{n})^{n-2+2(n-s)}\int_{B(0,1/2+2.5\sqrt{n}\xi )}\frac{1}{|x-z|^{n-2+(n-s)}}\,dx.\end{split}\end{equation}
For brevity, temporarily write $r=(1/2)+2.5\sqrt{n}\xi$ and $\theta=n-2+(n-s)$. By Tonelli's theorem or layer-cake formula, and by monotonicity of volume,
\begin{equation}\begin{split} \label{O-part2}
&\int_{B(0,r)}\lvert x-z\rvert^{-\theta}\,dx
 =\int_{0}^\infty \mathcal{L}^n\left(B(z,t^{-1/\theta})\cap B(0,r)\right)dt\\
&\qquad\leq \int_{0}^\infty \min\left(\mathcal{L}^n(B(z,t^{-1/\theta})),\mathcal{L}^n(B(0,r))\right)\,dt\\
&\qquad=\int_{0}^\infty \mathcal{L}^n\left(B(0,t^{-1/\theta})\cap B(0,r)\right)dt
 =\int_{B(0,r)}\lvert x\rvert^{-\theta}\,dx\\
&\qquad=n\omega_n\int_0^r t^{1-(n-s)}\,dt
 = \frac{n\omega_n}{2-(n-s)}\left(\frac{1}{2}+2.5\sqrt{n}\xi\right)^{2-(n-s)}.
\end{split}\end{equation}
Together \eqref{O-part0}, \eqref{O-part1}, and \eqref{O-part2} imply
\begin{equation*}\begin{split} \sup_{z\in\RR^n} O_z(n,s;\Delta) &\leq \left(\frac{2}{\xi}\right)^{n-s}(1+\xi\sqrt{n})^{n-2+2(n-s)}\frac{n\omega_n}{2-(n-s)}\left(\frac{1}{2}+2.5\sqrt{n}\xi\right)^{2-(n-s)}\\
&=\frac{n\omega_n}{8}\left(\frac{4}{\xi}\right)^{n-s}(1+\xi\sqrt{n})^{n-2+2(n-s)}\frac{1}{1-\frac{n-s}{2}}\left(1+5\sqrt{n}\xi\right)^{2-(n-s)}.\end{split}\end{equation*}
To derive the effective bound, set
\begin{equation}\label{xi-choice}\xi:=\min\left(\frac{1}{3h},\frac{1}{92n^{7.5}}\right)\end{equation}
and split the last display as $O(n,s;\Delta)\leq (n\omega_n/8)\cdot A\cdot B$, where
\begin{equation*} A:=(1+\xi\sqrt{n})^{n-2}\left(1+5\sqrt{n}\xi\right)^{2},\qquad B:=\frac{1}{1-\frac{n-s}{2}}\left(\frac{4(1+\xi\sqrt{n})^{2}}{\xi}\right)^{n-s}.\end{equation*}
Since $(1+0.04n^{-6})(1+0.95n^{-6})=1+0.99n^{-6}+0.038n^{-12}\leq 1+n^{-6}$, to prove \eqref{O-est-2} it suffices to bound $A$ by $1+0.04n^{-6}$ and $B$ by $1+0.95n^{-6}$.

From $1+x\leq e^x$ for all $x\in\RR$ and $n+8\leq 11n/3$, we obtain
\begin{equation*} A\leq \exp\left(\sqrt{n}\xi(n+8)\right)\leq \exp\left(\frac{11n^{1.5}\xi}{3}\right)\leq\exp\left(\frac{11}{276n^{6}}\right).\end{equation*}
Because $\log(1+x)\geq x-x^2/2$ and $n^{-6}\leq 1/729$,
\begin{equation*}
\log(1+0.04n^{-6})
\geq n^{-6}\left(0.04-\frac{0.0008}{729}\right)
>0.0399n^{-6}
>\frac{11}{276n^6}.
\end{equation*}
Therefore, $A\leq 1+0.04n^{-6}$.

Next, $\sqrt{n}\xi\leq (92n^{7})^{-1}<10^{-5}$, so that $4e^{0.51}(1+\xi\sqrt{n})^{2}<6.67$. Since $n-s\leq 0.2/(n^6\log h)<10^{-4}$, we have $1/(1-\frac{n-s}{2})\leq e^{0.51(n-s)}$ and therefore
\begin{equation*}\log B\leq (n-s)\log\frac{4e^{0.51}(1+\xi\sqrt{n})^{2}}{\xi}\leq (n-s)\log\frac{6.67}{\xi}.\end{equation*}
Since $h\geq n^3$ and $n\geq3$, we have $\log n\leq \frac13\log h$ and $\log h\geq 3\log 3>3.21$. Thus, by \eqref{xi-choice},
\[
\begin{split}
\log\frac{6.67}{\xi}
&=\max\left(\log(20.01)+\log h,\ \log(613.64)+7.5\log n\right)\\
&< \max\left(2.997+\log h,\ 6.42+7.5\log n\right)\\
&\leq \max\left(2\log h,\ 2\log h+2.5\log h\right)
=4.5\log h.
\end{split}
\]
Hence $\log B\leq 0.9n^{-6}$, while $\log(1+0.95n^{-6})\geq 0.95n^{-6}(1-0.475n^{-6})>0.949n^{-6}$. Thus, $B\leq 1+0.95n^{-6}$, and we have proved \eqref{O-est-2}.
\end{proof}

We now illustrate how to use the constants $O(n,s;\Delta)$ with other normalizations.

\begin{lemma}\label{scale-O} Let $n\geq 3$ and $h\geq 1$ be integers. Let $x_0\in\RR^n$, let $r>0$, and let $\Delta$ be a dyadic net with initial scale $2rh^{-1}$. Let $s\in(n-1,n]$. If $\nu$ is a Borel measure with $\nu(\RR^n\setminus B^{\Delta}(x_0,r))=0$ and $\nu(Q)\leq (\side Q)^s$ for all $Q\in\Delta$, then
\begin{equation*} \sup_{z\in\RR^n} \int_{B^{\Delta}(x_0,r)} \frac{1}{|x-z|^{n-2}}\,d\nu(x) \leq (2r)^{2-(n-s)}O(n,s;\widetilde\Delta),\end{equation*}
where $\widetilde\Delta$ is the dyadic net with initial scale $h^{-1}$ obtained by applying the homeomorphism $y\mapsto f(y)=(y-x_0)/(2r)$ to $\Delta$.\end{lemma}

\begin{proof} Note that $f$ maps $\Site(\Delta)$ onto $\Site(\widetilde\Delta)$ and $B(x_0,r)$ onto $B(0,1/2)$; hence $f$ maps $B^{\Delta}(x_0,r)$ onto $B^{\widetilde\Delta}(0,1/2)$. Define the image measure $\mu=(2r)^{-s} f_\sharp\nu$ by the rule $\mu(E)=(2r)^{-s}\nu(f^{-1}(E))$ for all Borel sets $E\subset\RR^n$. In particular, $\mu(\RR^n\setminus B^{\widetilde\Delta}(0,1/2))=0$ and $$\mu(Q) = (2r)^{-s} \nu(x_0+2r Q)
\leq (2r)^{-s} (\side(x_0+2rQ))^s = (\side Q)^s$$ for all $Q\in\widetilde\Delta$. Hence $\mu$ is an admissible measure for the constant $O(n,s;\widetilde\Delta)$. Let $w\in\RR^n$, let $z=f(w)$, and define $g(x)=|x-z|^{-(n-2)}$ for all $x\in\RR^n$. By the change-of-variables formula $\int g\,d f_\sharp\nu=\int g\circ f\,d\nu$ for image measures, e.g., see \cite[Theorem 1.19]{Mattila}, we have \begin{equation*}\begin{split} O(n,s;\widetilde\Delta) \geq \int_{B^{\widetilde \Delta}(0,1/2)} \frac{1}{|x-z|^{n-2}}\,d\mu(x)&= \frac{1}{(2r)^s}\int_{B^{\Delta}(x_0,r)} \frac{1}{|f(y)-f(w)|^{n-2}}\,d\nu(y) \\ &= \frac{(2r)^{n-2}}{(2r)^s}\int_{B^{\Delta}(x_0,r)} \frac{1}{|y-w|^{n-2}}\,d\nu(y).\end{split}\end{equation*} Rearranging and taking the supremum over $w\in\RR^n$ completes the proof.\end{proof}

By an abuse of notation, given integers $n\geq 3$ and $h\geq 1$, a point $x_0\in\RR^n$, a number $r>0$, a number $s\in(n-1,n]$, and a dyadic net $\Delta$ with initial scale $2rh^{-1}$, we write $O(n,s;\Delta)$ for $O(n,s;\widetilde\Delta)$, where $\widetilde\Delta$ is the net at initial scale $h^{-1}$ obtained from $\Delta$ as in Lemma \ref{scale-O}. The value so defined depends on the normalization, but our estimates use this constant only through upper bounds that are uniform over nets at the normalized scale. In particular, when $h\geq n^3$ and $n-s\leq 0.2/(n^6\log h)$, Theorem \ref{t:O} shows that \eqref{O-est-2} holds independently of the particular normalized net $\widetilde\Delta$. Thus, under these assumptions, \eqref{O-est-2} applies to $O(n,s;\Delta)$ under any normalization. In Lemma \ref{bourgain-round} below, the data $h$, $x_0$, and $r$ are supplied by the hypotheses.

\subsection{Bourgain's estimate, again}

The bounds on $O(n,s;\Delta)$ in Theorem \ref{t:O} yield an improved effective version of Bourgain's estimate for harmonic measure; compare to \cite[Lemma 1]{Bourgain} and \cite[Lemma 3.1]{15-9s}. Here and below, if $V\subset\RR^n$ is a bounded open set and $x\in V$, then $\omega^x_V$ denotes the exit distribution of Brownian motion started at $x$ and run until it first leaves $V$; equivalently, $\omega^x_V=\omega^x_{V^x}$, where $V^x$ is the connected component of $V$ containing $x$. Note that $V^x$ is open and $\partial V^x\subseteq\partial V$. For nonempty sets $A,B\subset\RR^n$, we write $\gap(A,B):=\inf_{a\in A}\inf_{b\in B}|a-b|$
for the \emph{gap} between $A$ and $B$ (often referred to as the \emph{distance} between $A$ and $B$, despite failing the triangle inequality).

\begin{lemma}[Bourgain's estimate for a round snapshot with a nearly sharp constant] \label{bourgain-round}
Let $n\geq 3$ and $h\geq n^3$ be integers and let $U\subset\RR^n$ be a bounded domain. Suppose that the open ball $B(x_0,r)\subset U$ satisfies $\gap(B(x_0,r),\partial U)=gr$ for some $g>1$. Let $\Delta$ be any dyadic net with initial scale $2r\cdot h^{-1}$ such that $\Site(\Delta)\cap \overline{B(x_0,r)}$ is compact. For all closed $E\subset\RR^n$ with $x_0\not\in E$ and for all $s\leq n$ satisfying $s\geq n-0.2/(n^6\log h)$,
\begin{equation}\label{generic-bourgain}
\omega^{x_0}_{U\setminus E}(E\cap U) \geq \frac{2}{n}\cdot \frac{1-g^{-(n-2)}}{1+n^{-6}}
\cdot \frac{\Net{s}_\infty(E\cap B(x_0,r))}{\omega_n r^s}.
\end{equation}
\end{lemma}

\begin{proof}
Let $E\subset\RR^{n}$ be closed. We may assume that $\Net{s}_\infty(E\cap B(x_0,r))>0$, since otherwise the right hand side of \eqref{generic-bourgain} is zero and the estimate is trivial. Then, under the assumption that $\Site(\Delta)\cap\overline{B(x_0,r)}$ is compact, the set $K:=E\cap \overline{B(x_0,r)}\cap \Site(\Delta)$ is compact and $K\subset\Site(\Delta)$. Let $s$ satisfy the bounds in the hypothesis and note that $s\in(n-1,n]$, since $0.2/(n^6\log h)<1$ for all $n\geq 3$ and $h\geq n^3$. Note that $\partial B(x_0,r)$ has $\Net{s}_\infty$ measure zero, since $\partial B(x_0,r)$ has Hausdorff dimension $n-1$ and $s>n-1$. Hence $\Net{s}_\infty(K)=\Net{s}_\infty(E\cap B(x_0,r))$. By Lemma \ref{frostman}, there exists a finite Borel measure $\mu$ on $\RR^{n}$ with support in $K$ such that $\mu(Q)\leq (\side Q)^s$ for all $Q\in\Delta$ and $\mu(\RR^n)=\mu(K)\geq \Net{s}_\infty(K)$. The measure $\mu$ does not charge $\partial B(x_0,r)$. Indeed, given $\epsilon>0$, pick cubes $Q_1,Q_2,\dots\in\Delta$ that cover $\partial B(x_0,r)\cap\Site(\Delta)$ with $\sum_i (\side Q_i)^s<\epsilon$; since $\mu$ is carried by $K\subset\Site(\Delta)$,
\begin{equation*}
\mu(\partial B(x_0,r))\leq \sum_i \mu(Q_i)\leq \sum_i (\side Q_i)^s<\epsilon.
\end{equation*}
Thus, $\mu(\RR^n\setminus B^{\Delta}(x_0,r))=0$, so that $\mu$ is admissible in Lemma \ref{scale-O}.

Consider the harmonic function $u(z)=\int_K |x-z|^{-(n-2)}\,d\mu(x)$ defined on $\RR^{n}\setminus K$. By Lemma \ref{scale-O}, the inequality $2^{2-(n-s)}\leq 4$, and elementary estimates, we have
\begin{equation}\label{estimate1}
u(z) \leq (2r)^{2-(n-s)}O(n,s;\Delta) \leq 4 \cdot r^{2-(n-s)}O(n,s;\Delta)
\quad \forall\,z\in \RR^n\setminus K,
\end{equation}
\begin{equation}\label{estimate2}
u(x_0)\geq r^{-(n-2)}\mu(K),
\end{equation}
\begin{equation}\label{estimate3}
u(x)\leq (\gap(B(x_0,r),\partial U))^{-(n-2)}\mu(K)=(g r)^{-(n-2)}\mu(K)
\quad \forall\,x\in\partial U.
\end{equation}
The estimates \eqref{estimate2} and \eqref{estimate3} are straightforward. Since $\mu(K)\geq\Net{s}_\infty(K)>0$, we have $u>0$ on $\RR^n\setminus K$; in particular, $\|u\|_{L^\infty(\RR^n\setminus K)}\in(0,\infty)$ by \eqref{estimate1}. Following \cite{Bourgain}, define an auxiliary harmonic function $w$ on $\RR^{n}\setminus K$ by setting
\[
w(x)=\frac{u(x)-\|u\|_{L^\infty(\partial U)}}{\|u\|_{L^\infty(\RR^n\setminus K)}}
\quad\text{for all }x\in\RR^{n}\setminus K,
\]
where the norms denote the supremum of the continuous function $u$ on the specified sets. By design, $w\leq 0$ on $\partial U$, $w$ is continuous on $\partial U$, and $w\leq 1$ on all of $\RR^n\setminus K$. Hence
\[
\limsup_{y\rightarrow x} w(y) \leq \chi_K(x)\quad\text{for all }x\in\partial(U\setminus K).
\]
Thus, $w(x)\leq \omega^{x}_{U\setminus K}(K)$ for all $x\in U\setminus K$ by the maximum principle for harmonic functions.

We next compare this harmonic measure with the one in the statement of the lemma. Since $U\setminus E\subseteq U\setminus K$ and $K\subseteq E$, any Brownian path that reaches $K$ before exiting $U$ must first leave $U\setminus E$ through $E\cap U$. Therefore,
\[
\omega^{x}_{U\setminus K}(K)\leq \omega^x_{U\setminus E}(E\cap U)
\quad\text{for all }x\in U\setminus E.
\]
Using \eqref{estimate1} to estimate the denominator in the definition of $w$ and \eqref{estimate2}, \eqref{estimate3} to estimate the numerator, we obtain
\[
\omega^{x_0}_{U\setminus E}(U\cap E)\geq
\frac{r^{-(n-2)}(1-g^{-(n-2)})\Net{s}_\infty(E\cap B(x_0,r))}
{4\cdot r^{2-(n-s)}O(n,s;\Delta)}.
\]
Applying Theorem \ref{t:O} to bound $O(n,s;\Delta)$ and rearranging yields \eqref{generic-bourgain}.
\end{proof}

\begin{example}
The phrase ``nearly sharp constant'' in the title of Lemma \ref{bourgain-round} refers to a constant obtained by following the skeleton of Bourgain's original proof rather than to an assertion that the constant in \eqref{generic-bourgain} is optimal among all possible proofs. For example, take $x_0=0$, $r=1/2$, and $s=n$, and suppose that $\Delta$ has initial scale $1$ with $\Site(\Delta)\supset\overline{B(0,1/2)}$. For $0<\eta\ll 1/2$, let
\[
C_\eta=\{(x_1,\dots,x_n)\in\RR^n:x_1>-\eta\text{ and }x_2^2+\cdots+x_n^2<\eta^2\},
\qquad
K=\overline{B(0,1/2)}\setminus C_\eta,
\]
and let $\mu$ be Lebesgue measure restricted to $K$. Then $\mu$ is an $n$-dimensional Frostman measure for $K$ relative to $\Delta$, with $\mu(K)=\Net{n}_\infty(K)$. Thus, we can use $\mu$ in Bourgain's argument. If
$u(z)=\int_K |x-z|^{-(n-2)}\,d\mu(x)$,
then
\[
\|u\|_{L^\infty(\RR^n\setminus K)}
\geq u(0)
=\int_{B(0,1/2)\setminus C_\eta}\frac{dx}{|x|^{n-2}}
\longrightarrow \frac{n\omega_n}{8}
\quad\text{as }\eta\downarrow0.
\]
On the other hand, Theorem \ref{t:O} gives $\|u\|_{L^\infty(\RR^n\setminus K)}\leq O(n,n;\Delta)\leq (1+n^{-6})n\omega_n/8$. Hence the estimate for the Frostman potential used in the proof is nearly attained in this example. A different proof of Lemma \ref{bourgain-round} might conceivably yield a better constant.\end{example}

\section{One-ring alternative}\label{sec:parameters}

Recall from the introduction that $d_n$ is the supremum of the upper Hausdorff dimensions of harmonic measure over bounded domains in $\RR^n$. We follow Bourgain's original strategy for bounding $d_n$ \cite{Bourgain}, as updated in \cite{15-9s}, but introduce two new estimates that substantially strengthen the barrier argument. Lemma \ref{alt-2-lemma} exploits repeated returns of Brownian motion to a separating surface, raising the lower bound on the probability of hitting the boundary inside an annular ring. Lemma \ref{alt-2-lemma-P} then allows us to use just one ``thin ring without holes'' in place of Bourgain's original multi-ring barrier. We also use the covers by cubes in general position and the Bourgain estimate for round snapshots with a nearly sharp constant developed above.

\subsection{Overview}\label{ss:overview}

Fix $n\geq 3$. The model depends on the following parameters:
\begin{itemize}
\item integers $M\geq 3$, $L\geq 2$, with $2L\geq 3M$, and $h\geq n^3$, which control the geometry and scales of the cube partitions described in \S\ref{ss:mha};
\item $0<\rho\leq\rho_0$ and $0<a<1$ that determine the size of an $(n-\rho)$-dimensional hole, where the upper bound on $\rho$ is
\begin{equation}\label{rho0-def}
\rho_0:=\frac{0.2}{n^6\log h};
\end{equation} see \eqref{hole-def};
\item a lower bound $\eta$ for harmonic measure in the absence of a hole, given by
\begin{equation}\label{eta-def}
\eta:=\frac{2a}{n+n^{-5}}(1-2^{-(n-2)});
\end{equation}
see Lemma \ref{local-harmonic-measure}.
\end{itemize}

The following theorem is the main result of this section. It converts any admissible set of parameters into an explicit upper bound for $d_n$.

\begin{theorem}\label{t:bb}
Let $n\geq 3$ and let $0<a<1$. Let $M\geq 3$, $L\geq 2$, and $h\geq n^3$ be integers with $2L\geq 3M$ such that
\begin{equation}\label{k-minimum-lower}
W:=\left(1+\frac{n\,\omega_n}{\omega_{n-1}}\left(\frac{4L}{3M}-1\right)\eta\right)^{-1}
<\left(1-\frac{2}{M}\right)^n=:V,
\end{equation}
and let $\rho\in(0,\rho_0]$ satisfy
\begin{equation}\begin{split}\label{rho-condition}
\left(1-\frac{1}{L}\right)^{n-\rho}
+\left(L^n-(L-1)^n-2^{n-1}\right)\left(\frac{1}{L}\right)^{n-\rho}
+2^{n-1}(2^n-1)\left(\frac{1}{2L}\right)^{n-\rho}&\\
+\left(2^{n-1}h^n-\frac{\omega_n}{2^n}\left(h-2\sqrt n\right)^n\right)
\left(\frac{1}{2Lh}\right)^{n-\rho}
+a\omega_n\left(\frac{1}{4L}\right)^{n-\rho}&\leq1.
\end{split}\end{equation}
Then $d_n\leq n-\lambda\rho/(1+\lambda)$,
where $\lambda>0$ is given by
\begin{equation}\label{lambda-def}
\lambda:=W\log_\tau\frac{W}{V}+(1-W)\log_\tau\frac{1-W}{1-V},
\end{equation}

\begin{equation}\label{tau-def}
\tau:=\left(1-\frac{2}{M}\right)^{n-\rho}
+\big(M^n-(M-2)^n\big)\left(\frac{1}{M}\right)^{n-\rho}
\qquad(\tau>1).
\end{equation}
\end{theorem}

After describing the one-ring alternative in \S\ref{ss:mha}, we establish lemmata in \S\ref{ss:lemmata} and record the proof of Theorem \ref{t:bb} in \S\ref{ss:proof:t:bb}. Suitable choices of parameters are then used to prove Theorem \ref{t:big-dim} in \S\ref{s:punchline}.

\subsection{One-ring alternative}\label{ss:mha}
Let $\Omega\subset\RR^n$ be a bounded domain and set $E:=\partial\Omega$. Thus, $E$ is closed and harmonic measure on $\Omega$ is supported on $E$. Fix a cube $P$ in $\RR^n$ of any type. We examine the geometry of $E$ inside $P$ in order to show that $P$ satisfies condition (I) or (II) of Theorem \ref{dim-lemma}. The model uses two related partitions of $P$, one for each alternative. In \textbf{Alternative I}, the partition is adapted to the presence of a hole in $E$. In \textbf{Alternative II}, the partition is used when no such hole is present.
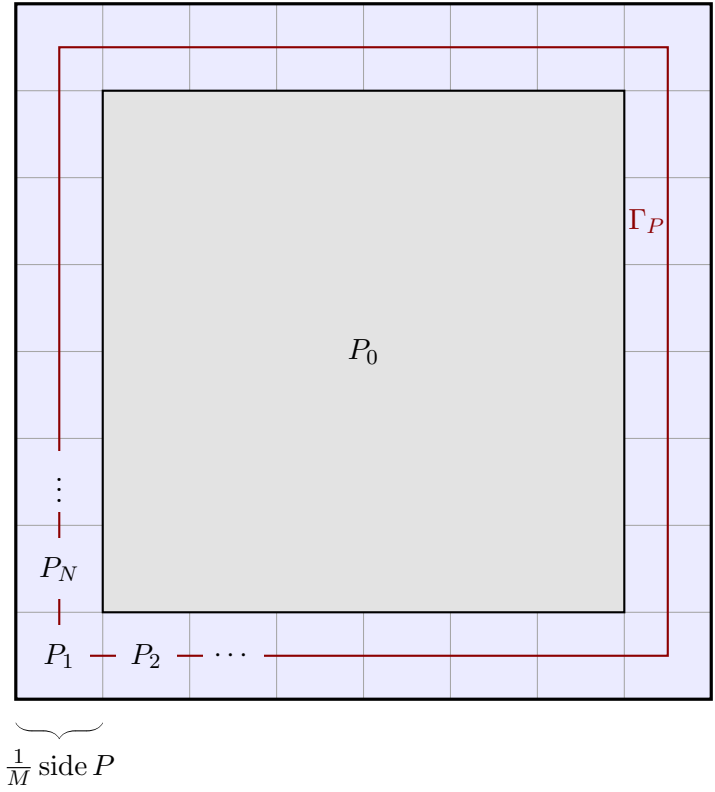
\begin{figure}[t]
\begin{tikzpicture}[
    scale=1.15,
    every node/.style={font=\small},
    >=stealth
  ]
  \def\M{8}          
  \pgfmathtruncatemacro{\Mm}{\M-1}   
  \fill[blue!8] (0,0) rectangle (\M,\M);          
  \fill[gray!22] (1,1) rectangle (\M-1,\M-1);     
  \foreach \i in {1,...,\Mm}{
    \draw[gray!70,thin] (\i,0) -- (\i,1);
    \draw[gray!70,thin] (\i,\M-1) -- (\i,\M);
  }
  \foreach \i in {1,...,\Mm}{
    \draw[gray!70,thin] (0,\i) -- (1,\i);
    \draw[gray!70,thin] (\M-1,\i) -- (\M,\i);
  }
  \draw[thick] (1,1) rectangle (\M-1,\M-1);   
  \draw[very thick] (0,0) rectangle (\M,\M);  

  \draw[red!55!black,thick] (0.5,0.5) rectangle (\M-0.5,\M-0.5);
  \node[red!55!black] at (\M-.74,\M-2.5) {$\Gamma_P$};

  \fill[blue!8] (0.5-0.35,0.5-0.35) rectangle (0.5+0.35,0.5+0.35); 
  \fill[blue!8] (1.5-0.35,0.5-0.35) rectangle (1.5+0.35,0.5+0.35); 
  \fill[blue!8] (2.5-0.35,0.5-0.35) rectangle (2.5+0.35,0.5+0.35); 
  \fill[blue!8] (0.5-0.35,1.5-0.35) rectangle (0.5+0.35,1.5+0.35); 
  \fill[blue!8] (0.5-0.35,2.5-0.35) rectangle (0.5+0.35,2.5+0.35); 

  \node at (0.5,0.5) {$P_1$};
  \node at (1.5,0.5) {$P_2$};
  \node at (2.5,0.5) {$\cdots$};
  \node at (0.5,1.5) {$P_N$};
  \node at (0.5,2.5) {$\vdots$};
  \node at (\M/2,\M/2) {$P_0$};
  \draw[decorate,decoration={calligraphic brace,amplitude=6pt,mirror,raise=4pt}]
        (0,-0.15) -- (1,-0.15)
        node[midway,below=10pt]{$\tfrac{1}{M}\operatorname{side}P$};
%
\end{tikzpicture}
\caption{Illustrated with $n=2$, $M=8$, and $L=12$: the partition used in \textbf{Alternative II}. The outer children form the ring $R_P$ around the inner child $P_0$, and the separating surface $\Gamma_P$ lies inside the ring; for the parameters shown, it lies midway across $R_P$.}
\label{fig:multi-hole}
\end{figure}%
\subsubsection{First partition: the annular ring}

Fix integers $M\geq 3$ and $L\geq 2$ with $2L\geq 3M$ (independent of $P$). Partition $P$ into subcubes (see Figure \ref{fig:multi-hole}): an \emph{inner child} $P_0$ of side length $(1-2/M)\side P$ and $N=M^n-(M-2)^n$ \emph{outer children} $P_1,\dots,P_N$ of side length $(1/M)\side P$, arranged so that $P_0$ contains the center of $P$ and is surrounded by $P_1,\dots,P_N$. As usual, the children may have different types; see the paragraph above Definition \ref{def:cubical-Hausdorff}. The outer children form a \emph{ring} $R_P=\bigcup_{i=1}^N P_i$ of width $(1/M)\side P$ around the inner child.

Define the \emph{separating surface} $\Gamma_P$ to be the boundary of any concentric subcube of $P$ of side length $(1-3/2L)\side P$. Thus, $\Gamma_P$ lies at distance $(3/4L)\side P$ from $\partial P$ and at distance $(1/M-3/4L)\side P$ from $\partial P_0$. Because $2L\geq 3M$, the first distance is at most the second. Hence $\Gamma_P$ lies inside $R_P$, at least as close to $\partial P$ as to $\partial P_0$. For $x\in\Gamma_P$, set
\[
B_{P,x}:=B(x,(1/4L)\side P),
\qquad
\widehat B_{P,x}:=B(x,(3/4L)\side P).
\]
The two balls are concentric and $B_{P,x}\subset \widehat B_{P,x}\subset R_P$.

\subsubsection{Second partition: the L-shaped collar}

Fix $x\in\Gamma_P$. To define a hole at $x$, we form a second partition of $P$ that depends on $x$. See Figure \ref{fig:x-partition}. It is not too important to remember the names of the various classes of cubes below; the essential information is the size and number of each type, which can be looked up as needed.

Choose a face $F$ of $P$ at distance $(3/4L)\side P$ from $x$ and choose any \emph{big cube} $P_{\star,x}\subset P$ of side length $(1-1/L)\side P$ so that half of the faces of $P_{\star,x}$ are contained in the faces of $P$ (i.e.~$P_{\star,x}$ and $P$ share a corner) and so that the face of $P_{\star,x}$ parallel to $F$ lies at distance $(1/L)\side P$ from $F$. In particular, $x\notin P_{\star,x}$.

\begin{figure}[t]
\begin{tikzpicture}[scale=0.77, baseline=(current bounding box.north), every node/.style={font=\small}, >=stealth]


  \fill[gray!15] (1,0) rectangle (12,11);
  \foreach \j in {0,...,11}{
    \pgfmathtruncatemacro{\parity}{Mod(\j,2)}
    \ifnum\parity=0 \fill[blue!14] (0,\j) rectangle (1,\j+1);
    \else \fill[blue!30] (0,\j) rectangle (1,\j+1); \fi
  }
  \foreach \i in {1,...,11}{
    \pgfmathtruncatemacro{\parity}{Mod(\i+11,2)}
    \ifnum\parity=0 \fill[blue!14] (\i,11) rectangle (\i+1,12);
    \else \fill[blue!30] (\i,11) rectangle (\i+1,12); \fi
  }

  \begin{scope}
    \clip (0,0) -- (1,0) -- (1,11) -- (12,11) -- (12,12) -- (0,12) -- cycle;
    \foreach \i in {0,1,...,12}{
      \draw[black,semithick] (\i,0) -- (\i,12);
      \draw[black,semithick] (0,\i) -- (12,\i);
    }
  \end{scope}

  \draw[black!55,thin] (0.5,7) -- (0.5,9);
  \draw[black!55,thin] (0,7.5) -- (1,7.5);
  \draw[black!55,thin] (0,8.5) -- (1,8.5);

  \draw[gray!55,thin] (1.5,1.5) rectangle (10.5,10.5);

  \draw[black,thick] (1,0) -- (1,11) -- (12,11);

  \draw[very thick] (0,0) rectangle (12,12);

  \draw[red!55!black,thick] (0.75,0.75) rectangle (11.25,11.25);

  \fill[orange!12] (0.75,8) circle (0.25);
  \draw[orange!75!black,thick] (0.75,8) circle (0.25);

  \fill[black] (0.75,8) circle (0.045);
  \node[left=2pt] at (0.74,8.3) {$x$};

  \node[left=6pt] at (0,8) {$F$};

  \node at (6.5,5.55) {$P_{\star,x}$};

  \node[align=left] at (2.5,8) {$r=\tfrac{1}{4L}\operatorname{side}P$};

  \node[red!55!black] at (10.85,8) {$\Gamma_{P}$};

  \draw[decorate,decoration={calligraphic brace,amplitude=6pt,mirror,raise=4pt}]
        (0,-0.15) -- (1,-0.15)
        node[midway,below=10pt]{$\tfrac{1}{L}\operatorname{side}P$};
\end{tikzpicture}
\begin{tikzpicture}[scale=2.3, baseline={([yshift=0.12cm]current bounding box.north)}, every node/.style={font=\normalsize}, >=stealth]

  \clip (-0.25,6.75) rectangle (2.25,9.25);

  \fill[gray!15] (1,0) rectangle (12,11);
  \draw[gray!55,thin] (1.5,6.75) -- (1.5,9.25);
  \foreach \j in {0,...,11}{
    \pgfmathtruncatemacro{\parity}{Mod(\j,2)}
    \ifnum\parity=0 \fill[blue!14] (0,\j) rectangle (1,\j+1);
    \else \fill[blue!30] (0,\j) rectangle (1,\j+1); \fi
  }

  \begin{scope}
    \clip (0,0) -- (1,0) -- (1,11) -- (12,11) -- (12,12) -- (0,12) -- cycle;
    \foreach \i in {0,1,...,12}{
      \draw[black,semithick] (\i,0) -- (\i,12);
      \draw[black,semithick] (0,\i) -- (12,\i);
    }
  \end{scope}

  \draw[black!55,thin] (0.5,7) -- (0.5,9);
  \draw[black!55,thin] (0,7.5) -- (1,7.5);
  \draw[black!55,thin] (0,8.5) -- (1,8.5);

  \draw[black,thick] (1,0) -- (1,11) -- (12,11);

  \draw[very thick] (0,0) rectangle (12,12);

  \draw[red!55!black,thick] (0.75,0.75) rectangle (11.25,11.25);

    \fill[blue!60] (0.50,7.50) rectangle (0.55,7.55);
    \fill[blue!29] (0.50,7.55) rectangle (0.55,7.60);
    \fill[blue!60] (0.50,7.60) rectangle (0.55,7.65);
    \fill[blue!29] (0.50,7.65) rectangle (0.55,7.70);
    \fill[blue!60] (0.50,7.70) rectangle (0.55,7.75);
    \fill[blue!29] (0.50,7.75) rectangle (0.55,7.80);
    \fill[blue!60] (0.50,7.80) rectangle (0.55,7.85);
    \fill[blue!29] (0.50,7.85) rectangle (0.55,7.90);
    \fill[blue!60] (0.50,7.90) rectangle (0.55,7.95);
    \fill[blue!29] (0.50,7.95) rectangle (0.55,8.00);
    \fill[blue!60] (0.50,8.00) rectangle (0.55,8.05);
    \fill[blue!29] (0.50,8.05) rectangle (0.55,8.10);
    \fill[blue!60] (0.50,8.10) rectangle (0.55,8.15);
    \fill[blue!29] (0.50,8.15) rectangle (0.55,8.20);
    \fill[blue!60] (0.50,8.20) rectangle (0.55,8.25);
    \fill[blue!29] (0.50,8.25) rectangle (0.55,8.30);
    \fill[blue!60] (0.50,8.30) rectangle (0.55,8.35);
    \fill[blue!29] (0.50,8.35) rectangle (0.55,8.40);
    \fill[blue!60] (0.50,8.40) rectangle (0.55,8.45);
    \fill[blue!29] (0.50,8.45) rectangle (0.55,8.50);
    \fill[blue!29] (0.55,7.50) rectangle (0.60,7.55);
    \fill[blue!60] (0.55,7.55) rectangle (0.60,7.60);
    \fill[blue!29] (0.55,7.60) rectangle (0.60,7.65);
    \fill[blue!60] (0.55,7.65) rectangle (0.60,7.70);
    \fill[blue!29] (0.55,7.70) rectangle (0.60,7.75);
    \fill[blue!60] (0.55,7.75) rectangle (0.60,7.80);
    \fill[blue!29] (0.55,7.80) rectangle (0.60,7.85);
    \fill[blue!60] (0.55,8.15) rectangle (0.60,8.20);
    \fill[blue!29] (0.55,8.20) rectangle (0.60,8.25);
    \fill[blue!60] (0.55,8.25) rectangle (0.60,8.30);
    \fill[blue!29] (0.55,8.30) rectangle (0.60,8.35);
    \fill[blue!60] (0.55,8.35) rectangle (0.60,8.40);
    \fill[blue!29] (0.55,8.40) rectangle (0.60,8.45);
    \fill[blue!60] (0.55,8.45) rectangle (0.60,8.50);
    \fill[blue!60] (0.60,7.50) rectangle (0.65,7.55);
    \fill[blue!29] (0.60,7.55) rectangle (0.65,7.60);
    \fill[blue!60] (0.60,7.60) rectangle (0.65,7.65);
    \fill[blue!29] (0.60,7.65) rectangle (0.65,7.70);
    \fill[blue!60] (0.60,7.70) rectangle (0.65,7.75);
    \fill[blue!29] (0.60,7.75) rectangle (0.65,7.80);
    \fill[blue!60] (0.60,8.20) rectangle (0.65,8.25);
    \fill[blue!29] (0.60,8.25) rectangle (0.65,8.30);
    \fill[blue!60] (0.60,8.30) rectangle (0.65,8.35);
    \fill[blue!29] (0.60,8.35) rectangle (0.65,8.40);
    \fill[blue!60] (0.60,8.40) rectangle (0.65,8.45);
    \fill[blue!29] (0.60,8.45) rectangle (0.65,8.50);
    \fill[blue!29] (0.65,7.50) rectangle (0.70,7.55);
    \fill[blue!60] (0.65,7.55) rectangle (0.70,7.60);
    \fill[blue!29] (0.65,7.60) rectangle (0.70,7.65);
    \fill[blue!60] (0.65,7.65) rectangle (0.70,7.70);
    \fill[blue!29] (0.65,7.70) rectangle (0.70,7.75);
    \fill[blue!60] (0.65,7.75) rectangle (0.70,7.80);
    \fill[blue!29] (0.65,8.20) rectangle (0.70,8.25);
    \fill[blue!60] (0.65,8.25) rectangle (0.70,8.30);
    \fill[blue!29] (0.65,8.30) rectangle (0.70,8.35);
    \fill[blue!60] (0.65,8.35) rectangle (0.70,8.40);
    \fill[blue!29] (0.65,8.40) rectangle (0.70,8.45);
    \fill[blue!60] (0.65,8.45) rectangle (0.70,8.50);
    \fill[blue!60] (0.70,7.50) rectangle (0.75,7.55);
    \fill[blue!29] (0.70,7.55) rectangle (0.75,7.60);
    \fill[blue!60] (0.70,7.60) rectangle (0.75,7.65);
    \fill[blue!29] (0.70,7.65) rectangle (0.75,7.70);
    \fill[blue!60] (0.70,7.70) rectangle (0.75,7.75);
    \fill[blue!29] (0.70,7.75) rectangle (0.75,7.80);
    \fill[blue!60] (0.70,8.20) rectangle (0.75,8.25);
    \fill[blue!29] (0.70,8.25) rectangle (0.75,8.30);
    \fill[blue!60] (0.70,8.30) rectangle (0.75,8.35);
    \fill[blue!29] (0.70,8.35) rectangle (0.75,8.40);
    \fill[blue!60] (0.70,8.40) rectangle (0.75,8.45);
    \fill[blue!29] (0.70,8.45) rectangle (0.75,8.50);
    \fill[blue!29] (0.75,7.50) rectangle (0.80,7.55);
    \fill[blue!60] (0.75,7.55) rectangle (0.80,7.60);
    \fill[blue!29] (0.75,7.60) rectangle (0.80,7.65);
    \fill[blue!60] (0.75,7.65) rectangle (0.80,7.70);
    \fill[blue!29] (0.75,7.70) rectangle (0.80,7.75);
    \fill[blue!60] (0.75,7.75) rectangle (0.80,7.80);
    \fill[blue!29] (0.75,8.20) rectangle (0.80,8.25);
    \fill[blue!60] (0.75,8.25) rectangle (0.80,8.30);
    \fill[blue!29] (0.75,8.30) rectangle (0.80,8.35);
    \fill[blue!60] (0.75,8.35) rectangle (0.80,8.40);
    \fill[blue!29] (0.75,8.40) rectangle (0.80,8.45);
    \fill[blue!60] (0.75,8.45) rectangle (0.80,8.50);
    \fill[blue!60] (0.80,7.50) rectangle (0.85,7.55);
    \fill[blue!29] (0.80,7.55) rectangle (0.85,7.60);
    \fill[blue!60] (0.80,7.60) rectangle (0.85,7.65);
    \fill[blue!29] (0.80,7.65) rectangle (0.85,7.70);
    \fill[blue!60] (0.80,7.70) rectangle (0.85,7.75);
    \fill[blue!29] (0.80,7.75) rectangle (0.85,7.80);
    \fill[blue!60] (0.80,8.20) rectangle (0.85,8.25);
    \fill[blue!29] (0.80,8.25) rectangle (0.85,8.30);
    \fill[blue!60] (0.80,8.30) rectangle (0.85,8.35);
    \fill[blue!29] (0.80,8.35) rectangle (0.85,8.40);
    \fill[blue!60] (0.80,8.40) rectangle (0.85,8.45);
    \fill[blue!29] (0.80,8.45) rectangle (0.85,8.50);
    \fill[blue!29] (0.85,7.50) rectangle (0.90,7.55);
    \fill[blue!60] (0.85,7.55) rectangle (0.90,7.60);
    \fill[blue!29] (0.85,7.60) rectangle (0.90,7.65);
    \fill[blue!60] (0.85,7.65) rectangle (0.90,7.70);
    \fill[blue!29] (0.85,7.70) rectangle (0.90,7.75);
    \fill[blue!60] (0.85,7.75) rectangle (0.90,7.80);
    \fill[blue!29] (0.85,8.20) rectangle (0.90,8.25);
    \fill[blue!60] (0.85,8.25) rectangle (0.90,8.30);
    \fill[blue!29] (0.85,8.30) rectangle (0.90,8.35);
    \fill[blue!60] (0.85,8.35) rectangle (0.90,8.40);
    \fill[blue!29] (0.85,8.40) rectangle (0.90,8.45);
    \fill[blue!60] (0.85,8.45) rectangle (0.90,8.50);
    \fill[blue!60] (0.90,7.50) rectangle (0.95,7.55);
    \fill[blue!29] (0.90,7.55) rectangle (0.95,7.60);
    \fill[blue!60] (0.90,7.60) rectangle (0.95,7.65);
    \fill[blue!29] (0.90,7.65) rectangle (0.95,7.70);
    \fill[blue!60] (0.90,7.70) rectangle (0.95,7.75);
    \fill[blue!29] (0.90,7.75) rectangle (0.95,7.80);
    \fill[blue!60] (0.90,7.80) rectangle (0.95,7.85);
    \fill[blue!29] (0.90,8.15) rectangle (0.95,8.20);
    \fill[blue!60] (0.90,8.20) rectangle (0.95,8.25);
    \fill[blue!29] (0.90,8.25) rectangle (0.95,8.30);
    \fill[blue!60] (0.90,8.30) rectangle (0.95,8.35);
    \fill[blue!29] (0.90,8.35) rectangle (0.95,8.40);
    \fill[blue!60] (0.90,8.40) rectangle (0.95,8.45);
    \fill[blue!29] (0.90,8.45) rectangle (0.95,8.50);
    \fill[blue!29] (0.95,7.50) rectangle (1.00,7.55);
    \fill[blue!60] (0.95,7.55) rectangle (1.00,7.60);
    \fill[blue!29] (0.95,7.60) rectangle (1.00,7.65);
    \fill[blue!60] (0.95,7.65) rectangle (1.00,7.70);
    \fill[blue!29] (0.95,7.70) rectangle (1.00,7.75);
    \fill[blue!60] (0.95,7.75) rectangle (1.00,7.80);
    \fill[blue!29] (0.95,7.80) rectangle (1.00,7.85);
    \fill[blue!60] (0.95,7.85) rectangle (1.00,7.90);
    \fill[blue!29] (0.95,7.90) rectangle (1.00,7.95);
    \fill[blue!60] (0.95,7.95) rectangle (1.00,8.00);
    \fill[blue!29] (0.95,8.00) rectangle (1.00,8.05);
    \fill[blue!60] (0.95,8.05) rectangle (1.00,8.10);
    \fill[blue!29] (0.95,8.10) rectangle (1.00,8.15);
    \fill[blue!60] (0.95,8.15) rectangle (1.00,8.20);
    \fill[blue!29] (0.95,8.20) rectangle (1.00,8.25);
    \fill[blue!60] (0.95,8.25) rectangle (1.00,8.30);
    \fill[blue!29] (0.95,8.30) rectangle (1.00,8.35);
    \fill[blue!60] (0.95,8.35) rectangle (1.00,8.40);
    \fill[blue!29] (0.95,8.40) rectangle (1.00,8.45);
    \fill[blue!60] (0.95,8.45) rectangle (1.00,8.50);
    \fill[red!80!black] (0.55,7.85) rectangle (0.60,7.90);
    \fill[orange!60] (0.55,7.90) rectangle (0.60,7.95);
    \fill[red!80!black] (0.55,7.95) rectangle (0.60,8.00);
    \fill[orange!60] (0.55,8.00) rectangle (0.60,8.05);
    \fill[red!80!black] (0.55,8.05) rectangle (0.60,8.10);
    \fill[orange!60] (0.55,8.10) rectangle (0.60,8.15);
    \fill[red!80!black] (0.60,7.80) rectangle (0.65,7.85);
    \fill[orange!60] (0.60,7.85) rectangle (0.65,7.90);
    \fill[red!80!black] (0.60,7.90) rectangle (0.65,7.95);
    \fill[orange!60] (0.60,7.95) rectangle (0.65,8.00);
    \fill[red!80!black] (0.60,8.00) rectangle (0.65,8.05);
    \fill[orange!60] (0.60,8.05) rectangle (0.65,8.10);
    \fill[red!80!black] (0.60,8.10) rectangle (0.65,8.15);
    \fill[orange!60] (0.60,8.15) rectangle (0.65,8.20);
    \fill[orange!60] (0.65,7.80) rectangle (0.70,7.85);
    \fill[red!80!black] (0.65,7.85) rectangle (0.70,7.90);
    \fill[orange!60] (0.65,7.90) rectangle (0.70,7.95);
    \fill[red!80!black] (0.65,7.95) rectangle (0.70,8.00);
    \fill[orange!60] (0.65,8.00) rectangle (0.70,8.05);
    \fill[red!80!black] (0.65,8.05) rectangle (0.70,8.10);
    \fill[orange!60] (0.65,8.10) rectangle (0.70,8.15);
    \fill[red!80!black] (0.65,8.15) rectangle (0.70,8.20);
    \fill[red!80!black] (0.70,7.80) rectangle (0.75,7.85);
    \fill[orange!60] (0.70,7.85) rectangle (0.75,7.90);
    \fill[red!80!black] (0.70,7.90) rectangle (0.75,7.95);
    \fill[orange!60] (0.70,7.95) rectangle (0.75,8.00);
    \fill[red!80!black] (0.70,8.00) rectangle (0.75,8.05);
    \fill[orange!60] (0.70,8.05) rectangle (0.75,8.10);
    \fill[red!80!black] (0.70,8.10) rectangle (0.75,8.15);
    \fill[orange!60] (0.70,8.15) rectangle (0.75,8.20);
    \fill[orange!60] (0.75,7.80) rectangle (0.80,7.85);
    \fill[red!80!black] (0.75,7.85) rectangle (0.80,7.90);
    \fill[orange!60] (0.75,7.90) rectangle (0.80,7.95);
    \fill[red!80!black] (0.75,7.95) rectangle (0.80,8.00);
    \fill[orange!60] (0.75,8.00) rectangle (0.80,8.05);
    \fill[red!80!black] (0.75,8.05) rectangle (0.80,8.10);
    \fill[orange!60] (0.75,8.10) rectangle (0.80,8.15);
    \fill[red!80!black] (0.75,8.15) rectangle (0.80,8.20);
    \fill[red!80!black] (0.80,7.80) rectangle (0.85,7.85);
    \fill[orange!60] (0.80,7.85) rectangle (0.85,7.90);
    \fill[red!80!black] (0.80,7.90) rectangle (0.85,7.95);
    \fill[orange!60] (0.80,7.95) rectangle (0.85,8.00);
    \fill[red!80!black] (0.80,8.00) rectangle (0.85,8.05);
    \fill[orange!60] (0.80,8.05) rectangle (0.85,8.10);
    \fill[red!80!black] (0.80,8.10) rectangle (0.85,8.15);
    \fill[orange!60] (0.80,8.15) rectangle (0.85,8.20);
    \fill[orange!60] (0.85,7.80) rectangle (0.90,7.85);
    \fill[red!80!black] (0.85,7.85) rectangle (0.90,7.90);
    \fill[orange!60] (0.85,7.90) rectangle (0.90,7.95);
    \fill[red!80!black] (0.85,7.95) rectangle (0.90,8.00);
    \fill[orange!60] (0.85,8.00) rectangle (0.90,8.05);
    \fill[red!80!black] (0.85,8.05) rectangle (0.90,8.10);
    \fill[orange!60] (0.85,8.10) rectangle (0.90,8.15);
    \fill[red!80!black] (0.85,8.15) rectangle (0.90,8.20);
    \fill[orange!60] (0.90,7.85) rectangle (0.95,7.90);
    \fill[red!80!black] (0.90,7.90) rectangle (0.95,7.95);
    \fill[orange!60] (0.90,7.95) rectangle (0.95,8.00);
    \fill[red!80!black] (0.90,8.00) rectangle (0.95,8.05);
    \fill[orange!60] (0.90,8.05) rectangle (0.95,8.10);
    \fill[red!80!black] (0.90,8.10) rectangle (0.95,8.15);
  \begin{scope}
    \clip (0.5,7.5) rectangle (1.0,8.5);
    \foreach \i in {0.5,0.55,...,1.0}{
      \draw[black!45,very thin] (\i,7.5) -- (\i,8.5);
    }
    \foreach \j in {7.5,7.55,...,8.5}{
      \draw[black!45,very thin] (0.5,\j) -- (1.0,\j);
    }
  \end{scope}

  \draw[black!75,thick] (0.5,7.5) rectangle (1.0,8.5);

  \draw[orange!20,thick] (0.75,8) circle (0.25);


  \node[red!70!black] at (1.2,8) {$\Delta_{P,x}$};


\end{tikzpicture}
\caption{Illustrated with $n=2$, $M=8$, $L=12$, and $h=10$: the $x$-dependent partition used in \textbf{Alternative I}. The `L'-shaped collar is divided into medium cubes; only those near $B_{P,x}$ are subdivided into little cubes, and only the little cubes covering $B_{P,x}$ are refined into tiny cubes. The tiny cubes contained in $B_{P,x}$ are the top cubes of $\Delta_{P,x}$.}\label{fig:x-partition}
\end{figure}
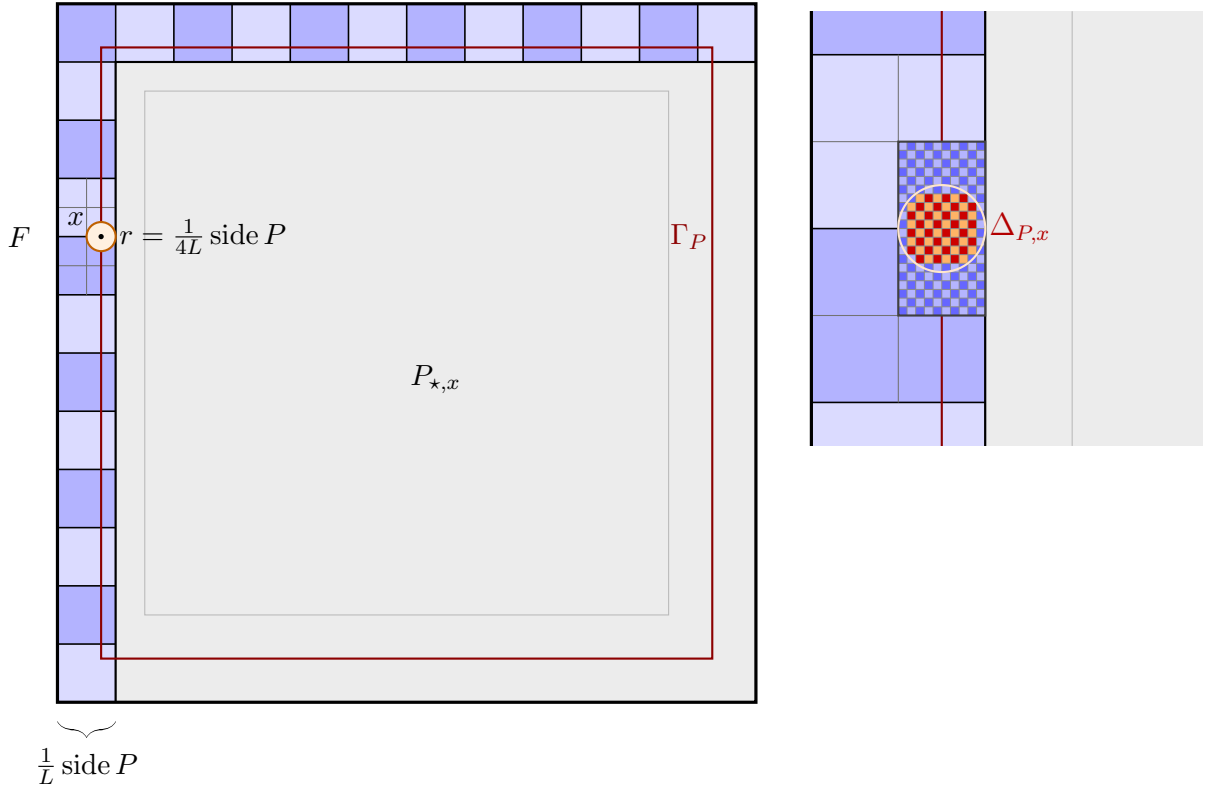

Next, partition the `L'-shaped \emph{collar} $P\setminus P_{\star,x}$ into $L^n-(L-1)^n$ \emph{medium cubes} of side length $(1/L)\side P$. Since $x$ lies at distance $(3/4L)\side P$ from $F$, the points of $B_{P,x}$ lie at distances between $(1/2L)\side P$ and $(1/L)\side P$ from $F$. Thus, $B_{P,x}$ is contained in the union of the medium cubes whose closures meet $F$. These cubes vary only in the $n-1$ coordinate directions parallel to $F$. Since the diameter of $B_{P,x}$ is half the side length of a medium cube, it can meet at most two such cubes in each of these directions, and hence at most $2^{n-1}$ medium cubes in total. Consequently, we may choose a family $\mathcal M_{P,x}$ of exactly $2^{n-1}$ medium cubes whose union contains $B_{P,x}$.

Partition each cube in $\mathcal M_{P,x}$ into $2^n$ \emph{little cubes} of side length $(1/2L)\side P$. Since the points of $B_{P,x}$ lie at distances between $(1/2L)\side P$ and $(1/L)\side P$ from $F$, the ball is contained in the union of the little cubes farthest from $F$ within the cubes of $\mathcal M_{P,x}$. These little cubes vary only in the $n-1$ coordinate directions parallel to $F$. Since the diameter of $B_{P,x}$ equals the side length of a little cube, it can meet at most two such cubes in each of these directions, and hence at most $2^{n-1}$ little cubes in total. Choose a family $\mathcal L_{P,x}$ of exactly $2^{n-1}$ little cubes whose union contains $B_{P,x}$. The cubes in $\mathcal L_{P,x}$ need not belong to the same medium cube.

We now refine the little cubes in $\mathcal L_{P,x}$. Let $h\geq n^3$ be an integer independent of $P$ and $x$, and partition each cube in $\mathcal L_{P,x}$ into \emph{tiny cubes} of side length $h^{-1}\diam B_{P,x}=(1/2Lh)\side P$. Let $\mathcal T_{P,x}$ denote the family of tiny cubes contained in $B_{P,x}$. Because $\overline{B_{P,x}}\subset\interior P$, we may reassign boundaries of the big, medium, little, and tiny cubes without altering the boundary of $P$ so that $\bigcup\mathcal T_{P,x}$ is compact. Choose any dyadic net $\Delta_{P,x}$ with initial scale $(1/2Lh)\side P$ whose top cubes are $\mathcal T_{P,x}$; see Definition \ref{sites-and-nets}. For emphasis, we note that $\Site(\Delta_{P,x})=\bigcup\mathcal T_{P,x}$ is compact.

\subsubsection{Holes}

Let $0<a<1$ be a model parameter independent of $P$ and $x$. We say that there is a \emph{hole} in $P$ at $x\in\Gamma_P$ if
\begin{equation}\label{hole-def}
\Net{n-\rho}_\infty(E)=\Net{n-\rho}_\infty(\partial\Omega\cap \textstyle\bigcup\mathcal T_{P,x})
<a\cdot\omega_n\big((1/4L)\side P\big)^{n-\rho},
\end{equation}
where $\Delta=\Delta_{P,x}$.

We say that $P$ satisfies \textbf{Alternative I} if there is a hole in $P$ at some $x\in\Gamma_P$. The existence of such a hole will allow us to construct an efficient $(n-\rho)$-dimensional cover of $E\cap P$. Otherwise, $P$ satisfies \textbf{Alternative II}: there is no hole in $P$ at any $x\in\Gamma_P$. The absence of holes gives a uniformly positive probability of hitting $E$ locally along the separating surface and ultimately limits how much harmonic measure can reach the inner child $P_0$. Every cube $P$ satisfies exactly one of the two alternatives.

\subsection{Lemmata}\label{ss:lemmata}

\begin{lemma} \label{l:alt-1-cover} Suppose that $P$ satisfies \textbf{Alternative I}. There is a covering of $E\cap P$ by a set $\mathcal{Q}_P$ of pairwise disjoint cubes with $\side Q\leq (1-1/L)\side P$ for all $Q\in\mathcal{Q}_P$ and
\[
\sum_{Q\in\mathcal{Q}_P}(\side Q)^{n-\rho}
\leq \left(\text{left hand side of }\eqref{rho-condition}\right)\cdot(\side P)^{n-\rho}.
\]
\end{lemma}

\begin{proof}
Without loss of generality, we may assume that $\side P=1$. Because $P$ satisfies \textbf{Alternative I}, there exists a hole in $P$ at some $x\in\Gamma_P$. Consider the associated partition of $P$ described in Figure \ref{fig:x-partition}. In the collection $\mathcal{Q}_P$, we include the big cube $P_{\star,x}$, the medium cubes in $P\setminus P_{\star,x}$ that do not belong to $\mathcal{M}_{P,x}$, the little cubes contained in $\bigcup\mathcal{M}_{P,x}\setminus\bigcup\mathcal{L}_{P,x}$, and the tiny cubes in $\mathcal{L}_{P,x}$ that are not contained in $B_{P,x}$. There are $L^n-(L-1)^n-2^{n-1}$ such medium cubes and $2^{n-1}(2^n-1)$ such little cubes. Thus, the big cube, the medium cubes, and the little cubes in $\mathcal{Q}_P$ contribute the first three terms in \eqref{rho-condition} to its $(n-\rho)$-dimensional cubical content. Now, if a tiny cube $T$ is not contained in $B_{P,x}$, then $T$ is contained in $\bigcup\mathcal{L}_{P,x}\setminus B(x,1/(4L)-\sqrt n/(2Lh))$. Dividing the volume of this set by the volume of a tiny cube shows that the number of such tiny cubes is at most $2^{n-1}h^n-(\omega_n/2^n)(h-2\sqrt n)^n$. Hence these tiny cubes contribute at most the fourth term in \eqref{rho-condition} to the $(n-\rho)$-dimensional cubical content of $\mathcal{Q}_P$. Finally, since $\Net{n-\rho}_\infty(E)$ is strictly less than the fifth term in \eqref{rho-condition} by the definition of a hole \eqref{hole-def}, Lemma \ref{pick-disjoint} applied with $s=n-\rho$ furnishes countably many pairwise disjoint cubes contained in $\Site(\Delta_{P,x})=\bigcup\mathcal{T}_{P,x}$ that cover $E\cap\Site(\Delta_{P,x})$ and whose total $(n-\rho)$-dimensional cubical content is at most the fifth term; these are the final cubes that we include in $\mathcal{Q}_P$. The collection $\mathcal{Q}_P$ has the desired properties: its cubes are pairwise disjoint, have side length no greater than that of the big cube, cover $E\cap P$, and satisfy the asserted estimate.
\end{proof}

\begin{lemma}[locally large harmonic measure] \label{local-harmonic-measure} Suppose that $P$ satisfies \textbf{Alternative II}. For all $x\in\Gamma_P\cap \Omega$, we have $\widehat B_{P,x}\subset R_P$ and \begin{equation}\label{eta-est} \omega^x_{\widehat B_{P,x}\setminus E}(E\cap \widehat B_{P,x}) \geq \eta.\end{equation}\end{lemma}

\begin{proof} Let $x\in\Gamma_P\cap\Omega$. The inclusion $\widehat B_{P,x}\subset R_P$ was noted in \S\ref{ss:mha}. Assign $s=n-\rho$, $U=\widehat B_{P,x}$, $r=(1/4L)\side P$, $g=2$, and $\Delta=\Delta_{P,x}$. Then $B(x,r)=B_{P,x}$ and the hypothesis of Lemma \ref{bourgain-round} is satisfied. Indeed, $s$ is in the valid range of exponents, because $\rho\leq \rho_0$. We also have $\gap(B_{P,x},\partial U)=(3/4L)\side P-(1/4L)\side P=2r=g r$. Finally, the cubes in $\Delta$ have initial scale $2r/h$, and $\Site(\Delta_{P,x})\cap \overline{B_{P,x}}$ is compact, because $\Site(\Delta_{P,x})$ is compact. Now, the opposite inequality to \eqref{hole-def} holds, because $P$ satisfies \textbf{Alternative II}. Inserting it into \eqref{generic-bourgain}, we obtain \begin{equation*} \omega^{x}_{U\setminus E}(E\cap U) \geq \frac{2}{n}\cdot \frac{1-2^{-(n-2)}}{1+n^{-6}}
\cdot \frac{\Net{n-\rho}_\infty(E\cap B(x,r))}{\omega_n r^{n-\rho}}\geq \frac{2}{n}\cdot \frac{1-2^{-(n-2)}}{1+n^{-6}}
\cdot \frac{a \omega_n r^{n-\rho}}{\omega_n r^{n-\rho}}=\eta,\end{equation*} with $\eta$ defined as in \eqref{eta-def}.
\end{proof}

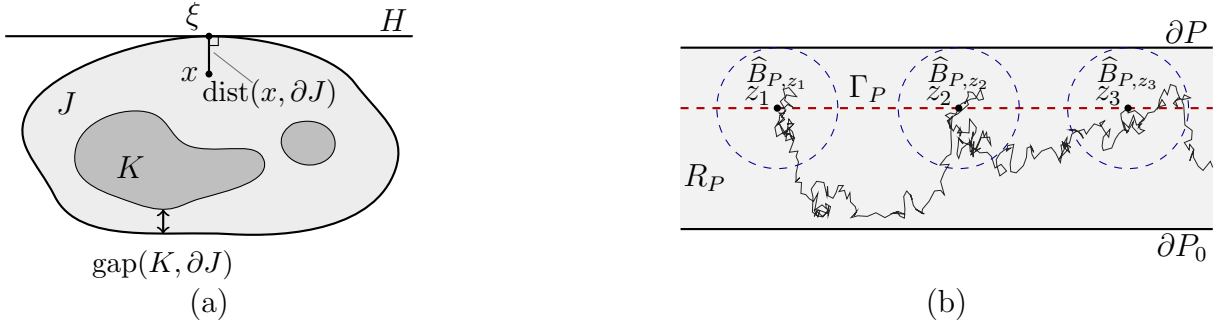
\begin{figure}[t]
\centering
\setlength{\tabcolsep}{0pt}   
\begin{tabular*}{\textwidth}{@{\extracolsep{\fill}}cc@{}}
\begin{tikzpicture}[baseline=(current bounding box.center),scale=0.77,
                    every node/.style={inner sep=1.5pt}]
  \def\Jpath{
    (0,1.7)      .. controls (0.95,1.7)    and (2.128,1.396) .. (2.5,0.9)
                 .. controls (2.908,0.356) and (3.425,0.042) .. (3.2,-0.6)
                 .. controls (2.853,-1.591) and (1.448,-1.766) .. (0.35,-1.7)
                 .. controls (-0.748,-1.634) and (-2.225,-1.9) .. (-2.65,-1.05)
                 .. controls (-2.945,-0.46) and (-3.418,-0.128) .. (-3.1,0.45)
                 .. controls (-2.714,1.151) and (-1.05,1.7)   .. (0,1.7)}
  \fill[black!7] \Jpath;
  \draw[thick] \Jpath;
  \filldraw[fill=black!25,draw=black]
    (-2.3,-0.3) .. controls (-2.25,0.25) and (-1.45,0.52) .. (-1,0.38)
                .. controls (-0.62,0.26) and (-0.68,-0.15) .. (-0.2,-0.22)
                .. controls (0.32,-0.3)  and (0.85,-0.09) .. (0.95,-0.45)
                .. controls (1.03,-0.8)  and (0.35,-0.95) .. (-0.15,-1.09)
                .. controls (-0.5,-1.18) and (-0.62,-1.4) .. (-1.25,-1.18)
                .. controls (-1.75,-1)   and (-2.38,-0.85) .. (-2.3,-0.3);
  \filldraw[fill=black!25,draw=black]
    (1.45,0.2) .. controls (1.9,0.34)  and (2.22,0.08) .. (2.16,-0.22)
               .. controls (2.1,-0.5)  and (1.68,-0.62) .. (1.4,-0.4)
               .. controls (1.18,-0.22) and (1.18,0.04) .. (1.45,0.2);
  \draw[thick] (-3.5,1.7) -- (3.5,1.7) node[above left]{$H$};
  \draw[thin] (0.16,1.7) -- (0.16,1.54) -- (0,1.54);   
  \draw[thick] (0,1.05) -- (0,1.7);
  \fill (0,1.7) circle (1.7pt); \fill (0,1.05) circle (1.7pt);
  \node[anchor=south east] at (-0.06,1.72) {$\xi$};
  \node[anchor=east] at (-0.1,1.05) {$x$};
  \draw[thin,black!55] (0.08,1.42) -- (0.75,0.9);      
  \node at (1.05,0.72) {\small $\operatorname{dist}(x,\partial J)$};
  \draw[<->,thick] (-0.786,-1.275) -- (-0.787,-1.691);
  \node[anchor=north] at (-0.79,-1.9) {\small $\operatorname{gap}(K,\partial J)$};
  \node at (-2.45,0.55) {$J$};
  \node at (-1.35,-0.55) {$K$};
\end{tikzpicture}
&
\begin{tikzpicture}[baseline=(current bounding box.center),scale=0.8,
                    every node/.style={inner sep=1.5pt}]
  \fill[black!5] (0,0) rectangle (8.8,3);
  \draw[thick] (0,3) -- (8.8,3) node[anchor=south east]{$\partial P$};
  \draw[thick] (0,0) -- (8.8,0) node[anchor=north east]{$\partial P_0$};
  \draw[thick,dashed,red!65!black] (0,2) -- (8.8,2);
  \node[anchor=south] at (3.1,2.03) {$\Gamma_P$};
  \node at (0.4,0.85) {$R_P$};
  \draw[black!85,line width=0.35pt]
    (1.60,2.00) -- (1.71,2.00) -- (1.77,2.06) -- (1.85,2.02) -- (1.72,2.13) -- (1.75,2.22) --
    (1.84,2.20) -- (1.92,2.19) -- (1.79,2.34) -- (1.70,2.31) -- (1.73,2.19) -- (1.65,2.11) --
    (1.73,1.98) -- (1.71,1.91) -- (1.63,1.92) -- (1.72,1.81) -- (1.61,1.73) -- (1.71,1.84) --
    (1.67,1.85) -- (1.73,1.95) -- (1.64,1.88) -- (1.63,1.79) -- (1.56,1.68) -- (1.59,1.53) --
    (1.72,1.61) -- (1.80,1.67) -- (1.72,1.58) -- (1.87,1.74) -- (1.76,1.74) -- (1.84,1.75) --
    (1.79,1.82) -- (1.77,1.67) -- (1.74,1.46) -- (1.88,1.51) -- (1.78,1.58) -- (1.70,1.57) --
    (1.64,1.73) -- (1.71,1.62) -- (1.73,1.53) -- (1.80,1.58) -- (1.87,1.45) -- (1.83,1.51) --
    (1.79,1.67) -- (1.85,1.61) -- (1.88,1.50) -- (1.86,1.43) -- (1.80,1.27) -- (1.94,1.29) --
    (1.91,1.20) -- (1.95,1.11) -- (1.91,0.92) -- (1.91,0.79) -- (1.97,0.85) -- (1.90,0.78) --
    (2.00,0.77) -- (1.99,0.59) -- (2.07,0.57) -- (2.13,0.50) -- (2.20,0.46) -- (2.27,0.59) --
    (2.22,0.32) -- (2.30,0.30) -- (2.27,0.20) -- (2.21,0.31) -- (2.34,0.40) -- (2.41,0.34) --
    (2.40,0.42) -- (2.48,0.37) -- (2.44,0.22) -- (2.44,0.30) -- (2.35,0.21) -- (2.25,0.29) --
    (2.21,0.21) -- (2.27,0.32) -- (2.43,0.26) -- (2.45,0.38) -- (2.42,0.28) -- (2.35,0.20) --
    (2.27,0.22) -- (2.18,0.28) -- (2.19,0.53) -- (2.16,0.66) -- (2.25,0.79) -- (2.37,0.77) --
    (2.42,0.55) -- (2.43,0.71) -- (2.44,0.55) -- (2.52,0.51) -- (2.57,0.60) -- (2.61,0.72) --
    (2.69,0.72) -- (2.62,0.62) -- (2.71,0.64) -- (2.75,0.50) -- (2.80,0.60) -- (2.81,0.45) --
    (2.89,0.37) -- (2.91,0.29) -- (2.87,0.22) -- (2.91,0.29) -- (2.82,0.19) -- (2.93,0.22) --
    (3.09,0.21) -- (3.23,0.27) -- (3.38,0.19) -- (3.49,0.25) -- (3.60,0.29) -- (3.69,0.25) --
    (3.50,0.26) -- (3.59,0.29) -- (3.62,0.22) -- (3.70,0.18) -- (3.71,0.33) -- (3.69,0.54) --
    (3.74,0.63) -- (3.80,0.56) -- (3.83,0.49) -- (3.80,0.58) -- (3.85,0.39) -- (3.90,0.32) --
    (4.03,0.25) -- (3.89,0.33) -- (3.99,0.26) -- (4.00,0.33) -- (4.11,0.46) -- (4.20,0.63) --
    (4.30,0.64) -- (4.36,0.57) -- (4.30,0.79) -- (4.36,0.94) -- (4.44,1.07) -- (4.29,1.06) --
    (4.40,1.20) -- (4.47,1.24) -- (4.48,1.38) -- (4.47,1.60) -- (4.49,1.86) -- (4.43,1.78) --
    (4.46,1.66) -- (4.45,1.83) -- (4.38,1.88) -- (4.46,1.92) -- (4.55,1.90) -- (4.60,2.00) --
    (4.64,2.07) -- (4.80,2.15) -- (4.88,2.16) -- (4.96,2.31) -- (5.01,2.15) -- (4.93,2.13) --
    (4.83,2.26) -- (4.75,2.27) -- (4.76,2.20) -- (4.70,2.07) -- (4.73,2.13) -- (4.76,2.06) --
    (4.59,1.91) -- (4.50,1.86) -- (4.41,1.79) -- (4.51,1.62) -- (4.56,1.72) -- (4.65,1.61) --
    (4.62,1.54) -- (4.62,1.67) -- (4.58,1.59) -- (4.66,1.60) -- (4.77,1.59) -- (4.78,1.82) --
    (4.89,1.68) -- (4.88,1.60) -- (4.82,1.51) -- (4.84,1.61) -- (4.80,1.38) -- (4.78,1.47) --
    (4.84,1.37) -- (4.80,1.23) -- (4.96,1.34) -- (4.88,1.29) -- (4.97,1.51) -- (5.11,1.52) --
    (5.13,1.39) -- (5.07,1.49) -- (5.06,1.60) -- (5.02,1.53) -- (5.00,1.62) -- (4.94,1.46) --
    (4.88,1.33) -- (4.79,1.28) -- (4.78,1.44) -- (4.65,1.33) -- (4.60,1.17) -- (4.57,1.24) --
    (4.59,1.37) -- (4.60,1.50) -- (4.67,1.41) -- (4.61,1.49) -- (4.68,1.43) -- (4.81,1.38) --
    (4.84,1.19) -- (4.84,1.26) -- (4.85,1.45) -- (4.84,1.35) -- (4.95,1.26) -- (5.07,1.27) --
    (5.14,1.20) -- (5.20,1.01) -- (5.30,1.07) -- (5.39,0.96) -- (5.28,0.88) -- (5.23,1.05) --
    (5.22,1.15) -- (5.28,1.25) -- (5.37,1.32) -- (5.46,1.34) -- (5.39,1.29) -- (5.41,1.38) --
    (5.51,1.40) -- (5.45,1.08) -- (5.46,0.95) -- (5.55,0.93) -- (5.67,1.05) -- (5.72,1.19) --
    (5.73,1.34) -- (5.80,1.37) -- (5.84,1.27) -- (5.79,1.37) -- (5.81,1.24) -- (5.85,1.16) --
    (5.97,1.29) -- (6.13,1.35) -- (6.20,1.26) -- (6.28,1.10) -- (6.31,1.27) -- (6.29,1.39) --
    (6.39,1.40) -- (6.34,1.34) -- (6.31,1.45) -- (6.36,1.33) -- (6.43,1.27) -- (6.38,1.41) --
    (6.45,1.35) -- (6.38,1.41) -- (6.39,1.51) -- (6.33,1.41) -- (6.42,1.33) -- (6.27,1.17) --
    (6.36,1.35) -- (6.33,1.42) -- (6.43,1.42) -- (6.47,1.57) -- (6.55,1.62) -- (6.43,1.64) --
    (6.52,1.65) -- (6.63,1.57) -- (6.66,1.50) -- (6.58,1.62) -- (6.67,1.65) -- (6.71,1.57) --
    (6.75,1.63) -- (6.84,1.74) -- (6.95,1.70) -- (7.01,1.79) -- (6.93,1.75) -- (7.01,1.62) --
    (7.08,1.52) -- (7.10,1.65) -- (6.99,1.52) -- (6.97,1.37) -- (6.99,1.30) -- (7.13,1.34) --
    (7.19,1.49) -- (7.25,1.58) -- (7.32,1.64) -- (7.41,1.62) -- (7.34,1.65) -- (7.29,1.82) --
    (7.28,1.91) -- (7.40,2.00) -- (7.40,1.92) -- (7.51,1.96) -- (7.60,1.89) -- (7.66,1.83) --
    (7.58,1.74) -- (7.60,1.65) -- (7.67,1.63) -- (7.58,1.57) -- (7.47,1.69) -- (7.49,1.78) --
    (7.38,1.59) -- (7.32,1.77) -- (7.41,1.68) -- (7.49,1.88) -- (7.54,1.72) -- (7.62,1.74) --
    (7.57,1.68) -- (7.59,1.75) -- (7.66,1.80) -- (7.73,1.92) -- (7.83,1.94) -- (7.90,2.02) --
    (8.02,2.05) -- (7.96,2.14) -- (7.90,2.01) -- (7.92,2.14) -- (8.00,2.31) -- (8.05,2.22) --
    (8.10,2.39) -- (8.19,2.36) -- (8.21,2.27) -- (8.26,2.15) -- (8.18,2.12) -- (8.22,2.02) --
    (8.15,2.08) -- (8.23,2.17) -- (8.30,2.14) -- (8.43,2.17) -- (8.33,2.22) -- (8.29,2.12) --
    (8.32,2.05) -- (8.34,1.72) -- (8.46,1.63) -- (8.45,1.48) -- (8.33,1.43) -- (8.41,1.31) --
    (8.42,1.40) -- (8.49,1.31) -- (8.42,1.13) -- (8.45,1.23) -- (8.48,1.30) -- (8.48,1.19) --
    (8.56,1.04) -- (8.64,0.97) -- (8.76,1.12) -- (8.80,1.02);
  \foreach \x/\i in {1.6/1, 4.6/2, 7.4/3}{
    \draw[dashed,blue!55!black] (\x,2) circle (1);
    \fill (\x,2) circle (1.7pt);
    \node[anchor=south east,inner sep=1pt] at (\x-0.05,2.0) {$z_\i$};
    \node at (\x,2.58) {\footnotesize $\widehat B_{P,z_\i}$};
  }
\end{tikzpicture}
\\[2pt]
(a) & (b)
\end{tabular*}
\caption{(a) Lemma \ref{lem:slab}: a supporting hyperplane
$H$ for $J$ at a point $\xi\in\partial J$ nearest to $x\in J\setminus K$.
(b) Lemma \ref{alt-2-lemma}: a Brownian path departing the separating surface from $z_1\in \Gamma_P$ may return to
$\Gamma_P$ at $z_2,z_3,\dots$ without leaving the ring $R_P$, where $z_{i+1}$ is the first point of $\Gamma_P$ met after the path leaves $\widehat B_{P,z_i}$. Under \textbf{Alternative II}, at each $z_i$ there is uniformly large probability that the path meets $\partial\Omega$ inside $\widehat B_{P,z_i}$.}
\label{fig:supporting-and-returns}
\end{figure}

\begin{lemma}[strong Markov property for harmonic measure]\label{lem:strong-markov}
Let $\Omega_1\subset\Omega_2\subset\mathbb{R}^n$ be bounded domains. Then for
every $x\in\Omega_1$ and every Borel set $A\subset\partial\Omega_2$,
\begin{equation}\label{eq:strong-markov}
  \omega^x_{\Omega_2}(A)
  \;=\;\omega^x_{\Omega_1}\bigl(A\cap\partial\Omega_1\bigr)
  \;+\;\int_{\partial\Omega_1\cap\Omega_2}\omega^z_{\Omega_2}(A)
        \,d\omega^x_{\Omega_1}(z).
\end{equation}
We will use the following immediate consequences. If $A,B\subset\partial\Omega_2$
are Borel sets and $c\in\RR$, write
$\varphi(y):=\omega^y_{\Omega_2}(A)-c\,\omega^y_{\Omega_2}(B)$ for
$y\in\Omega_2$. Then, for every $x\in\Omega_1$,
\allowdisplaybreaks
\begin{gather}
\tag{6.9a}\label{sm-bounds}
\omega^x_{\Omega_1}(A\cap\partial\Omega_1)\leq\omega^x_{\Omega_2}(A)
\quad\text{and}\quad
\int_{\partial\Omega_1\cap\Omega_2}\omega^z_{\Omega_2}(A)\,d\omega^x_{\Omega_1}(z)
\leq\omega^x_{\Omega_2}(A),\\[2pt]
\tag{6.9b}\label{sm-null}
\omega^x_{\Omega_2}(A)
=\int_{\partial\Omega_1\cap\Omega_2}\omega^z_{\Omega_2}(A)\,d\omega^x_{\Omega_1}(z)
\quad\text{if }\omega^x_{\Omega_1}(A\cap\partial\Omega_1)=0,\\[2pt]
\tag{6.9c}\label{sm-linear}
\varphi(x)
=\omega^x_{\Omega_1}(A\cap\partial\Omega_1)
-c\,\omega^x_{\Omega_1}(B\cap\partial\Omega_1)
+\int_{\partial\Omega_1\cap\Omega_2}\varphi\,d\omega^x_{\Omega_1},\\[2pt]
\tag{6.9d}\label{sm-linear-ineq}
\varphi(x)
\leq\omega^x_{\Omega_1}(A\cap\partial\Omega_1)
+\int_{\partial\Omega_1\cap\Omega_2}\varphi\,d\omega^x_{\Omega_1}
\quad\text{if }c\geq0.
\end{gather}
\end{lemma}

\begin{proof}
For \eqref{eq:strong-markov}, see e.g.~\cite[Lemma B.1]{AAM19} or
\cite[Theorem 3.8]{dim-caloric}. Since the two terms on the right hand side are
nonnegative, \eqref{sm-bounds} follows, and \eqref{sm-null} is the case in
which the first term vanishes. Applying \eqref{eq:strong-markov} separately
to $A$ and $B$ and subtracting $c$ times the second identity gives
\eqref{sm-linear}. If $c\geq0$, omitting the nonpositive term
$-c\,\omega^x_{\Omega_1}(B\cap\partial\Omega_1)$ gives
\eqref{sm-linear-ineq}.
\end{proof}

\begin{lemma}\label{lem:slab}
Let $J\subset\RR^n$ be a bounded open convex set and let $K\subset J$ be compact. If $U$ is a domain contained in $J\setminus K$, then
\begin{equation}\label{slab-bound}
\omega^x_U(\partial K)\leq \frac{\dist(x,\partial J)}{\gap(K,\partial J)}
\qquad\text{for all }x\in U.
\end{equation}
\end{lemma}

\begin{proof}
Let $S$ be the connected component of $J\setminus K$ containing $U$. Fix
$x\in U$ and choose $\xi\in\partial J$ with
$|x-\xi|=\dist(x,\partial J)$. If $H$ is a hyperplane supporting $J$ at
$\xi$, let $\ell$ be the affine function with $\ell(y)=\dist(y,H)$ for
$y\in J$. Then
$\ell(x)\leq\dist(x,\partial J)$. For $y\in K$, the segment from $y$ to the
nearest point of $H$ ends outside $J$ and therefore meets $\partial J$, so
$\ell(y)\geq\gap(K,\partial J)$. Since $\ell\geq0$ on $\overline J$ and
$\ell$ is harmonic,
\[
\dist(x,\partial J)\geq\ell(x)=\int_{\partial S}\ell\,d\omega^x_S
\geq\gap(K,\partial J)\,\omega^x_S(\partial K).
\]
Moreover, $\partial K\cap\partial U\subset\partial K\cap\partial S$.
The first inequality in \eqref{sm-bounds}, applied to $U\subset S$ with
$A=\partial K\cap\partial S$, gives
$\omega^x_U(\partial K)\leq\omega^x_S(\partial K)$.
\end{proof}

Figure \ref{fig:supporting-and-returns}(b) suggests the probabilistic idea behind the next lemma. Under
\textbf{Alternative II}, at each return to $\Gamma_P$ at $z$, the path has
probability at least $\eta$ of hitting $E$ inside
$\widehat B_{P,z}\subset R_P$. Carrying out this idea gives the constant
$\theta$ below. The proof is self-contained, uses only harmonic measure, and
encodes all returns in one supremum.

\begin{lemma}\label{alt-2-lemma}
Suppose that $P$ satisfies \textbf{Alternative II}. If
$x_0\in\Omega\setminus P$ and $\omega\equiv\omega^{x_0}_\Omega$, then
\begin{equation}\label{k-g-estimate}
\omega(P_0)\leq \theta\,\omega(R_P),
\qquad
\theta:=\left(
\frac{n\,\omega_n}{\omega_{n-1}}
\left(\frac{4L}{3M}-1\right)\eta
\right)^{-1}.
\end{equation}
\end{lemma}

\begin{proof}
Let $S$ be the closed cube with $\partial S=\Gamma_P$. Set
\[
r:=\frac{3}{4L}\side P,\qquad
 d:=\left(\frac1M-\frac{3}{4L}\right)\side P,\qquad
\widehat B_{P,z}=B(z,r),\qquad
\gap(\overline{P_0},\Gamma_P)=d.
\]
Then $r\leq d$ by $2L\geq3M$, and
$\theta\eta=(r/d)\omega_{n-1}/(n\omega_n)$. For $y\in\Omega$, set
\[
\varphi(y):=\omega^y_\Omega(P_0)-\theta\,\omega^y_\Omega(R_P).
\]
Since harmonic measure is supported on $E$, $\varphi$ is the function in
\eqref{sm-linear} and \eqref{sm-linear-ineq} with $A=E\cap P_0$,
$B=E\cap R_P$, and $c=\theta$.
In particular, $\varphi\leq1$. If $\Gamma_P\cap\Omega=\emptyset$,
connectedness and $x_0\notin S$ force $\Omega\cap S=\emptyset$, so
$\omega(P_0)=0$. Thus, assume that $\Gamma_P\cap\Omega\neq\emptyset$ and
write $\Gamma:=\Gamma_P\cap\Omega$. It suffices to prove
$\sup_\Gamma\varphi\leq0$. Indeed, let $\Omega'$ be the component of
$\Omega\setminus S$ containing $x_0$. Then $\partial\Omega'\cap\Omega\subset
\Gamma$ and $(E\cap P_0)\cap\partial\Omega'=\emptyset$. If
$\sup_\Gamma\varphi\leq0$, then
\[
\omega^y_\Omega(P_0)\leq\theta\,\omega^y_\Omega(R_P)
\qquad\text{for all }y\in\partial\Omega'\cap\Omega.
\]
Hence \eqref{sm-null} and \eqref{sm-bounds} give
\[
\omega(P_0)
=\int_{\partial\Omega'\cap\Omega}\omega^y_\Omega(P_0)
\,d\omega^{x_0}_{\Omega'}(y)
\leq\theta\int_{\partial\Omega'\cap\Omega}\omega^y_\Omega(R_P)
\,d\omega^{x_0}_{\Omega'}(y)
\leq\theta\,\omega(R_P).
\]

Set $C=\max\{0,\sup_\Gamma\varphi\}$. Since $\varphi\leq1$, we have
$0\leq C\leq1$. To complete the proof, it remains to show that $C=0$.

First let $y\in\Omega\setminus S$, and let $U$ be the component of
$\Omega\setminus S$ containing $y$. Then $\partial U\cap\Omega\subset\Gamma$
and $\partial U\cap\overline{P_0}=\emptyset$. By \eqref{sm-linear},
\[
\varphi(y)=-\theta\,\omega^y_U((E\cap R_P)\cap\partial U)
+\int_{\partial U\cap\Omega}\varphi\,d\omega^y_U\leq C.
\]
Next let $y\in\Omega\cap\interior S\setminus\overline{P_0}$, and let $V$ be
the component of $\Omega\cap\interior S\setminus\overline{P_0}$ containing
$y$. Here $\partial V\cap\Omega\subset\Gamma\cup(\partial P_0\cap\Omega)$.
The two contributions on $\partial P_0$---the first term in
\eqref{sm-linear-ineq} and the part of the integral over
$\partial V\cap\Omega$---are supported on disjoint subsets of $\partial P_0$.
Since $\varphi\leq C$ on $\Gamma$ and $\varphi\leq1$ on $\Omega$,
\eqref{sm-linear-ineq} gives
$\varphi(y)\leq C+\omega^y_V(\partial P_0)
\leq C+\dist(y,\Gamma_P)/d$, where the last inequality is Lemma
\ref{lem:slab} with $J=\interior S$ and $K=\overline{P_0}$.

Fix $z\in\Gamma$. Let $D$ be the component of
$\widehat B_{P,z}\cap\Omega$ containing $z$, equivalently, the component of
$\widehat B_{P,z}\setminus E$ containing $z$, and let $e$ be an inward unit
normal to a face of $S$ containing $z$. The supporting hyperplane of that face
gives $\dist(y,\Gamma_P)\leq\langle y-z,e\rangle$ for
$y\in\interior S$. The preceding estimates therefore give
$\varphi(y)\leq C+\langle y-z,e\rangle_+/d$ on $\partial D\cap\Omega$, except
possibly at the point of tangency with $\overline{P_0}$, which has
$\omega^z_D$-measure zero.

Let $\nu=\omega^z_D|_{\partial D\cap\Omega}$ and let
$\sigma_z=\omega^z_{\widehat B_{P,z}}$, the normalized surface measure on
$\partial\widehat B_{P,z}$. Since
$\partial D\cap\Omega\subset\partial\widehat B_{P,z}$, the first inequality
in \eqref{sm-bounds}, applied to $D\subset\widehat B_{P,z}$, gives
$\nu\leq\sigma_z$. By \eqref{eta-est},
$\omega^z_D(E\cap\widehat B_{P,z})\geq\eta$, and therefore
$\nu(\partial D\cap\Omega)\leq1-\eta$. Since
$\widehat B_{P,z}\subset R_P$, we also have
$\omega^z_D((E\cap R_P)\cap\partial D)\geq\eta$, while
$\omega^z_D((E\cap P_0)\cap\partial D)=0$. Applying \eqref{sm-linear} to
$D\subset\Omega$ now gives
\[
\begin{split}
\varphi(z)
&\leq-\theta\eta+\int\varphi\,d\nu
\leq-\theta\eta+C(1-\eta)
+\frac1d\int\langle y-z,e\rangle_+\,d\sigma_z(y)\\
&=-\theta\eta+C(1-\eta)
+\frac rd\frac{\omega_{n-1}}{n\omega_n}
=(1-\eta)C.
\end{split}
\]
Here
$\int\langle y-z,e\rangle_+\,d\sigma_z(y)
=r\omega_{n-1}/(n\omega_n)$, and the last equality uses the identity for
$\theta\eta$ above. Taking the supremum over $z\in\Gamma$ and then the maximum with $0$ gives $C\leq(1-\eta)C$. Since $\eta>0$, we have
$C=0$, as required.
\end{proof}

\begin{lemma}\label{alt-2-lemma-P}
Suppose that $P$ satisfies \textbf{Alternative II}. If
$x_0\in\Omega\setminus P$ and $\omega\equiv\omega^{x_0}_\Omega$, then
\begin{equation}\label{k-g-estimate-P}
\omega(P_0)\leq W\,\omega(P),
\qquad
W=\left(
1+\frac{n\,\omega_n}{\omega_{n-1}}
\left(\frac{4L}{3M}-1\right)\eta
\right)^{-1}.
\end{equation}
\end{lemma}

\noindent\emph{Remark.} In \cite{Bourgain} and \cite{15-9s}, this
estimate appeared as
$\omega(P_0)\leq\eta^{-1}(1-\eta)^k\omega(P)$, where $k$ is the
number of outer rings. Here the large factor $\eta^{-1}$ is
eliminated, and $W<1$ decreases as $\eta$ increases and as
$\Gamma_P$ moves toward $\partial P$. This allows us to prove
$\sup\{\omega(P_0)/\omega(P):P\text{ satisfies \textbf{Alternative II}}\}<1$
with $k=1$. Using one ring instead of several also allows a larger
choice of $\rho$ in \eqref{rho-condition}.

\begin{proof}
By Lemma \ref{alt-2-lemma},
$\omega(P)=\omega(R_P)+\omega(P_0)\geq
(1+\theta^{-1})\omega(P_0)$. Therefore,
$\omega(P_0)\leq(1+\theta^{-1})^{-1}\omega(P)=W\omega(P)$.
\end{proof}

\begin{lemma}\label{alt-2-bound}
Suppose $P$ satisfies \textbf{Alternative II}. Let $V$ and $W$ be
defined as in \eqref{k-minimum-lower}, so that
$V=(\vol P_0)/(\vol P)$. If $W<V$, then for every exponent
$\alpha\in(0,1)$, the quantity $\gamma_\alpha:=
W^\alpha V^{1-\alpha}
+(1-W)^\alpha(1-V)^{1-\alpha}$ satisfies $\gamma_\alpha<1$ and
\begin{equation}\label{apply-CS}
\sum_{i=0}^{M^n-(M-2)^n}
\omega(P_i)^\alpha(\vol P_i)^{1-\alpha}
\leq
\gamma_\alpha\,\omega(P)^\alpha(\vol P)^{1-\alpha}.
\end{equation}
\end{lemma}

\begin{proof}
We may assume that $\omega(P)>0$, since \eqref{apply-CS} is trivial
otherwise. Fix $\alpha\in(0,1)$ and set
$H(x)=x^\alpha V^{1-\alpha}+(1-x)^\alpha(1-V)^{1-\alpha}$ for
$0\leq x\leq1$. We have $H(V)=1$ and $H'(x)>0$ for $0<x<V$, so
$H$ is strictly increasing on $[0,V]$. Put
$N=M^n-(M-2)^n$. By H\"older's inequality,
\[
\begin{split}
\sum_{i=0}^{N}\omega(P_i)^\alpha(\vol P_i)^{1-\alpha}
&\leq
\omega(P_0)^\alpha(\vol P_0)^{1-\alpha}
+\left(\sum_{i=1}^{N}\omega(P_i)\right)^\alpha
 \left(\sum_{i=1}^{N}\vol P_i\right)^{1-\alpha}\\
&=
H\left(\frac{\omega(P_0)}{\omega(P)}\right)
\omega(P)^\alpha(\vol P)^{1-\alpha}.
\end{split}
\]
By Lemma \ref{alt-2-lemma-P},
$\omega(P_0)/\omega(P)\leq W<V$. Therefore,
$H(\omega(P_0)/\omega(P))\leq H(W)=\gamma_\alpha<H(V)=1$, which
proves \eqref{apply-CS}.
\end{proof}

\subsection{Proof of Theorem \ref{t:bb}}\label{ss:proof:t:bb}

Let $n\geq3$ and $0<a<1$. Define $\eta$ by \eqref{eta-def}. Fix
integers $M\geq3$, $L\geq2$, and $h\geq n^3$ with $2L\geq3M$ and
$W<V$, where $W$ and $V$ are defined in \eqref{k-minimum-lower}.
Define $\rho_0$ by \eqref{rho0-def}, and fix
$\rho\in(0,\rho_0]$ satisfying \eqref{rho-condition}. For
$\alpha\in(0,1)$, define $\gamma_\alpha$ as in
Lemma \ref{alt-2-bound}, and let $\lambda$ and $\tau$ be given as in \eqref{lambda-def} and \eqref{tau-def}.

Let $\Omega\subset\RR^n$ be a bounded domain. By translation and
dilation, we may assume that $0\in\Omega$ and
$\dist(0,\partial\Omega)>\sqrt n$. Thus, every cube of side length at
most $1$ that meets $\partial\Omega$ lies away from the pole $0$.
Set $\omega=\omega^0_\Omega$, $E=\partial\Omega$, and
$\beta=1-1/L$. Then $E$ is closed,
$\omega(\RR^n\setminus E)=0$, and $\beta<1$. Fix a cube $P$ of side
length at most $1$.

Suppose first that $P$ satisfies \textbf{Alternative I}. By
Lemma \ref{l:alt-1-cover} and \eqref{rho-condition}, there is a
pairwise disjoint cover $\mathcal Q$ of $E\cap P$ by cubes
$Q$ with $\side Q\leq\beta\side P$ and
$\sum_{Q\in\mathcal Q}(\side Q)^{n-\rho}
\leq(\side P)^{n-\rho}$. Therefore, $P$ satisfies condition (I) of
Theorem \ref{dim-lemma}.

Suppose instead that $P$ satisfies \textbf{Alternative II}. If
$\omega(P)=0$, then condition (II) of Theorem \ref{dim-lemma} holds
trivially. Assume that \(\omega(P)>0\). Since \(\omega\) is supported on \(E\), the cube \(P\) meets \(\partial\Omega\), so the pole \(0\) lies outside \(P\) by our normalization.
Let $P_0$ be the inner child and
$P_1,\dots,P_N$, $N=M^n-(M-2)^n$, be the outer children. Then
$\mathcal R=\{P_0,\dots,P_N\}$ is a pairwise disjoint cover of
$E\cap P$. Every cube in $\mathcal R$ has side length at most
$\side P_0=(1-2/M)\side P\leq\beta\side P$, while its
$(n-\rho)$-dimensional cubical content is
$\tau(\side P)^{n-\rho}$. Thus, $\tau>1$ by
Principle \ref{p:bigger}. Finally, Lemma \ref{alt-2-bound} gives
$\sum_{R\in\mathcal R}\omega(R)^\alpha(\vol R)^{1-\alpha}
\leq\gamma_\alpha\omega(P)^\alpha(\vol P)^{1-\alpha}$.
Therefore, $P$ satisfies condition (II) of Theorem \ref{dim-lemma}
with exponent $\alpha$.

Fix $\alpha\in(0,1)$. We have shown that every cube of side length at
most $1$ satisfies condition (I) or (II) of Theorem
\ref{dim-lemma}. Hence
$\overline{\dim}_H\omega\leq
n-\lambda_\alpha\rho/(1+\lambda_\alpha)$, where
$\lambda_\alpha=-\log_\tau(\gamma_\alpha)/(1-\alpha)$. It remains to
let $\alpha\to1$. Put $f(\alpha)=\log\gamma_\alpha$. The function
$f$ is convex on $[0,1]$, since it is the logarithm of a sum of
exponentials of affine functions of $\alpha$, and $f(1)=0$.
Therefore, the slope of the secant line joining \((\alpha,f(\alpha))\) to \((1,f(1))\), namely \([f(1)-f(\alpha)]/(1-\alpha)\), is nondecreasing as \(\alpha\to1\) and converges to \(f'(1)\).
Direct differentiation gives
$f'(1)=W\log(W/V)+(1-W)\log((1-W)/(1-V))$. Thus,
$\lambda_\alpha$ increases to $\lambda$ defined in
\eqref{lambda-def}. Since $x\mapsto x/(1+x)$ is continuous and
increasing, letting $\alpha\to1$ yields
$\overline{\dim}_H\omega\leq n-\lambda\rho/(1+\lambda)$. Taking the
supremum over bounded domains gives
$d_n\leq n-\lambda\rho/(1+\lambda)$. This concludes the proof of Theorem \ref{t:bb}.

\section{Lower bounds on \texorpdfstring{Bourgain's constant $b_n$}{Bourgain's constant}}\label{s:punchline}

We now derive explicit lower bounds for Bourgain's constant $b_n=n-d_n$ from
Theorem \ref{t:bb}. We first treat $n\geq 6$ uniformly and then bound $b_3$, $b_4$, and $b_5$ directly.

\begin{theorem}\label{t:sharp-4}
For all $n\geq 6$,
\begin{equation}\label{e:high-dimensional-bound}
b_n\geq \left(\frac{1}{n-1}\right)^{2n+6.12\sqrt n}.
\end{equation}
\end{theorem}

\begin{proof}
Fix $n\geq6$. Set $a=n/(n+1)$ and $h=n^6$, define $\eta$ by
\eqref{eta-def}, and put $q:=\omega_{n-1}/(\omega_n\eta)$,
$m:=n+\lceil q\rceil+1$, $M:=2m$, and $L:=3m$. Then
$2L=3M$, $W=1-n/(n+q)$, and $V=(1-1/m)^n$.

\emph{Step 1: $W<V$.}
Lemma \ref{ball-volume} gives
\[
q\leq\frac{n^{3/2}}{2\sqrt{2\pi}}\,
e^{1/(6n+12)}
\left(1+\frac2n\right)^{1/2}
\left(1+\frac1n\right)
\frac{1+n^{-6}}{1-2^{2-n}}
<0.3\,n^{3/2},
\]
where the last inequality follows by bounding the $n$-dependent
factors by their values at $n=6$. Hence
$m<n+0.3n^{3/2}+2<n^{3/2}$. Since $m>n+q$, Bernoulli's
inequality gives $V\geq1-n/m>W$. Moreover, since
$m-n-q\geq1$ and $n+q<m<n^{3/2}$,
\[
V-W\geq\frac{n(m-n-q)}{m(n+q)}
>\frac{n}{m^2}>\frac1{n^2}.
\]

\emph{Step 2: an admissible $\rho$.}
Let $\Sigma(\rho)$ denote the left hand side of
\eqref{rho-condition}. We estimate $\Sigma(\rho)$ by linearizing at
$\rho=0$. Direct substitution and Bernoulli's inequality give
\[
1-\Sigma(0)=\frac{\omega_n}{(4L)^n}
\left[\left(1-\frac{2\sqrt n}{h}\right)^n-a\right]
\geq\frac{\omega_n}{2(n+1)(4L)^n}.
\]
Indeed, the expression in brackets is at least
$1/(n+1)-2n^{-9/2}\geq1/[2(n+1)]$.

At $\rho=0$, the first two terms of $\Sigma$ sum to
$1-2^{n-1}/L^n$. Since $\Sigma(0)<1$, the last three terms sum to
less than $2^{n-1}/L^n$. Differentiating \eqref{rho-condition} termwise gives
\[
\begin{aligned}
\Sigma'(t)={}&
\underbrace{\log\frac{L}{L-1}
\left(1-\frac1L\right)^{n-t}}_{\mathrm{I}}
+\underbrace{\log L\,
\bigl(L^n-(L-1)^n-2^{n-1}\bigr)L^{t-n}}_{\mathrm{II}}\\
&+\boxed{\begin{aligned}
&\log(2L)\,2^{n-1}(2^n-1)(2L)^{t-n}\\
&+\log(2Lh)
\left(2^{n-1}h^n-\frac{\omega_n}{2^n}(h-2\sqrt n)^n\right)
(2Lh)^{t-n}
+a\omega_n\log(4L)(4L)^{t-n}
\end{aligned}}_{\,\mathrm{III}}.
\end{aligned}
\]For $0\leq t\leq\rho_0$, we have
$\mathrm{I}\leq2/L$, while Bernoulli's inequality gives
$\mathrm{II}\leq(2Lh)^t(n/L)\log L$. At $t=0$, the three terms
comprising $\mathrm{III}$ have total coefficient less than
$2^{n-1}/L^n$; since their bases are at most $2Lh$, it follows that
\[
\mathrm{III}\leq
(2Lh)^t\frac{2^{n-1}}{L^n}\log(2Lh).
\]
Therefore,
\[
\Sigma'(t)\leq(2Lh)^t
\left[
\frac2L+\frac nL\log L
+\frac{2^{n-1}}{L^n}\log(2Lh)
\right]<\frac n2.
\]

 Set
\begin{equation}\label{rho-choice}
\rho:=\frac{\omega_n}{2n^2(4L)^n}.
\end{equation}
Since $\omega_n\leq2^n$ and $L\geq3n$, we have
$\rho\leq1/[2n^2(6n)^n]<1/(30n^6\log n)=\rho_0$. By the mean value
theorem,
\[
\Sigma(\rho)-\Sigma(0)
<\frac n2\rho
=\frac{\omega_n}{4n(4L)^n}
\leq\frac{\omega_n}{2(n+1)(4L)^n}
\leq1-\Sigma(0).
\]
Thus, \eqref{rho-condition} holds.

\emph{Step 3: a lower bound for $\lambda$.}
Put
$\Phi(x):=x\log(x/V)+(1-x)\log((1-x)/(1-V))$. Then
$\Phi(V)=\Phi'(V)=0$ and $\Phi''(x)=1/[x(1-x)]\geq4$, so Taylor's
theorem and Step 1 give $\Phi(W)\geq2(W-V)^2>1/n^4$. Also,
$\tau\leq(2m)^\rho$, and since $2m<n^2$ and $\rho\leq\rho_0$,
we have $\log\tau\leq\rho_0\log(n^2)=1/(15n^6)<1/n^4$.
By \eqref{lambda-def}, $\lambda>1$, and hence
$\lambda/(1+\lambda)>1/2$.

\emph{Step 4: comparison with \eqref{e:high-dimensional-bound}.}
By Theorem \ref{t:bb} and \eqref{rho-choice},
$b_n>\rho/2=\omega_n/[4n^2(4L)^n]$. Lemma \ref{ball-volume} gives
\[
\omega_n\geq
\frac{e^{1-1/(6n+12)}}{\sqrt{\pi(n+2)}}
\left(\frac{\sqrt{2\pi e}}{\sqrt{n+2}}\right)^n
>\frac1{n^2}\left(\frac4{\sqrt{n+2}}\right)^n.
\]
Here we used $\sqrt{2\pi e}>4$ and
$n^2>\sqrt{\pi(n+2)}$. Consequently,
\[
b_n>\frac1{4n^4}
\left(\frac1{L\sqrt{n+2}}\right)^n.
\]
Since $L=3m<n^{3/2}+3n+6$, put
$A:=\sqrt{1+2/n}\,(1+3/\sqrt n+6n^{-3/2})$. Then
$L\sqrt{n+2}<n^2A$, while
$\sqrt n\log A\leq1/\sqrt n+3+6/n$. It therefore suffices to show
\[
6.12\log(n-1)>
\frac{2\sqrt n}{n-1}
+\frac1{\sqrt n}+3+\frac6n
+\frac{4\log n+\log4}{\sqrt n}.
\]
For $n\geq6$, the left hand side is greater than
$6.12\log5>9.8$. Using $\log n/\sqrt n\leq2/e$, the right hand
side is at most
$2\sqrt6/5+1/\sqrt6+4+8/e+\log4/\sqrt6<9$.
This proves \eqref{e:high-dimensional-bound}.
\end{proof}

\begin{theorem}\label{t:sharp-small}
We have
\[
b_3\geq1.0118\times10^{-5},\qquad
b_4\geq1.206\times10^{-7},\qquad
b_5\geq8.006\times10^{-10}.
\]
In particular, the upper Hausdorff dimension of harmonic measure on any
domain in $\RR^3$ is less than $2.99999$.
\end{theorem}

\begin{proof}
Find a calculator and apply Theorem \ref{t:bb} with
\[
\begin{array}{c|ccccc}
n&M&L&a&h&\rho\\ \hline
3&9&17&117/200&10^5&1.05628\times10^{-5}\\
4&10&17&13/20&10^5&1.22747\times10^{-7}\\
5&11&19&143/200&10^6&8.0084\times10^{-10}.
\end{array}
\]
In each row $2L>3M$. Let $\Sigma(\rho)$ denote the left hand side of
\eqref{rho-condition}. By direct evaluation,
\[
\begin{array}{c|ccccc}
n&W&V&\rho_0&1-\Sigma(\rho)&\lambda\rho/(1+\lambda)\\ \hline
3&0.4581\ldots&0.4705\ldots&
2.3829\ldots\times10^{-5}&
1.7678\ldots\times10^{-11}&
1.01183\ldots\times10^{-5}\\
4&0.4073\ldots&0.4096&
4.2411\ldots\times10^{-6}&
6.6680\ldots\times10^{-13}&
1.206\ldots\times10^{-7}\\
5&0.3650\ldots&0.3666\ldots&
9.2649\ldots\times10^{-7}&
7.6165\ldots\times10^{-15}&
8.006\ldots\times10^{-10}.
\end{array}
\]
Thus, $W<V$, $\rho<\rho_0$, and \eqref{rho-condition} holds in each
case. The conclusion follows.
\end{proof}

\begin{proof}[Proof of Theorem \ref{t:big-dim}]
When $n=2$, both sides of \eqref{e:main-bound} equal $1$ by the
Jones--Wolff theorem. For $n=3,4,5$, Theorem \ref{t:sharp-small}
gives, respectively,
\[
\begin{aligned}
b_3&\geq1.0118\times10^{-5}
   >1.006\ldots\times10^{-5}
   =2^{-(6+6.12\sqrt3)},\\
b_4&\geq1.206\times10^{-7}
   >2.203\ldots\times10^{-10}
   =3^{-(8+6.12\sqrt{4})},\\
b_5&\geq8.006\times10^{-10}
   >5.500\ldots\times10^{-15}
   =4^{-(10+6.12\sqrt5)}.
\end{aligned}
\]
For $n\geq6$, \eqref{e:main-bound} is Theorem \ref{t:sharp-4}.
\end{proof}

\begin{proposition}[optimality of the leading exponent]
\label{p:sharpness-exponent}
Let $B_n$ denote the supremum of the lower bounds on $b_n$ furnished
by Theorem \ref{t:bb} over all admissible choices of its parameters.
Then
\[
B_n\leq
\left(\frac{1}{n-1}\right)^{2n-o(n)}
\qquad\text{as }n\to\infty.
\]
Thus, the coefficient $2$ in the leading term of the exponent in
Theorem \ref{t:big-dim} cannot be improved by a different choice of
parameters in Theorem \ref{t:bb}.
\end{proposition}

\begin{proof}
Since $\lambda\rho/(1+\lambda)<\rho$, it suffices to bound every
admissible $\rho$ from above. We first show that $W<V$ forces
$L\gtrsim n^{3/2}$. Put $x=4L/(3M)-1\geq1$. Since $\eta<2/n$,
the inequality $W<V$ and Bernoulli's inequality give
\[
1+2\frac{\omega_n}{\omega_{n-1}}x
>
1+\frac{n\omega_n}{\omega_{n-1}}x\eta
>
\left(1-\frac2M\right)^{-n}
\geq1+\frac{2n}{M}.
\]
Thus, $Mx>n\omega_{n-1}/\omega_n$, and hence
$L=\frac34M(1+x)>\frac34n\omega_{n-1}/\omega_n\gtrsim n^{3/2}$,
where the last estimate follows from Lemma \ref{ball-volume}.

Let $\Sigma(\rho)$ denote the left hand side of
\eqref{rho-condition}. At $\rho=0$,
\[
1-\Sigma(0)
=
\frac{\omega_n}{(4L)^n}
\left[\left(1-\frac{2\sqrt n}{h}\right)^n-a\right]
<
\frac{\omega_n}{(4L)^n}.
\]
The second term of $\Sigma(\rho)$ is $c_2L^\rho$, where
\[
c_2=1-\left(1-\frac1L\right)^n-\frac{2^{n-1}}{L^n}
=\frac nL(1+o(1)),
\]
uniformly over admissible parameters, since $L\gtrsim n^{3/2}$.
All five terms of $\Sigma$ are nondecreasing in $\rho$, and
$L^\rho-1\geq\rho\log L$. Therefore, \eqref{rho-condition} implies
\[
\rho\,\frac nL(1+o(1))\log L
\leq\Sigma(\rho)-\Sigma(0)
\leq1-\Sigma(0)
<
\frac{\omega_n}{(4L)^n}.
\]
Consequently,
\[
\rho
\leq
\frac{L}{n\log L}(1+o(1))
\frac{\omega_n}{(4L)^n}.
\]
Taking $n$th roots and using
$\omega_n^{1/n}=n^{-1/2+o(1)}$ and $L\gtrsim n^{3/2}$ gives
\[
\rho^{1/n}
\leq
\frac{\omega_n^{1/n}}4\,
L^{-1+1/n}(n\log L)^{-1/n}(1+o(1))
\leq n^{-2+o(1)}.
\]
Thus, every lower bound furnished by Theorem \ref{t:bb} is at most
$n^{-2n+o(n)}=(n-1)^{-2n+o(n)}$, which proves the proposition.
\end{proof}

\subsection*{Acknowledgement of AI use} The mathematics in this paper is the product of more than two years of research, and the principal ideas predate the authors' use of AI tools. Claude Opus 5, Claude Fable 5.1, and GPT 5.6 Sol were used as aids for symbolic and numerical computation, parameter search, figure generation, and manuscript preparation. All statements, proofs, and computations are the authors' own and have been verified by the authors.

\raggedbottom
\newcommand{\etalchar}[1]{$^{#1}$}
\providecommand{\bysame}{\leavevmode\hbox to3em{\hrulefill}\thinspace}
\providecommand{\MR}{\relax\ifhmode\unskip\space\fi MR }
\providecommand{\MRhref}[2]{%
  \href{http://www.ams.org/mathscinet-getitem?mr=#1}{#2}
}
\providecommand{\href}[2]{#2}

\end{document}